\documentclass[a4paper,reqno,10pt]{amsart}

\usepackage[
  colorlinks=true,
  linkcolor=blue,
  citecolor=blue,
  urlcolor=blue
]{hyperref}
\usepackage{a4wide}
\usepackage{amssymb,amstext,amsmath,amscd,amsthm,amsfonts}
\usepackage{mathrsfs}
\usepackage{enumitem}
\setlist[enumerate,1]{label={\upshape(\arabic*)}}
\setlist[enumerate,2]{label={\upshape(\alph*)}}
\usepackage{tikz}
\usepackage{tikz-cd}
\usetikzlibrary{arrows,matrix,positioning,calc}
\usepackage{array}
\usepackage{booktabs}
\newcolumntype{C}{>{$}c<{$}}

\newtheorem{theorem}{Theorem}[section]
\newtheorem{theoremi}{Theorem}

\newtheorem{corollary}[theorem]{Corollary}
\newtheorem{lemma}[theorem]{Lemma}
\newtheorem{proposition}[theorem]{Proposition}
\newtheorem{conjecture}[theorem]{Conjecture}

\theoremstyle{definition}
\newtheorem{definition}[theorem]{Definition}
\newtheorem{construction}[theorem]{Construction}
\newtheorem{remark}[theorem]{Remark}
\newtheorem{example}[theorem]{Example}
\newtheorem*{conv}{Conventions and notation}
\newtheorem*{org}{Organization}
\newtheorem*{formalization}{Formalization and verification}
\newtheorem*{acknowledgments}{Acknowledgments}
\numberwithin{equation}{section}

\newcommand{\kk}{k}
\newcommand{\NN}{\mathbb{N}}
\newcommand{\ZZ}{\mathbb{Z}}
\newcommand{\DDual}{\mathbb{D}}
\renewcommand{\AA}{\mathcal{A}}
\newcommand{\CC}{\mathcal{C}}
\newcommand{\DD}{\mathcal{D}}
\newcommand{\LL}{\mathcal{L}}
\newcommand{\PP}{\mathcal{P}}
\newcommand{\QQ}{\mathcal{Q}}
\newcommand{\UU}{\mathcal{U}}
\newcommand{\VV}{\mathcal{V}}

\newcommand{\XX}{\mathcal{X}}
\newcommand{\quot}{\mathrm{quot}}
\newcommand{\sub}{\mathrm{sub}}
\newcommand{\LQ}{\mathcal{L}_{\quot}}
\newcommand{\LS}{\mathcal{L}_{\sub}}
\newcommand{\RQ}{R_{\quot}}
\newcommand{\RS}{R_{\sub}}
\newcommand{\Hom}{\operatorname{Hom}\nolimits}
\newcommand{\End}{\operatorname{End}\nolimits}

\newcommand{\rad}{\operatorname{rad}\nolimits}
\newcommand{\soc}{\operatorname{soc}\nolimits}
\renewcommand{\top}{\operatorname{top}\nolimits}
\newcommand{\ann}{\operatorname{ann}\nolimits}
\newcommand{\tr}{\operatorname{tr}\nolimits}
\newcommand{\ex}{\operatorname{ex}\nolimits}
\newcommand{\Image}{\operatorname{Im}\nolimits}
\newcommand{\op}{\mathrm{op}}

\DeclareMathOperator{\moduleCategory}{\mathsf{mod}}
\renewcommand{\mod}{\moduleCategory}
\DeclareMathOperator{\ind}{\mathsf{ind}}
\DeclareMathOperator{\proj}{\mathsf{proj}}
\DeclareMathOperator{\add}{\mathsf{add}}
\DeclareMathOperator{\Gen}{\mathsf{Gen}}
\DeclareMathOperator{\Cogen}{\mathsf{Cogen}}
\DeclareMathOperator{\qcl}{\mathsf{qcl}}
\DeclareMathOperator{\scl}{\mathsf{scl}}
\allowdisplaybreaks

\title[An equidistribution conjecture for closed subcategories]
{An equidistribution conjecture for quotient-closed\\
and submodule-closed subcategories}
\author[H.\ Enomoto]{Haruhisa Enomoto}
\address{Parakeet Inc., Japan}
\email{haruhisa.enomoto.math@gmail.com}
\date{\today}

\begin{document}

\begin{abstract}
We study subcategories of the module category of a finite-dimensional algebra
that are closed under quotients or submodules. We propose the
quotient--submodule equidistribution conjecture: over a representation-finite
algebra, the number of quotient-closed subcategories of size $i$ is equal to
that of submodule-closed subcategories of size $i$ for every $i$, where size
is the number of indecomposable modules in the subcategory. We prove the
following cases of the conjecture: (1) the five smallest and five largest
values of $i$, hence all algebras with at most nine indecomposable modules;
(2) Nakayama algebras; (3) algebras with radical square zero; and
(4) representation-directed algebras. In the last case, we classify
quotient-closed subcategories by a lower Bruhat interval in a Coxeter group
constructed from the Auslander--Reiten quiver, extending the classification
of Oppermann--Reiten--Thomas for Dynkin quivers. We also prove that quotient
closure defines a finitary convex geometry and classify the functorially
finite quotient-closed subcategories by Gen-minimal modules.
\end{abstract}

\maketitle
\begingroup
\footnotesize
\setcounter{tocdepth}{2}
\tableofcontents
\endgroup

\section{Introduction}\label{sec:introduction}

Torsion classes and torsion-free classes are basic objects in the
representation theory of finite-dimensional algebras. Taking
\(\Hom\)-orthogonals gives a bijection between them
\cite[Section~VI.1]{AssemSimsonSkowronski}. Dropping extension closure from
their definitions leads to quotient-closed and submodule-closed
subcategories, for which orthogonality no longer gives a bijection. This
raises two questions: how the two families are related numerically, and what
structure each family carries. We propose a conjectural answer to the first
question, prove it in several cases, and develop a structure theory for the
second.

Let \(\kk\) be a field and \(A\) a finite-dimensional \(\kk\)-algebra.
Choose a set \(\ind A\) of representatives of the indecomposable modules,
and for a subcategory \(\CC\) let \(\ind\CC\) be the subset of \(\ind A\)
consisting of the indecomposable modules in \(\CC\).

\begin{definition}\label{def:quot-sub-counting-data}
An additive subcategory of \(\mod A\) is \emph{quotient-closed} if it is
closed under factor modules, and it is \emph{submodule-closed} if it is
closed under submodules.  Quotient-closed subcategories are also called
\emph{pretorsion classes} \cite{CFY}. We write \(\LQ(A)\) and \(\LS(A)\) for
the lattices of quotient-closed and submodule-closed subcategories of
\(\mod A\), respectively, both ordered by inclusion.

The \emph{size} of a subcategory \(\CC\) is \(|\ind\CC|\).  For every
integer \(i\geq0\), write \(\LQ^i(A)\) and \(\LS^i(A)\) for the
quotient-closed and submodule-closed subcategories of size \(i\),
respectively.  If \(A\) is representation-finite, define the \emph{size generating polynomials} by
\[
 \RQ(A;q):=\sum_{i=0}^{|\ind A|}|\LQ^i(A)|q^i,\qquad
 \RS(A;q):=\sum_{i=0}^{|\ind A|}|\LS^i(A)|q^i.
\]
\end{definition}

Motivated by extensive computer experiments, we propose the following
conjecture.

\begin{conjecture}[Quotient--Submodule Equidistribution Conjecture]
\label{conj:equidistribution}
Let \(A\) be a representation-finite finite-dimensional algebra over a
field.  Then
\[
 \RQ(A;q)=\RS(A;q).
\]
Equivalently, \(|\LQ^i(A)|=|\LS^i(A)|\) for every \(0\leq i\leq |\ind A|\).
\end{conjecture}

If \(A^{\op}\) is Morita equivalent to \(A\), duality gives a
size-preserving lattice isomorphism between \(\LQ(A)\) and \(\LS(A)\), so
the conjecture holds. However, such a lattice isomorphism need not exist in
general; see Example~\ref{ex:inward-a3}.

The conjecture holds for path algebras of Dynkin quivers over an
algebraically closed field. To see this, let \(Q\) be a Dynkin quiver and
let \(W_Q\) be its Weyl group. Proposition~\ref{prop:ort-dynkin}, due to
Oppermann--Reiten--Thomas, gives a bijection between \(W_Q\) and the
quotient-closed subcategories of \(\mod\kk Q\), under which \(w\)
corresponds to a subcategory omitting exactly \(\ell(w)\) indecomposable
modules. Applying Proposition~\ref{prop:ort-dynkin} to \(Q\) and its
opposite, together with vector-space duality, proves this assertion
(Theorem~\ref{thm:dynkin}).

\begin{remark}\label{rem:computer-verification}
Over \(\kk=\mathbb{F}_2\), computer calculations using QPA verify the
conjecture for
all cluster-tilted algebras of Dynkin type with at most seven simple
modules, all Brauer tree algebras with at most five edges and
exceptional multiplicity at most three, and all representation-finite
gentle algebras with at most four simple modules \cite{QPA}.
The largest verified algebra is the gentle algebra \(A=\kk Q/I\), where
\[
\begin{tikzcd}[column sep=large]
1
& 2 \arrow[l, "\varepsilon"']
    \arrow[r, shift left=.6ex, "\zeta"]
& 3 \arrow[l, shift left=.6ex, "\beta"]
    \arrow[r, shift left=.6ex, "\alpha"]
& 4 \arrow[l, shift left=.6ex, "\gamma"]
    \arrow[loop right, "\delta"]
\end{tikzcd}
\qquad
I=(\alpha\gamma,\ \gamma\alpha,\ \beta\zeta,\ \zeta\beta,\ \delta^2).
\]
It has \(140\) indecomposable modules and \(6{,}395{,}330{,}502\) quotient-closed subcategories, with \(A\not\cong A^{\op}\).  The
\(141\) coefficients of \(\RQ(A;q)\) and \(\RS(A;q)\) agree.
Each polynomial has largest coefficient \(175{,}352{,}963\) at \(q^{64}\).
\end{remark}

The following is the main result of this paper.

\begin{theoremi}[{= Theorems~\ref{thm:size-four-intro},
\ref{thm:nakayama}, \ref{thm:radical-square-zero}, and
\ref{thm:directed-main}(3)}]\label{thm:main}
Let \(A\) be a representation-finite finite-dimensional algebra over an
algebraically closed field.  The quotient--submodule equidistribution
conjecture holds in each of the following cases.
\begin{enumerate}
\item \(A\) has at most nine indecomposable modules.
\item \(A\) is a Nakayama algebra.
\item The radical of \(A\) has square zero.
\item \(A\) is representation-directed.
\end{enumerate}
\end{theoremi}

The first part of Theorem~\ref{thm:main} follows from a more uniform
result.

\begin{theoremi}[{= Lemma~\ref{lem:boundary-coefficients} and
Propositions~\ref{prop:bottom-two}, \ref{prop:bottom-three},
\ref{prop:bottom-four}, \ref{prop:top-four}}]\label{thm:size-four-intro}
Let \(A\) be a representation-finite finite-dimensional algebra over an
algebraically closed field, and put \(N=|\ind A|\).  Then
\[
 |\LQ^i(A)|=|\LS^i(A)|
 \quad\text{and}\quad
 |\LQ^{N-i}(A)|=|\LS^{N-i}(A)|
\]
for every \(0\leq i\leq\min\{4,N\}\).
\end{theoremi}

If \(N\leq9\), every \(0\leq j\leq N\) satisfies \(j\leq4\) or
\(N-j\leq4\).  Thus Theorem~\ref{thm:size-four-intro} implies
Theorem~\ref{thm:main}(1).

For the classes in Theorem~\ref{thm:main}(2)--(4), we compute the
size generating polynomial in closed form.  For a Nakayama algebra, both \(\LQ(A)\)
and \(\LS(A)\) are products of chains whose multisets of lengths agree,
and the polynomial is a product of \(q\)-integers (Theorem~\ref{thm:nakayama}).  For an algebra with radical
square zero and \(r\) simple modules, the two polynomials \(\RQ(A;q)\) and
\(\RS(A;q)\) are both equal to the Poincar\'e polynomial of the Weyl group of
the separated quiver divided by \((1+q)^r\)
(Theorem~\ref{thm:radical-square-zero}).

For a representation-directed algebra, we classify quotient-closed
subcategories via a lower Bruhat interval of a Coxeter group, thereby showing
that the size generating polynomial is the Poincar\'e polynomial of the lower
Bruhat interval with degrees reversed.
First, let \(\Delta_A\) be the simple graph whose vertices are the
\(\tau\)-orbits in the Auslander--Reiten quiver of \(\mod A\). Two distinct
\(\tau\)-orbits are joined by an edge if an arrow in the
Auslander--Reiten quiver has one endpoint in each orbit.
Let \(W(\Delta_A)\) be its simply-laced Coxeter group.

Next, we construct the element that determines the Bruhat interval. List the
indecomposable modules as \(X_1,\ldots,X_N\) in a Hom-compatible order, so
that
\[
 i\ne j,\quad\Hom_A(X_i,X_j)\ne0\quad\Longrightarrow\quad i<j,
\]
and let \(i_d\) be the vertex corresponding to the \(\tau\)-orbit of
\(X_d\). Consider the element
\[
 w_A:=s_{i_1}\cdots s_{i_N}\in W(\Delta_A).
\]
We call this expression the \emph{Auslander--Reiten word} associated with the
chosen order. We prove that the expression is reduced; hence
\(\ell(w_A)=N\). We prove
\[
 \RQ(A;q)=\RS(A;q)=\sum_{u\leq w_A}q^{N-\ell(u)},
\]
where \(\leq\) is the Bruhat order (Theorem~\ref{thm:directed-main}(3)).
Unlike the Weyl groups arising from Dynkin quivers, \(W(\Delta_A)\) can be
infinite; nevertheless, the lower Bruhat interval \([1,w_A]\) is finite
(Example~\ref{ex:directed-affine-d4}).
Moreover, we explicitly classify quotient-closed subcategories of \(\mod A\)
as follows. For each \(u\leq w_A\), delete the indecomposable modules indexed
by the positions of the lexicographically first reduced subword of the
Auslander--Reiten word representing \(u\). This gives every quotient-closed
subcategory exactly once (Theorem~\ref{thm:directed-main}(1)).
This extends the classification of Oppermann--Reiten--Thomas for Dynkin
quivers to representation-directed algebras
(Remark~\ref{rem:directed-recovers-ort}).

We now turn to the structure theory of quotient-closed subcategories.  For \(B\subseteq\ind A\),
let \(\add B\) be the subcategory formed by direct summands of finite
direct sums of modules in \(B\), and let \(\Gen(\add B)\) be the
subcategory of factor modules of finite direct sums of modules in
\(B\).  Put
\[
 \qcl(B)=\ind\Gen(\add B),
\]
the set of indecomposable modules in the smallest quotient-closed
subcategory containing \(B\).  We prove that $\qcl$
satisfies the anti-exchange property: if \(\QQ=\qcl(\QQ)\) and
\(X,Y\in\ind A\setminus\QQ\) are nonisomorphic, then
\[
 X\in\qcl(\QQ\cup\{Y\})
 \quad\Longrightarrow\quad
 Y\notin\qcl(\QQ\cup\{X\})
\]
(Proposition~\ref{prop:anti-exchange}).  Consequently, \(\ind A\) with
\(\qcl\) is a finitary convex geometry.  Therefore,
\(\LQ(A)\) and \(\LS(A)\) are the lattices of closed
sets of two convex geometries on \(\ind A\)
(Corollary~\ref{cor:two-convex-geometries}). The convex-geometric
interpretation has several useful consequences. For example, when \(A\) is
representation-finite, both \(\LQ(A)\) and \(\LS(A)\) are graded by size
(Corollary~\ref{cor:finite-convex-consequences}).

The convex-geometric viewpoint also gives a short route to the classification
of functorially finite quotient-closed subcategories. A basic fact about
convex geometries is that every finitely generated closed set has a unique
minimal generating set (Proposition~\ref{prop:compact-extreme-bijection}).
For quotient closure, this set consists precisely of the indecomposable
splitting projectives in the subcategory
(Proposition~\ref{prop:as-minimal-cover}(3)): an indecomposable module
\(X\in\CC\) is \emph{splitting projective} if every surjection
\(C\twoheadrightarrow X\) with \(C\in\CC\) splits.

For a module \(M\), write \(\Gen M=\Gen(\add M)\). By the classical result
of Auslander--Smal{\o}, a quotient-closed subcategory \(\CC\) is
functorially finite exactly when \(\CC=\Gen M\) for some module \(M\)
(Proposition~\ref{prop:as-approximation}). The direct sum of the unique
minimal generators is therefore a basic \(\Gen\)-minimal module. Here a
basic module \(M\) is \emph{\(\Gen\)-minimal} if \(X\notin\Gen M'\) for every
decomposition \(M=X\oplus M'\) with \(X\) indecomposable
(Definition~\ref{def:gen-minimal}). We obtain the following classification.
\begin{theoremi}[{= Proposition~\ref{prop:compact-functorially-finite}(2)}]
\label{thm:gen-minimal-classification-intro}
Let \(A\) be a finite-dimensional algebra over a field.  The map
\[
 M\longmapsto\Gen M
\]
is a bijection from the isomorphism classes of basic \(\Gen\)-minimal
modules to the
functorially finite quotient-closed subcategories of \(\mod A\).  The
inverse map sends a subcategory \(\CC\) to the direct sum of the
indecomposable splitting projective modules in $\CC$.
\end{theoremi}

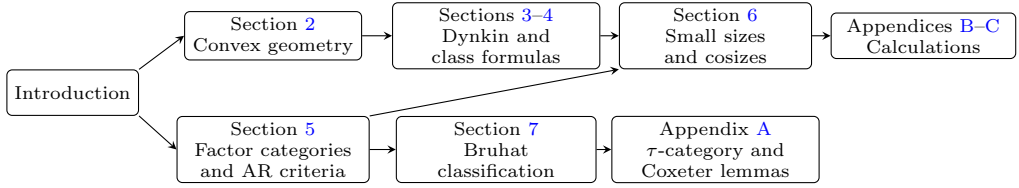
\begin{figure}
\centering
\begin{tikzpicture}[
  >=stealth,
  roadmap/.style={draw, rounded corners=2pt, align=center,
    font=\scriptsize, inner sep=2pt, minimum height=6mm}
]
\node[roadmap, text width=16mm] (intro) at (0,0)
  {Introduction};
\node[roadmap, text width=22mm] (convex) at (2.65,0.75)
  {Section~\ref{sec:convex}\\Convex geometry};
\node[roadmap, text width=24mm] (categorical) at (2.65,-0.75)
  {Section~\ref{sec:categorical}\\Factor categories and AR criteria};
\node[roadmap, text width=26mm] (classes) at (5.6,0.75)
  {Sections~\ref{sec:conjecture}--\ref{sec:classes}\\Dynkin and class formulas};
\node[roadmap, text width=25mm] (directed) at (5.6,-0.75)
  {Section~\ref{sec:representation-directed}\\Bruhat\\classification};
\node[roadmap, text width=24mm] (coefficients) at (8.5,0.75)
  {Section~\ref{sec:coefficients}\\Small sizes and cosizes};
\node[roadmap, text width=26mm] (appa) at (8.5,-0.75)
  {Appendix~\ref{sec:directed-technical}\\\(\tau\)-category and\\Coxeter lemmas};
\node[roadmap, text width=23mm] (smallapps) at (11.25,0.75)
  {Appendices~\ref{app:small-core}--\ref{sec:four-ladders}\\Calculations};

\draw[->] (intro.north east) -- (convex.west);
\draw[->] (intro.south east) -- (categorical.west);
\draw[->] (convex) -- (classes);
\draw[->] (classes) -- (coefficients);
\draw[->] (categorical) -- (directed);
\draw[->] (directed) -- (appa);
\draw[->] (categorical.north east) -- (coefficients.south west);
\draw[->] (coefficients) -- (smallapps);
\end{tikzpicture}
\caption{Suggested reading paths.}
\end{figure}

\begin{org}
Section~\ref{sec:convex} develops the finitary convex geometry, applies
it to quotient closure, and proves Theorem~\ref{thm:gen-minimal-classification-intro}.
Section~\ref{sec:conjecture} establishes basic
properties of the conjecture, proves the Dynkin case, and counts the
quotient-closed subcategories omitting one or two indecomposable modules.
Section~\ref{sec:classes} computes the size generating polynomials for
Nakayama algebras and algebras with radical square zero, thereby proving the
equidistribution conjecture for both classes.
Section~\ref{sec:categorical} uses Iyama's theory of \(\tau\)-categories to
study the ideal quotient \((\mod A)/[\CC]\) and gives a criterion for a
subcategory to be quotient-closed using only the combinatorics of the
Auslander--Reiten quiver.
Section~\ref{sec:coefficients} proves the equality for \(i\leq4\) and
\(i\geq |\ind A|-4\); the supporting calculations are given in
Appendices~\ref{app:small-core} and~\ref{sec:four-ladders}.
Section~\ref{sec:representation-directed} classifies quotient-closed
subcategories for representation-directed algebras by lower Bruhat intervals,
thereby proving the conjecture for this class. Technical lemmas are proved in
Appendix~\ref{sec:directed-technical}.
\end{org}

\phantomsection\label{note:formalization-and-verification}
\begin{formalization}
A Lean~4 formalization of results in this paper and the verification scripts
for the computer-assisted arguments are available at
\url{https://github.com/haruhisa-enomoto/quotient-submodule-equidistribution}.
\end{formalization}

\begin{acknowledgments}
The author thanks Ryu Tomonaga for organizing the workshop ``Young Researchers
Seminar in Ring Theory and Representation Theory'' in March 2026 at Nagoya
University and for his lecture ``Advances in the Representation Theory of
Algebras,'' which strongly renewed the author's motivation to pursue
mathematics and led to the observation that quotient closure defines a convex
geometry.  The author also thanks OpenAI and Anthropic for developing the AI
models used in this project.  See
\hyperref[sec:human-ai-collaboration]{the final note on human--AI
collaboration} for an account of how the author and the AI systems collaborated
to produce this paper.
Finally, I thank Mikazuki Hanai, a
character in the Japanese visual novel
\href{https://www.entergram.co.jp/musicus/index.html}{\emph{MUSICUS!}}, and its writer,
Renya Setoguchi. Mikazuki confronted me with the existential question of why
human beings exist at all, compelling me to reflect upon it; she also led me
to consider what it is that I, and I alone, can do in this accursed world.
\end{acknowledgments}

\begin{conv}
Unless otherwise stated, \(\kk\) is an algebraically closed field and
\(A\) is a finite-dimensional \(\kk\)-algebra.  Modules are finitely
generated right modules.  For a positive integer \(n\), put
\([n]:=\{1,\ldots,n\}\).  Let \(\DDual=\Hom_\kk(-,\kk)\) denote the usual
duality between \(\mod A\) and \(\mod A^{\op}\).  All subcategories are
full, additive, and closed under isomorphisms and direct summands.  We
identify such a subcategory with the set of isomorphism classes of
indecomposable modules which it contains.  Paths are composed from left to
right: \(\alpha\beta\) means first \(\alpha\), then \(\beta\).  By the
Auslander--Reiten quiver we mean the usual translation quiver of \(\mod A\):
its vertices are the isomorphism classes of indecomposable modules, the
number of arrows \(X\to Y\) is the dimension of the space of irreducible
maps from \(X\) to \(Y\), and it is equipped with the Auslander--Reiten
translation \(\tau\).  We denote the Auslander--Reiten quiver by
\(\Gamma(\mod A)\).
\end{conv}

\section{Quotient closure as a finitary convex geometry}\label{sec:convex}

In this section we introduce finitary convex geometries, show that
quotient closure defines one on the set of indecomposable modules, and
interpret its extreme points as the splitting projectives of Auslander and
Smal{\o}. We then use this convex-geometric description to classify
functorially finite quotient-closed subcategories.

\subsection{Lattice structure and extreme points of finitary convex geometries}

\begin{definition}\label{def:finitary-convex-geometry}
A \emph{closure operator} on a set \(E\) is a map
\(c\colon2^E\to2^E\) such that
\[
 X\subseteq c(X),\qquad
 X\subseteq Y\Longrightarrow c(X)\subseteq c(Y),\qquad
 c(c(X))=c(X).
\]
A subset \(C\subseteq E\) is \emph{closed} if \(c(C)=C\).  The closure
operator is \emph{finitary} if
\[
 c(X)=\bigcup_{\substack{Y\subseteq X\\Y\text{ finite}}}c(Y)
\]
for every \(X\subseteq E\).  It satisfies \emph{anti-exchange} if, for every
closed set \(C\) and distinct \(x,y\notin C\),
\[
 x\in c(C\cup\{y\})\quad\Longrightarrow\quad
 y\notin c(C\cup\{x\}).
\]
A \emph{finitary convex geometry} is a set \(E\) equipped with a finitary
closure operator \(c\) such that \(c(\varnothing)=\varnothing\) and
anti-exchange holds.
\end{definition}

When \(E\) is finite, Definition~\ref{def:finitary-convex-geometry} is the
usual definition of a convex geometry \cite{EdelmanJamison}.

\begin{definition}\label{def:independent-extreme}
A subset \(B\subseteq E\) is \emph{\(c\)-independent} if
\[
 b\notin c(B\setminus\{b\})
 \qquad\text{for every }b\in B.
\]
For \(S\subseteq C\), we say that \(S\) \emph{generates} the closed set
\(C\) if \(c(S)=C\).
For a closed set \(C\), put
\[
 \ex(C)=\{\,x\in C\mid x\notin c(C\setminus\{x\})\,\},
\]
the set of \emph{extreme points} of \(C\).
\end{definition}

For a compact closed set, \(\ex(C)\) will turn out to be the unique minimal
generating set (Proposition~\ref{prop:compact-extreme-bijection}).

The closed sets of a finitary convex geometry form a lattice which is
algebraic and spatial; we recall these terms.

\begin{definition}\label{def:algebraic-spatial-lattice}
An element \(x\) of a complete lattice is \emph{compact} if
\[
 x\leq\bigvee_{\lambda\in\Lambda}x_\lambda
 \quad\Longrightarrow\quad
 x\leq\bigvee_{\lambda\in\Lambda'}x_\lambda
\]
for some finite \(\Lambda'\subseteq\Lambda\).  A complete lattice is
\emph{algebraic} if every element is a join of compact elements.  It is
\emph{spatial} if every element is a join of completely join-irreducible
elements, where \(j\) is \emph{completely join-irreducible} if
\[
 j=\bigvee_{\lambda\in\Lambda}x_\lambda
 \quad\Longrightarrow\quad
 j=x_\lambda\text{ for some }\lambda.
\]
\end{definition}

We first recall the lattice structure of a finitary convex geometry.

\begin{theorem}[{\cite[Section~2 and Corollary~3.2]{AdarichevaPouzet},
\cite[Lemma~5-3.1 and Corollaries~5-3.2 and~5-3.5]
{AdarichevaNationConvexGeometries}}]
\label{thm:infinite-convex-structure}
Let \((E,c)\) be a finitary convex geometry.
\begin{enumerate}
\item The closed subsets form a complete lattice, with
\[
 \bigwedge_\lambda C_\lambda=\bigcap_\lambda C_\lambda,
 \qquad
 \bigvee_\lambda C_\lambda=c\left(\bigcup_\lambda C_\lambda\right).
\]
This lattice is algebraic.  Its compact elements are precisely the sets
\(c(F)\) with \(F\subseteq E\) finite.
\item The lattice is spatial.  Its completely join-irreducible elements are
precisely the point closures
\[
 J_x=c(\{x\})\qquad(x\in E).
\]
The point closures are compact and pairwise distinct, and for each
\(x\) the set \(J_x\setminus\{x\}\) is the greatest proper closed subset
of \(J_x\).
\end{enumerate}
\end{theorem}

\begin{proof}
(1)
The formulas for meets and joins follow from the definition of a closure
operator.  It is shown in \cite[Section~2]{AdarichevaPouzet} that this lattice
is algebraic and that its compact elements are the sets \(c(F)\) with \(F\)
finite.

(2)
It is shown in
\cite[Lemma~5-3.1 and Corollaries~5-3.2 and~5-3.5]{AdarichevaNationConvexGeometries}
that \(J_x\setminus\{x\}\) is closed,
that \(J_x\) is completely join-irreducible, and that the lattice is spatial.
Part~(1) makes every \(J_x\) compact, and anti-exchange with the closed set
\(\varnothing\) makes the point closures pairwise distinct.  Since every
closed set \(C\) is the join of the \(J_x\) with \(x\in C\), every completely
join-irreducible closed set is one of the \(J_x\).  Finally, any proper closed
subset of \(J_x\) omits \(x\), so it is contained in the closed set
\(J_x\setminus\{x\}\).
\end{proof}

The compact elements have the following intrinsic minimal generators.

\begin{proposition}\label{prop:compact-extreme-bijection}
Let \((E,c)\) be a finitary convex geometry.  The maps
\[
 \left\{\begin{array}{c}
 \text{finite \(c\)-independent}\\[-2pt]
 \text{subsets of \(E\)}
 \end{array}\right\}
 \ \underset{C\mapsto\ex(C)}{\overset{B\mapsto c(B)}{\rightleftarrows}}
 \left\{\begin{array}{c}
 \text{compact closed}\\[-2pt]
 \text{subsets of \(E\)}
 \end{array}\right\}
\]
are inverse bijections.
\end{proposition}

\begin{proof}
Let \(C\) be compact and closed.  By
Theorem~\ref{thm:infinite-convex-structure}(1), there is a finite set
\(S\subseteq C\) with \(c(S)=C\); choose a subset \(B\subseteq S\) minimal
with this property.  Every generating set \(S'\) of \(C\) contains
\(\ex(C)\): for \(x\in\ex(C)\setminus S'\) we would get
\(C=c(S')\subseteq c(C\setminus\{x\})\), which does not contain \(x\).
In particular \(\ex(C)\subseteq B\).  Suppose that \(b\in B\setminus\ex(C)\).  Since \(c\)
is finitary, there are \(y_1,\ldots,y_m\in C\setminus\{b\}\) such that
\[
 b\in c(\{y_1,\ldots,y_m\}).
\]
Set \(K_0=c(B\setminus\{b\})\) and
\(K_i=c(K_{i-1}\cup\{y_i\})\).  Minimality of \(B\) gives
\(b\notin K_0\).  Each \(K_{i-1}\) is a subset of \(C\) containing
\(B\setminus\{b\}\), so
\[
 c(K_{i-1}\cup\{b\})=c(B)=C.
\]
Inductively, suppose that \(b\notin K_{i-1}\).  If \(y_i\notin K_{i-1}\),
then \(y_i\in C=c(K_{i-1}\cup\{b\})\), so anti-exchange gives
\(b\notin K_i\).  If \(y_i\in K_{i-1}\), then
\(K_i=K_{i-1}\), so again \(b\notin K_i\).  This contradicts
\(b\in c(\{y_1,\ldots,y_m\})\subseteq K_m\).  Consequently
\(B=\ex(C)\), and \(\ex(C)\) is the unique minimal generating set of \(C\).

A finite set \(B\) is \(c\)-independent precisely when it is a minimal
generating set of \(c(B)\).  Since \(c(B)\) is compact, the uniqueness of
its minimal generating set gives
\[
 \ex(c(B))=B.
\]
For every compact closed set \(C\), the set \(\ex(C)\) generates \(C\), so
\[
 c(\ex(C))=C.
\]
This proves the bijection.
\end{proof}

For later use, we record the finite deletion, grading, and generation
properties.

\begin{corollary}[{\cite[Theorem~2.1]{EdelmanJamison}}]
\label{cor:finite-abstract-convex}
Let \((E,c)\) be a finite convex geometry and let \(C\) be closed.
\begin{enumerate}
\item If \(C\neq\varnothing\), then \(\ex(C)\neq\varnothing\), and
\(C\setminus\{x\}\) is closed for every \(x\in\ex(C)\).
\item If \(C\subsetneq D\) is a covering relation between closed sets, then
\(D=C\cup\{x\}\) for a unique \(x\in E\).  Thus the lattice of closed sets
is graded by cardinality.
\item One has \(C=c(\ex(C))\).
\end{enumerate}
\end{corollary}

Corollary~\ref{cor:finite-abstract-convex}(1) extends to finite intervals in
a finitary convex geometry, even when the ground set is infinite.

\begin{corollary}\label{cor:finite-interval-chain}
Let \((E,c)\) be a finitary convex geometry, and let \(C\subseteq D\) be
closed sets such that \(D\setminus C\) is finite.  There is an ordering
\(x_1,\ldots,x_n\) of \(D\setminus C\) such that
\[
 D_i=D\setminus\{x_1,\ldots,x_i\}
\]
is closed for every \(0\leq i\leq n\).  Equivalently, one can pass from
\(D\) to \(C\) by successively deleting extreme points.
\end{corollary}

\begin{proof}
Put \(F=D\setminus C\).  For \(S\subseteq F\), define
\[
 \overline c(S)=c(C\cup S)\setminus C.
\]
Since \(D\) is closed, \(c(C\cup S)\subseteq D\), so \(\overline c\) maps
subsets of \(F\) to subsets of \(F\).  The axioms for a closure operator and
anti-exchange follow from those of \(c\), while
\(\overline c(\varnothing)=\varnothing\) because \(C\) is closed.  Thus
\((F,\overline c)\) is a finite convex geometry.  Starting with its closed
set \(F\), apply Corollary~\ref{cor:finite-abstract-convex}(1) repeatedly.
The resulting sequence of closed subsets of \(F\) corresponds, after
adjoining \(C\), to the required sequence of closed subsets of \(E\).
\end{proof}

\subsection{The convex geometry of quotient closure}

Let \(A\) be a finite-dimensional algebra over a field.

\begin{definition}\label{def:quotient-closure-operator}
For \(\XX\subseteq\ind A\), set
\[
 \qcl(\XX)=\ind\Gen(\add\XX).
\]
Here \(\Gen(\add\XX)\) consists of the factor modules of finite direct sums
of modules in \(\XX\).
\end{definition}

The direct sums are essential: an indecomposable can
be generated by several modules together without being generated by any one
of them.  Since every such factor module uses only finitely many modules
of \(\XX\), the map \(\qcl\) is a finitary closure operator on
\(\ind A\).  Under the identification of a subcategory with its set of
indecomposable modules, the closed sets of \(\qcl\) are precisely the
members of \(\LQ(A)\); moreover \(\qcl(\varnothing)=\varnothing\), since
the empty set generates only the zero module.

The following test reduces quotient closure to the omitted indecomposable
modules.

\begin{lemma}\label{lem:indecomposable-quotient-test}
Let \(K\subseteq\ind A\), and put \(D=\ind A\setminus K\).  For an
\(A\)-module \(X\), let \(\tr_{\add K}(X)\) be the trace of \(\add K\)
in \(X\).  The following are equivalent.
\begin{enumerate}
\item The subcategory \(\add K\) is quotient-closed.
\item One has \(X\notin\Gen(\add K)\) for every \(X\in D\).
\item One has \(X/\tr_{\add K}(X)\ne0\) for every \(X\in D\).
\end{enumerate}
\end{lemma}

\begin{proof}
(1) $\Leftrightarrow$ (2):
The subcategory \(\add K\) is quotient-closed if and only if
\(\Gen(\add K)=\add K\).  If a module in \(\Gen(\add K)\) has an
indecomposable direct summand \(X\), then the projection onto \(X\) shows
that \(X\in\Gen(\add K)\).  Hence this equality holds if and only if no
member of \(D\) belongs to \(\Gen(\add K)\).

(2) $\Leftrightarrow$ (3):
The trace of \(\add K\) in \(X\) equals \(X\) if and only if finitely many
maps from objects of \(\add K\) give a surjection onto \(X\).
\end{proof}

\begin{definition}[{\cite[Section~2]{AuslanderSmaloPreprojective}}]
\label{def:cover-splitting-projective}
A subset
\(B\subseteq\ind\CC\) is a \emph{cover} of an additive subcategory \(\CC\)
if every object of \(\CC\) is a factor module of an object of \(\add B\).
An indecomposable
\(X\in\CC\) is a \emph{splitting projective} if every surjection
\(C\twoheadrightarrow X\) with \(C\in\CC\) splits; write
\(\mathcal P_0(\CC)\) for the set of these modules.
\end{definition}

\begin{proposition}[{\cite[Proposition~2.1, Theorem~2.3, and
Corollary~2.4]{AuslanderSmaloPreprojective}}]
\label{prop:as-minimal-cover}
Let \(\CC\) be an additive subcategory.
\begin{enumerate}
\item Every cover of \(\CC\) contains \(\mathcal P_0(\CC)\).
\item A cover \(B\) of \(\CC\) is minimal if and only if
\(B=\mathcal P_0(\CC)\).
\item If \(\CC\) has a finite cover, then \(\mathcal P_0(\CC)\) is finite
and is the unique minimal cover of \(\CC\).
\end{enumerate}
\end{proposition}

We also use the following elimination of repeated copies of an indecomposable
target.

\begin{lemma}[{\cite[Proposition~2.2(a)]{AuslanderSmaloPreprojective}}]
\label{lem:as-remove-target}
Let \(X\) be indecomposable and let
\(M=X^n\oplus M'\), where \(X\) is not a direct summand of \(M'\).  If there
is a nonsplit surjection \(M\twoheadrightarrow X\), then there is a
surjection \((M')^m\twoheadrightarrow X\) for some \(m\geq1\).
\end{lemma}

We apply Proposition~\ref{prop:as-minimal-cover} to quotient closure.

\begin{proposition}\label{prop:anti-exchange}
The closure operator \(\qcl\) satisfies anti-exchange.
\end{proposition}

\begin{proof}
Let \(\QQ\subseteq\ind A\) be \(\qcl\)-closed, and suppose that
nonisomorphic indecomposable modules \(X,Y\notin\QQ\) satisfy
\[
 X\in\qcl(\QQ\cup\{Y\}),\qquad
 Y\in\qcl(\QQ\cup\{X\}).
\]
Since \(\qcl\) is finitary, there is a finite set \(F\subseteq\QQ\) such
that
\[
 X\in\qcl(F\cup\{Y\}),\qquad
 Y\in\qcl(F\cup\{X\}).
\]
Put
\[
 \mathcal D=\Gen\bigl(\add(F\cup\{X,Y\})\bigr).
\]
The two membership relations give
\[
 \mathcal D
 =\Gen\bigl(\add(F\cup\{X\})\bigr)
 =\Gen\bigl(\add(F\cup\{Y\})\bigr).
\]
Thus \(F\cup\{X\}\) and \(F\cup\{Y\}\) are finite covers of
\(\mathcal D\).  By Proposition~\ref{prop:as-minimal-cover}(1),(3),
\(\mathcal P_0(\mathcal D)\) is a cover contained in both finite covers.
Since \(F\subseteq\QQ\) while
\(X,Y\notin\QQ\) and \(X\not\simeq Y\),
\[
 \mathcal P_0(\mathcal D)
 \subseteq(F\cup\{X\})\cap(F\cup\{Y\})=F.
\]
It follows that
\[
 \mathcal D
 =\Gen\bigl(\add\mathcal P_0(\mathcal D)\bigr)
 \subseteq\Gen(\add F)
 \subseteq\add\QQ,
\]
contrary to \(X\in\mathcal D\) and \(X\notin\QQ\).
\end{proof}

The closure operator defined by submodules is
\[
 \scl(\XX)=\ind\Cogen(\add\XX),
\]
where \(\Cogen(\add\XX)\) consists of the submodules of finite direct
sums of modules in \(\XX\).  The duality \(\DDual\)
exchanges factor modules and submodules between \(\mod A\) and
\(\mod A^{\op}\), so
\[
 \DDual\bigl(\Cogen(\add\XX)\bigr)
 =\Gen\bigl(\add\{\,\DDual X\mid X\in\XX\,\}\bigr)
\]
inside \(\mod A^{\op}\).  Proposition~\ref{prop:anti-exchange} for
\(A^{\op}\) therefore gives anti-exchange for \(\scl\).  Thus
\((\ind A,\qcl)\) and \((\ind A,\scl)\) are finitary convex geometries.

\begin{corollary}\label{cor:two-convex-geometries}
If \(A\) is representation-finite, then \(\LQ(A)\) and \(\LS(A)\) are the
lattices of closed sets of two convex geometries on \(\ind A\).
\end{corollary}

\begin{proof}
The closure operators \(\qcl\) and \(\scl\) are finitary and satisfy
anti-exchange, and \(\qcl(\varnothing)=\scl(\varnothing)=\varnothing\).
Since \(\ind A\) is finite, they define convex geometries.  The closed sets
of \(\qcl\) form \(\LQ(A)\), and the closed sets of \(\scl\) form
\(\LS(A)\).
\end{proof}

Proposition~\ref{prop:split-projective-extreme} translates splitting
projectivity into convex-geometric language and identifies the possible
one-point deletions.

\begin{proposition}\label{prop:split-projective-extreme}
Let \(\CC\in\LQ(A)\) and \(X\in\ind\CC\).  The following are equivalent.
\begin{enumerate}
\item \(X\in\ex(\ind\CC)\) for the closure operator \(\qcl\).
\item \(\ind\CC\setminus\{X\}\) is quotient-closed.
\item \(X\) is a splitting projective in \(\CC\).
\end{enumerate}
\end{proposition}

\begin{proof}
(1) $\Leftrightarrow$ (2):
Since \(\qcl(\ind\CC\setminus\{X\})\subseteq\ind\CC\), both conditions are
equivalent to
\[
 X\notin\Gen\bigl(\add(\ind\CC\setminus\{X\})\bigr).
\]

(2) $\Rightarrow$ (3):
Suppose that (2) holds but \(X\) is not a splitting projective in \(\CC\).
Choose a nonsplit surjection \(M\twoheadrightarrow X\) with \(M\in\CC\).
By Krull--Schmidt,
\[
 M\simeq X^n\oplus M',
 \qquad M'\in\add(\ind\CC\setminus\{X\}).
\]
Lemma~\ref{lem:as-remove-target} gives a surjection
\((M')^m\twoheadrightarrow X\) for some \(m\geq1\).  Hence
\(X\in\Gen(\add(\ind\CC\setminus\{X\}))\), contrary to (2).

(3) $\Rightarrow$ (2):
If (2) does not hold, there is a surjection onto \(X\) from an object of
\(\add(\ind\CC\setminus\{X\})\). This surjection cannot split, since
Krull--Schmidt would otherwise make \(X\) a summand of its source.  Thus
\(X\) is not a splitting projective in \(\CC\), contrary to (3).
\end{proof}

\begin{corollary}\label{cor:finite-quotient-interval-chain}
Let \(\CC\subseteq\DD\) be quotient-closed subcategories such that
\(\ind\DD\setminus\ind\CC\) is finite.  There is an ordering
\(X_1,\ldots,X_n\) of \(\ind\DD\setminus\ind\CC\) such that, on putting
\[
 \DD_0=\DD,
 \qquad
 \DD_i=\add\bigl(\ind\DD\setminus\{X_1,\ldots,X_i\}\bigr)
 \quad(1\leq i\leq n),
\]
each \(\DD_i\) is quotient-closed, \(\DD_n=\CC\), and \(X_i\) is
splitting projective in \(\DD_{i-1}\) for every \(1\leq i\leq n\).
\end{corollary}

\begin{proof}
Apply Corollary~\ref{cor:finite-interval-chain} to the \(\qcl\)-closed sets
\(\ind\CC\subseteq\ind\DD\).  At the \(i\)-th deletion, \(X_i\) is an
extreme point of \(\ind\DD_{i-1}\), so
Proposition~\ref{prop:split-projective-extreme} shows that it is splitting
projective in \(\DD_{i-1}\).
\end{proof}

The abstract structure theorem now has the following module-theoretic form.

\begin{corollary}\label{cor:quotient-convex-structure}
Let \(A\) be a finite-dimensional algebra over a field.
\begin{enumerate}
\item The poset \(\LQ(A)\) is a complete algebraic and spatial lattice, with
\[
 \bigwedge_\lambda\CC_\lambda=\bigcap_\lambda\CC_\lambda,
 \qquad
 \bigvee_\lambda\CC_\lambda
 =\Gen\left(\add\left(\bigcup_\lambda\CC_\lambda\right)\right).
\]
Its completely join-irreducible elements are the subcategories
\(\Gen X\) with \(X\in\ind A\); these are compact and pairwise distinct.
\item The assignments
\[
 B\longmapsto\Gen(\add B),\qquad
 \CC\longmapsto
 \{\,X\in\ind\CC\mid X\text{ is a splitting projective in }\CC\,\}
\]
are inverse bijections between the finite \(\qcl\)-independent subsets of
\(\ind A\) and the compact quotient-closed subcategories.
\end{enumerate}
\end{corollary}

\begin{proof}
(1)
Apply Theorem~\ref{thm:infinite-convex-structure} to \(\qcl\).

(2)
Apply Proposition~\ref{prop:compact-extreme-bijection} to \(\qcl\), and use
Proposition~\ref{prop:split-projective-extreme} to identify
\(\ex(\ind\CC)\) with \(\mathcal P_0(\CC)\).
\end{proof}

Thus, for a compact quotient-closed subcategory, its convex-geometric
extreme set is finite and is the unique minimal cover in the sense of
Auslander and Smal{\o}.

For a representation-finite algebra, the convex-geometric conclusions specialize
as follows.

\begin{corollary}\label{cor:finite-convex-consequences}
Let \(A\) be representation-finite.
\begin{enumerate}
\item Every nonzero quotient-closed subcategory contains a splitting
projective, and deleting any such indecomposable preserves
quotient closure.
\item If \(\CC'\subsetneq\CC\) is a covering relation in \(\LQ(A)\), then
\(\ind\CC\setminus\ind\CC'\) has one element.  Consequently,
\(\LQ(A)\) is graded by size.
\item Every quotient-closed subcategory is generated by its splitting
projectives.
\item For submodule-closed subcategories:
\begin{enumerate}
\item every nonzero submodule-closed subcategory contains a splitting
injective, and deleting any such indecomposable preserves submodule closure;
\item if \(\CC'\subsetneq\CC\) is a covering relation in \(\LS(A)\), then
\(\ind\CC\setminus\ind\CC'\) has one element, and hence \(\LS(A)\) is graded
by size;
\item every submodule-closed subcategory is cogenerated by its splitting
injectives.
\end{enumerate}
\end{enumerate}
\end{corollary}

\begin{proof}
(1)
Apply Corollary~\ref{cor:finite-abstract-convex}(1) and
Proposition~\ref{prop:split-projective-extreme}.

(2)
Apply Corollary~\ref{cor:finite-abstract-convex}(2).

(3)
For \(\CC\in\LQ(A)\), the set \(\ind\CC\) is a finite cover of \(\CC\).
Proposition~\ref{prop:as-minimal-cover}(3) shows that
\(\mathcal P_0(\CC)\) is a cover, so
\(\CC=\Gen(\add\mathcal P_0(\CC))\).

(4)(a)
Apply (1) to \(A^{\op}\), and then apply the duality
\(\DDual\colon\mod A\to\mod A^{\op}\).

(4)(b)
Apply (2) to \(A^{\op}\), and then apply \(\DDual\).

(4)(c)
Apply (3) to \(A^{\op}\), and then apply \(\DDual\).
\end{proof}

In particular, a maximal chain in \(\LQ(A)\), read in decreasing order,
removes one splitting projective at each step and ends at the zero
subcategory.

\subsection{Functorially finite subcategories and
\texorpdfstring{\(\Gen\)}{Gen}-minimal modules}

We next relate compactness to functorial finiteness.
\begin{definition}\label{def:functorially-finite}
A morphism \(C_X\to X\) with \(C_X\in\CC\) is a right
\(\CC\)-approximation if every morphism from an object of \(\CC\) to \(X\)
factors through it.  The subcategory is \emph{contravariantly finite} if
every module has a right \(\CC\)-approximation.  Left approximations and
\emph{covariantly finite} subcategories are defined analogously, and
\(\CC\) is \emph{functorially finite} if it has both properties.
\end{definition}

\begin{definition}[{\cite[Section~VI.6]{AssemSimsonSkowronski}}]
\label{def:gen-minimal}
A basic module \(M\) is \emph{\(\Gen\)-minimal} if
\(X\notin\Gen M'\) for every decomposition \(M=X\oplus M'\) with \(X\)
indecomposable.
Dually, \(M\) is \emph{\(\Cogen\)-minimal} if
\(X\notin\Cogen M'\) for every such decomposition.
\end{definition}

We use the following approximation result.

\begin{proposition}[{\cite[Proposition~4.6(a), (c)]
{AuslanderSmaloPreprojective}}]
\label{prop:as-approximation}
Let \(A\) be a finite-dimensional algebra over a field, and let \(\CC\) be
an additive quotient-closed subcategory of \(\mod A\).  Then \(\CC\) is
contravariantly finite.  Moreover, it is functorially finite if and only if
\(\CC=\Gen M\) for some \(M\in\CC\).
\end{proposition}

Combining Propositions~\ref{prop:as-minimal-cover} and
\ref{prop:as-approximation} gives the bijection between basic
\(\Gen\)-minimal modules and functorially finite quotient-closed
subcategories.  The following formulation also records compactness and
\(\qcl\)-independence.

\begin{proposition}\label{prop:compact-functorially-finite}
Let \(A\) be a finite-dimensional algebra over a field.
\begin{enumerate}
\item For \(\CC\in\LQ(A)\), the following are equivalent:
\begin{enumerate}
\item \(\CC\) is compact;
\item \(\CC=\Gen M\) for some \(M\in\CC\);
\item \(\CC\) is functorially finite.
\end{enumerate}
\item The following sets are in bijection:
\[
 \left\{\begin{array}{c}
 \text{finite \(\qcl\)-independent}\\[-2pt]
 \text{subsets \(B\subseteq\ind A\)}
 \end{array}\right\},\quad
 \left\{\begin{array}{c}
 \text{functorially finite}\\[-2pt]
 \text{members of \(\LQ(A)\)}
 \end{array}\right\},\quad
 \left\{\begin{array}{c}
 \text{isomorphism classes of}\\[-2pt]
 \text{basic \(\Gen\)-minimal modules}
 \end{array}\right\}.
\]
The first set maps to the second by \(B\mapsto\Gen(\add B)\) and to the
third by \(B\mapsto\bigoplus_{X\in B}X\).
The inverse map from the second set to the first is
\(\CC\mapsto\mathcal P_0(\CC)\).
\end{enumerate}
\end{proposition}

\begin{proof}
(1)
We prove the equivalence in the order (a) $\Leftrightarrow$ (b)
$\Leftrightarrow$ (c).  By
Theorem~\ref{thm:infinite-convex-structure}(1), a quotient-closed subcategory
is compact if and only if it is generated by a finite set of indecomposables,
equivalently, if and only if it has the form \(\Gen M\).  Thus (a) and (b)
are equivalent.  Proposition~\ref{prop:as-approximation} shows that (b) and
(c) are equivalent.

(2)
By Corollary~\ref{cor:quotient-convex-structure}(2), every compact member of
\(\LQ(A)\) has the form \(\Gen(\add B)\) for a unique finite
\(\qcl\)-independent set \(B\subseteq\ind A\).  Part~(1) identifies these
compact subcategories with the functorially finite members of \(\LQ(A)\).
The indecomposable summands of a basic module \(M\) are
\(\qcl\)-independent exactly when \(M\) is \(\Gen\)-minimal, so
\[
 B\longmapsto\bigoplus_{X\in B}X
\]
identifies finite \(\qcl\)-independent sets with the isomorphism classes of
basic \(\Gen\)-minimal modules.  Under the bijection with quotient-closed
subcategories, Proposition~\ref{prop:split-projective-extreme} identifies
\(B\) with the splitting projectives of
\(\Gen(\add B)\).
\end{proof}

Applying Proposition~\ref{prop:compact-functorially-finite} to \(A^{\op}\)
gives the dual classification.  The submodule-closed subcategories that are
functorially finite are exactly
\(\Cogen M\) for basic \(\Cogen\)-minimal modules \(M\), unique up to
isomorphism.

\section{The Dynkin case and counts for one or two omitted indecomposables}
\label{sec:conjecture}

In this section we prove the quotient--submodule equidistribution conjecture
for path algebras of Dynkin quivers. We then count the subcategories omitting
one or two indecomposable modules.

\subsection{Boundary coefficients and the Dynkin formula}

We first recall the working conjecture.  For a representation-finite
finite-dimensional algebra \(A\) over a field, the quotient--submodule
equidistribution conjecture (Conjecture~\ref{conj:equidistribution})
asserts the equality of size generating polynomials
\[
 \RQ(A;q)=\RS(A;q).
\]

\begin{remark}\label{rem:opposite-duality}
The duality \(\DDual\) gives a contravariant equivalence between
\(\mod A\) and \(\mod A^{\op}\) which preserves size, and hence an isomorphism
\[
 \LQ(A^{\op})\xrightarrow{\sim}\LS(A).
\]
Consequently, the equidistribution conjecture holds whenever \(A\) is
Morita equivalent to \(A^{\op}\).
\end{remark}

\begin{lemma}\label{lem:boundary-coefficients}
Let \(A\) be a finite-dimensional algebra with \(r\) simple modules.
\begin{enumerate}
\item The quotient-closed subcategories of size one are precisely \(\add S\)
with \(S\) simple.  The submodule-closed subcategories of size one are also
precisely \(\add S\) with \(S\) simple.
\item The set \(\ind A\setminus\{X\}\) is quotient-closed if and only if
\(X\) is projective, and \(\ind A\setminus\{X\}\) is submodule-closed if
and only if \(X\) is injective.
\item If \(A\) is representation-finite and \(N=|\ind A|\), then
\[
 |\LQ^0(A)|=|\LS^0(A)|=|\LQ^N(A)|=|\LS^N(A)|=1
\]
and
\[
 |\LQ^1(A)|=|\LS^1(A)|=|\LQ^{N-1}(A)|=|\LS^{N-1}(A)|=r.
\]
\end{enumerate}
\end{lemma}

\begin{proof}
(1)
If \(X\) is not simple, it has a simple quotient and a simple submodule,
neither isomorphic to \(X\).  Conversely, \(\add X\) is quotient-closed and
submodule-closed when \(X\) is simple.

(2)
Apply Proposition~\ref{prop:split-projective-extreme} to \(\CC=\mod A\).
The splitting projectives in \(\mod A\) are precisely the indecomposable
projective modules, so the quotient-closed assertion follows.  Applying the
assertion to \(A^{\op}\) and using \(\DDual\) gives
\(\ind A\setminus\{X\}\) is submodule-closed if and only if \(X\) is
injective.

(3)
The unique subcategories of sizes zero and \(N\) are \(0\) and \(\mod A\),
respectively, and both are quotient-closed and submodule-closed.  Therefore
\[
 |\LQ^0(A)|=|\LS^0(A)|=|\LQ^N(A)|=|\LS^N(A)|=1.
\]
Part~(1) gives
\[
 |\LQ^1(A)|=|\LS^1(A)|=r.
\]
By part~(2), the bijections from the simple modules to the indecomposable
projective and indecomposable injective modules give
\[
 |\LQ^{N-1}(A)|=|\LS^{N-1}(A)|=r.
\]
\end{proof}

\begin{example}\label{ex:inward-a3}
Let
\[
 Q:\quad 1\longrightarrow2\longleftarrow3,\qquad A=\kk Q.
\]
The six indecomposable modules are
\[
 S_1,\ S_2,\ S_3,\ M_{12},\ M_{23},\ M_{123},
\]
where the subscript records the support.  The indecomposable projectives are
\[
 \{M_{12},S_2,M_{23}\},
\]
and the indecomposable injectives are
\[
 \{S_1,M_{123},S_3\}.
\]
The two size generating polynomials of \(A\) are
\begin{equation}\label{eq:inward-a3-polynomial}
 \RQ(A;q)=\RS(A;q)
 =1+3q+5q^2+6q^3+5q^4+3q^5+q^6.
\end{equation}

Among the six elements of \(\LQ^3(A)\), one has exactly one lower cover,
whereas no element of \(\LS^3(A)\) has exactly one lower cover.  Hence the
two lattices are not isomorphic.  In particular, no permutation of
\(\ind A\) maps \(\LQ(A)\) onto \(\LS(A)\).
\end{example}

Let \(Q\) be a Dynkin quiver, and let
\(W_Q\) be its Weyl group.  For each vertex \(i\), put
\(P_i=e_i\kk Q\), and number the vertices so that
\(i<j\) implies \(\Hom_{\kk Q}(P_j,P_i)=0\).
Let \(s_i\) be the simple reflection at vertex \(i\), let \(r=|Q_0|\),
and put \(c=s_1\cdots s_r\).  Match the \(k\)-th occurrence of \(s_i\)
in the infinite word \(c^\infty=s_1\cdots s_rs_1\cdots s_r\cdots\) with
\(\tau^{-(k-1)}P_i\).
Since \(Q\) is Dynkin, \(\tau^{-(k-1)}P_i=0\) for all large \(k\).  After
the letters matched with zero modules are omitted, the remaining letters
form a finite word with \(N=|\ind\kk Q|\) letters, matched bijectively
with \(\ind\kk Q\).  The following is the Dynkin
specialization of the Oppermann--Reiten--Thomas correspondence.

\begin{proposition}[{\cite[Theorem~2.2]{ORT}}]
\label{prop:ort-dynkin}
For \(w\in W_Q\), take the reduced subword of \(c^\infty\) whose increasing
set of positions is lexicographically least among those representing
\(w\).  This subword uses only positions matched with nonzero modules.  Let
\(\XX(w)\) be the set of indecomposable \(\kk Q\)-modules matched with its
positions.
Then
\[
 w\longmapsto\add\bigl(\ind\kk Q\setminus\XX(w)\bigr)
\]
is a bijection from \(W_Q\) to \(\LQ(\kk Q)\).
\end{proposition}

\begin{corollary}\label{cor:ort-dynkin-length}
For every \(w\in W_Q\), the subcategory
\(\add(\ind\kk Q\setminus\XX(w))\) omits exactly \(\ell(w)\)
indecomposable modules.
\end{corollary}

\begin{proof}
The reduced subword defining \(\XX(w)\) has \(\ell(w)\) positions, and these
positions are matched with distinct indecomposable modules.
\end{proof}

Remark~\ref{rem:directed-recovers-ort} shows that
Theorem~\ref{thm:directed-main} specializes to
Proposition~\ref{prop:ort-dynkin}, and hence gives another proof of the
proposition.

\begin{lemma}[{\cite[Proposition~2.3.2(ii), Corollary~2.3.3(i), and
Theorem~7.1.5]{BjornerBrenti}}]
\label{lem:weyl-length-poincare}
Let \(W\) be a finite Weyl group of rank \(r\), with longest element \(w_0\)
and exponents \(m_1,\ldots,m_r\).  Then the following statements hold.
\begin{enumerate}
\item For every \(w\in W\),
\[
 \ell(w_0w)=\ell(w_0)-\ell(w).
\]
\item One has
\[
 \sum_{w\in W}q^{\ell(w)}
 =\prod_{i=1}^{r}(1+q+\cdots+q^{m_i}).
\]
\end{enumerate}
\end{lemma}

\begin{theorem}\label{thm:dynkin}
Let \(Q\) be a Dynkin quiver.  Then
\[
 \RQ(\kk Q;q)=\RS(\kk Q;q)
 =\sum_{w\in W_Q}q^{\ell(w)}
 =\prod_{i=1}^{r}(1+q+\cdots+q^{m_i}),
\]
where \(r=|Q_0|\) and \(m_1,\ldots,m_r\) are the exponents of \(W_Q\).
\end{theorem}

\begin{proof}
Put \(N=|\ind\kk Q|\), and let \(w_0\) be the longest element of \(W_Q\).
Proposition~\ref{prop:ort-dynkin} and
Corollary~\ref{cor:ort-dynkin-length} give
\[
 \RQ(\kk Q;q)=\sum_{w\in W_Q}q^{N-\ell(w)}.
\]
Gabriel's theorem
\cite[Theorem~VII.5.10(b),(c)]{AssemSimsonSkowronski} and the
inversion-set formula for Coxeter length
\cite[Proposition~4.4.4]{BjornerBrenti} give
\[
 N=|\Phi_Q^+|=\ell(w_0).
\]

Reversing the arrows of \(Q\) does not change its Weyl group or length
function.  Proposition~\ref{prop:ort-dynkin} and
Corollary~\ref{cor:ort-dynkin-length} for \(Q^{\op}\) therefore give
\[
 \RQ(\kk Q^{\op};q)=\sum_{w\in W_Q}q^{N-\ell(w)}.
\]
The size-preserving duality in Remark~\ref{rem:opposite-duality} gives
\[
 \LQ(\kk Q^{\op})\longrightarrow\LS(\kk Q).
\]
Consequently,
\[
 \RS(\kk Q;q)=\sum_{w\in W_Q}q^{N-\ell(w)}.
\]
Lemma~\ref{lem:weyl-length-poincare}(1) and \(N=\ell(w_0)\) show that
multiplication by \(w_0\) sends length \(\ell(w)\) to \(N-\ell(w)\).  Hence
\[
 \RQ(\kk Q;q)=\RS(\kk Q;q)=\sum_{w\in W_Q}q^{\ell(w)}.
\]
Finally, Lemma~\ref{lem:weyl-length-poincare}(2) gives the product formula.
\end{proof}

For type \(A_3\), Theorem~\ref{thm:dynkin} gives
\((1+q)(1+q+q^2)(1+q+q^2+q^3)\), which is precisely
\eqref{eq:inward-a3-polynomial}.

\subsection{A formula for two omitted indecomposable modules}
\label{sec:two-omitted}

Let \(A\) be a finite-dimensional algebra.  We count the subcategories whose
complement in \(\ind A\) consists of two indecomposable modules.
The case of one omitted indecomposable is
Lemma~\ref{lem:boundary-coefficients}; with two omitted modules the counts
need not agree
(Example~\ref{ex:infinite-counterexample}).

Let \(\overline q_2(A)\) be the number of subsets
\(D\subseteq\ind A\) with \(|D|=2\) such that
\(\add(\ind A\setminus D)\) is quotient-closed.  Let
\(\overline s_2(A)\) be the number of subsets
\(D\subseteq\ind A\) with \(|D|=2\) such that
\(\add(\ind A\setminus D)\) is submodule-closed.  In both symbols, the
index \(2\) is
\(|D|\), the number of omitted indecomposables, rather than the size of
\(\add(\ind A\setminus D)\).
We use the Auslander--Reiten mesh correspondence in the following form.

\begin{lemma}[{\cite[Proposition~IV.3.8(a)]{AssemSimsonSkowronski}}]
\label{lem:ar-mesh-correspondence}
Let \(X,Y\in\ind A\), and suppose that \(Y\) is nonprojective.  Then there
is an arrow \(X\to Y\) in the Auslander--Reiten quiver if and only if there
is an arrow \(\tau Y\to X\).
\end{lemma}

\begin{proposition}\label{prop:cofinite-two}
Let \(r\) be the number of simple \(A\)-modules.  Put
\[
\begin{aligned}
 \beta_{\quot}(A)
 &=|\{\,(P,Z)\mid P\text{ indecomposable projective},\
 Z\text{ nonprojective},\ \text{there is an arrow }P\to Z\,\}|,\\
 \beta_{\sub}(A)
 &=|\{\,(Z,I)\mid Z\text{ noninjective},\
 I\text{ indecomposable injective},\ \text{there is an arrow }Z\to I\,\}|.
\end{aligned}
\]
Then the following equalities hold.
\begin{enumerate}
\item \(\displaystyle
 \overline q_2(A)=\binom r2+\beta_{\quot}(A)\).
\item \(\displaystyle
 \overline s_2(A)=\binom r2+\beta_{\sub}(A)\).
\end{enumerate}
\end{proposition}

\begin{proof}
(1)
Put \(E=\ind A\).  Suppose that \(E\setminus D\) is quotient-closed for a
two-element set \(D\).  Applying
Corollary~\ref{cor:finite-quotient-interval-chain} to
\(\add(E\setminus D)\subseteq\mod A\) gives an ordering \(D=\{P,Z\}\)
such that \(P\) is splitting projective in \(\mod A\) and \(Z\) is
splitting projective in \(\CC_P=\add(E\setminus\{P\})\).
In particular, \(P\) is projective.  Conversely, if both members of \(D\)
are projective, then \(E\setminus D\) is the intersection of two
quotient-closed one-point complements, and hence is quotient-closed.

It remains to characterize the pairs \(D=\{P,Z\}\), where \(P\) is
projective and \(Z\) is nonprojective.  Put
\(\CC_P=\add(E\setminus\{P\})\).
Lemma~\ref{lem:boundary-coefficients}(2) shows that \(\CC_P\) is
quotient-closed.  Proposition~\ref{prop:split-projective-extreme} gives
\begin{equation}\label{eq:cofinite-two-splitting}
 \add(E\setminus\{P,Z\})\text{ is quotient-closed}
 \quad\Longleftrightarrow\quad
 Z\text{ is splitting projective in }\CC_P.
\end{equation}

Let \(\theta Z\to Z\) be the minimal right almost split map.  If \(Z\) is
splitting projective in \(\CC_P\), then the nonsplit surjection
\(\theta Z\to Z\) cannot have \(\theta Z\in\CC_P\).  Hence \(P\) is a summand of
\(\theta Z\), or equivalently, there is an arrow \(P\to Z\) in the
Auslander--Reiten quiver.

Conversely, suppose that there is an arrow \(P\to Z\), and choose an
irreducible map \(f:P\to Z\). For a surjection \(g:M\twoheadrightarrow
Z\) with \(M\in\CC_P\), projectivity gives a factorization
\[
 P\longrightarrow M\xrightarrow{g} Z
\]
of \(f\).  Since \(f\) is irreducible, the first map is a section or \(g\)
is a retraction.  The first possibility would make \(P\) a summand of an
object of \(\CC_P\), which is impossible.  Thus \(g\) is a retraction, so
\(Z\) is splitting projective in \(\CC_P\).  By
\eqref{eq:cofinite-two-splitting}, the complement of \(\{P,Z\}\) is
quotient-closed.  The allowed sets \(D\) are therefore the pairs of
indecomposable projectives and the pairs \(\{P,Z\}\) counted by
\(\beta_{\quot}(A)\), which proves the formula for \(\overline q_2(A)\).

(2)
Apply (1) to \(A^{\op}\) and then apply \(\DDual\).  The duality identifies
\(\overline q_2(A^{\op})\) with
\(\overline s_2(A)\) and \(\beta_{\quot}(A^{\op})\) with
\(\beta_{\sub}(A)\), which proves the formula for
\(\overline s_2(A)\).
\end{proof}

For \(M\in\mod A\), let \(\ind M\) be the subset of \(\ind A\) consisting
of the indecomposable direct summands of \(M\).

\begin{corollary}\label{cor:cofinite-two-rad-soc}
The quantities \(\overline q_2(A)\) and \(\overline s_2(A)\) are finite, and
\[
 \beta_{\quot}(A)
 =\sum_P|\ind(\rad P)|,
 \qquad
 \beta_{\sub}(A)
 =\sum_I|\ind(I/\soc I)|.
\]
Here \(P\) ranges over the indecomposable projective modules and \(I\)
ranges over the indecomposable injective modules.
\end{corollary}

\begin{proof}
For an indecomposable projective module \(P\), the inclusion
\(\rad P\to P\) is minimal right almost split.  Hence the indecomposable
summands of \(\rad P\) are precisely the modules \(X\) for which
there is an arrow \(X\to P\).  No such summand is injective, since an injective
submodule of \(P\) would split.  The Auslander--Reiten translation \(\tau\)
gives a bijection from the nonprojective indecomposables to the noninjective
indecomposables, so
Lemma~\ref{lem:ar-mesh-correspondence} gives
\[
 |\{\,Z\in\ind A\mid Z\text{ nonprojective and }
       \text{there is an arrow }P\to Z\,\}|=|\ind(\rad P)|.
\]
Summing over the indecomposable projective modules gives the formula for
\(\beta_{\quot}(A)\).  Applying the formula for \(\beta_{\quot}(A)\) to
\(A^{\op}\) gives the formula for \(\beta_{\sub}(A)\).  The two right-hand
sums in Corollary~\ref{cor:cofinite-two-rad-soc} are finite, so
Proposition~\ref{prop:cofinite-two} shows that \(\overline q_2(A)\) and
\(\overline s_2(A)\) are finite.
\end{proof}

\begin{example}\label{ex:infinite-counterexample}
Let \(\kk\) be any field and consider
\[
 Q:\quad 1\xrightarrow{a}2
   \mathrel{\substack{\xrightarrow{\ b\ }\\[-.6ex]\xrightarrow[\ c\ ]{}}}3,
 \qquad A=\kk Q/(ac),
\]
where the path \(a\) followed by \(c\) is zero.  A direct calculation gives
\[
 \sum_P|\ind(\rad P)|=2,
 \qquad
 \sum_I|\ind(I/\soc I)|=3.
\]
Thus Corollary~\ref{cor:cofinite-two-rad-soc} gives
\(\beta_{\quot}(A)=2\) and \(\beta_{\sub}(A)=3\).  Since \(r=3\),
Proposition~\ref{prop:cofinite-two} gives
\[
 \overline q_2(A)=5,\qquad \overline s_2(A)=6.
\]
The factor algebra of \(A\) by the ideal generated by the vertex
idempotent at \(1\) is the Kronecker algebra, so \(A\) has infinite
representation type.  Thus the two counts can differ, over every
field, already with two omitted indecomposable modules:
\(\overline q_2(A)=5\ne6=\overline s_2(A)\).
\end{example}

When \(A\) is representation-finite, we later prove that
\(\beta_{\quot}(A)=\beta_{\sub}(A)\) by
Corollary~\ref{cor:cosize-two}.

\section{Product formulas for Nakayama and radical-square-zero algebras}
\label{sec:classes}

In this section, we prove the quotient--submodule equidistribution conjecture
for Nakayama algebras and radical square zero algebras. To this end, we obtain
closed product formulas for the polynomials
\(\RQ(A;q)\) and \(\RS(A;q)\).

\subsection{Lattice decompositions for Nakayama algebras}\label{sec:nakayama}

Recall that a finite-dimensional algebra is \emph{Nakayama} if all its
indecomposable projective and injective modules are uniserial.  Write
\([a]_q=1+q+\cdots+q^{a-1}\).

\begin{proposition}[{\cite[Theorem~VI.2.1 and
Lemma~IV.2.5]{ARS}}]\label{prop:nakayama-structure}
Every indecomposable module over a finite-dimensional Nakayama algebra is
uniserial and is uniquely determined by its simple top and its composition
length.
\end{proposition}

Proposition~\ref{prop:nakayama-structure} shows that the indecomposables with
a fixed simple top form a chain under quotients.  The indecomposables with a
fixed simple socle form a chain under submodules.

\begin{theorem}\label{thm:nakayama}
Let \(A\) be a finite-dimensional Nakayama algebra over a field.  Let
\(S_1,\ldots,S_r\) be the simple modules, let \(P_i\) be the projective
cover of \(S_i\), and let \(c_i\) be the composition length of
\(P_i\).  Let \(I_i\) be the injective envelope of \(S_i\), and let \(d_i\)
be its composition length.  Then the following hold.
\begin{enumerate}
\item There is a graded lattice isomorphism
\[
 \LQ(A)\cong\prod_{i=1}^r[0,c_i].
\]
\item There is a graded lattice isomorphism
\[
 \LS(A)\cong\prod_{i=1}^r[0,d_i],
\]
and the multisets \(\{c_1,\ldots,c_r\}\) and
\(\{d_1,\ldots,d_r\}\) are equal.  In particular,
\(\LQ(A)\cong\LS(A)\) as graded lattices.
\item One has
\[
 \RQ(A;q)=\RS(A;q)=\prod_{i=1}^r[c_i+1]_q.
\]
\end{enumerate}
\end{theorem}

\begin{proof}
(1)
We may assume that \(A\) is basic.  By
Proposition~\ref{prop:nakayama-structure}, every indecomposable \(A\)-module
is uniserial and has a projective cover \(P_i\).  For each \(i\), the
indecomposable modules with top \(S_i\) are the \(c_i\) nonzero quotients of
\(P_i\), one in each possible composition length.  Ordered by increasing
length, they form a chain in which each module is a quotient of the
next.  Every quotient of a member of a quotient-closed
subcategory again belongs to it, so a quotient-closed subcategory contains
precisely the \(t_i\) shortest modules of this chain for some
\(0\leq t_i\leq c_i\).

Every indecomposable quotient of a finite direct sum of uniserial modules is
a quotient of one of its indecomposable summands.  Indeed, if a uniserial
module \(U\) is generated by maps
from finitely many uniserial modules, their images generate
\(U/\rad U\).  Since this quotient is simple, the image of one of the
modules is nonzero in \(U/\rad U\); by Nakayama's lemma, a submodule of
\(U\) with nonzero image in \(U/\rad U\) equals \(U\).  Thus \(U\) is a
quotient of that single module.  We have therefore obtained
\[
 \LQ(A)\cong\prod_{i=1}^r[0,c_i]
\]
as a graded lattice, where \([0,c_i]\) is the chain
\(0<1<\cdots<c_i\) and the grade of an element is its integer value.

(2)
Dually, the indecomposables with socle \(S_j\) form a submodule chain
of lengths \(1,\ldots,d_j\).  A submodule-closed
subcategory contains an
initial segment of the chain ordered by increasing length.  A monomorphism
from a uniserial module into a finite direct sum has a monomorphic component:
otherwise every component has a nonzero kernel, and the linearly ordered
nonzero kernels have a nonzero intersection, which lies in the kernel of
the map to the direct sum, contradicting its injectivity.  Thus
\[
 \LS(A)\cong\prod_{j=1}^r[0,d_j].
\]

For every \(a\geq1\), the
indecomposables of length \(a\) are exactly the length-\(a\) quotients
of the projective covers \(P_i\) with \(c_i\geq a\), one for each such
\(i\); dually, they are exactly the length-\(a\) submodules of the
injective envelopes \(I_j\) with \(d_j\geq a\), one for each such
\(j\).  Counting them in these two ways gives
\[
 |\{i\mid c_i\geq a\}|=|\{j\mid d_j\geq a\}|.
\]
Thus the multisets \(\{c_i\}\) and \(\{d_j\}\) agree, which proves (2).

(3)
Since the isomorphisms in (1) and (2) preserve size, they give
\[
 \RQ(A;q)=\prod_{i=1}^r(1+q+\cdots+q^{c_i}),
 \qquad
 \RS(A;q)=\prod_{j=1}^r(1+q+\cdots+q^{d_j}).
\]
The multiset equality in (2) and the definition of \([a]_q\) give the
formula in (3).
\end{proof}

\begin{example}
\label{ex:nakayama-a3}
Let
\[
 Q:\quad 1\longrightarrow2\longrightarrow3,\qquad A=\kk Q.
\]
The indecomposable projective modules have composition lengths
\((c_1,c_2,c_3)=(3,2,1)\).  Hence Theorem~\ref{thm:nakayama} gives
\[
 \RQ(A;q)=\RS(A;q)=[4]_q[3]_q[2]_q.
\]
\end{example}

Example~\ref{ex:inward-a3} shows that the graded lattice isomorphism in
Theorem~\ref{thm:nakayama} can fail outside the Nakayama case.

\subsection{The separated-quiver formula for radical-square-zero algebras}
\label{sec:radical-square-zero}

Let \(A\) be basic over the algebraically closed field \(\kk\), suppose
that \((\rad A)^2=0\), and let \(Q\) be its Gabriel quiver.  Let \(J_Q\)
denote the ideal of \(\kk Q\) generated by the arrows of \(Q\).

\begin{proposition}\label{prop:radsq-presentation}
There is an isomorphism \(A\cong\kk Q/J_Q^2\).
\end{proposition}

\begin{proof}
By \cite[Theorem~II.3.7 and
Corollary~II.2.11]{AssemSimsonSkowronski}, there is a bound-quiver
presentation \(A\cong\kk Q/I\) with \(I\subseteq J_Q^2\), and
\((\rad A)^2=J_Q^2/I\).  The assumption \((\rad A)^2=0\) gives
\(I=J_Q^2\).
\end{proof}

The separated quiver
records the top and the radical of an \(A\)-module as a representation of a
hereditary algebra.  We use it to compare quotient closure for \(A\) with
quotient closure for that hereditary algebra.

\begin{definition}\label{def:separated-quiver}
The \emph{separated quiver} \(Q^{\rm s}\) has two copies \(i^+,i^-\) of
every vertex \(i\) of \(Q\), and an arrow \(i^+\to j^-\) for every arrow
\(i\to j\) of \(Q\).  Isolated copies are also vertices of
\(Q^{\rm s}\).  Let \(G\) be its underlying unoriented graph.
\end{definition}

Write \(W(G)\) for the Coxeter group generated by the vertices of \(G\),
write \(\ell\) for its Coxeter length function, and put
\[
 P_G(q)=\sum_{w\in W(G)}q^{\ell(w)}.
\]

\begin{construction}\label{con:separated-representation}
Put \(B=\kk Q^{\rm s}\).  For an \(A\)-module \(M\), define its separated
representation \(M^{\rm s}\) by
\[
 (M^{\rm s})_{i^+}=(M/\rad M)e_i,\qquad
 (M^{\rm s})_{i^-}=(\rad M)e_i,
\]
with the maps induced by the action of the arrows.
\end{construction}

\begin{proposition}[{\cite[Lemma~X.2.1]{ARS}}]
\label{prop:separated-representations}
The assignment \(M\mapsto M^{\rm s}\) defines a full functor
\(\mod A\to\mod B\), and for all \(M,N\in\mod A\) there is an exact sequence
\begin{equation}\label{eq:radsq-hom-sequence}
 0\longrightarrow\Hom_A(M,\rad N)
 \longrightarrow\Hom_A(M,N)
 \longrightarrow\Hom_B(M^{\rm s},N^{\rm s})
 \longrightarrow0.
\end{equation}
The functor reflects isomorphisms and indecomposability and induces a
bijection between the nonsimple indecomposable \(A\)-modules and the
nonsimple indecomposable \(B\)-modules.
\end{proposition}

\begin{lemma}[{\cite[Theorem~X.2.6]{ARS}}]
\label{lem:radsq-separated-dynkin}
If \(A\) is representation-finite, then the underlying graph \(G\) of its
separated quiver is a disjoint union of Dynkin diagrams.
\end{lemma}

\begin{theorem}\label{thm:radical-square-zero}
Let \(A\) be a representation-finite finite-dimensional algebra over an
algebraically closed field, and suppose that \((\rad A)^2=0\).  Let \(G\)
be the underlying graph of the separated quiver of a basic algebra Morita
equivalent to \(A\), including its isolated vertices, and let \(r\) be
the number of simple \(A\)-modules.  Then
\[
 \RQ(A;q)=\RS(A;q)=\frac{P_G(q)}{(1+q)^r}.
\]
\end{theorem}

\begin{proof}
We may assume that \(A\) is basic.  Proposition~\ref{prop:radsq-presentation}
gives \(A=\kk Q/J_Q^2\), and we put \(B=\kk Q^{\rm s}\).
We divide the proof into four steps.

\medskip
\noindent\textbf{Step 1: Comparison of nonsimple modules.}
Proposition~\ref{prop:separated-representations} identifies the nonsimple
indecomposable \(A\)-modules with the nonsimple indecomposable \(B\)-modules.

Let \(\XX\) be a set of nonsimple indecomposable \(A\)-modules, put
\(\XX^{\rm s}=\{X^{\rm s}\mid X\in\XX\}\), and let \(T(\XX)\) be the set of
vertices which occur in the tops of members of \(\XX\).  For every nonsimple
indecomposable \(Y\), Nakayama's lemma gives
\begin{equation}\label{eq:radsq-top-coverage}
 Y\in\Gen(\add\XX)
 \quad\Longleftrightarrow\quad
 \tr_{\add\XX}(Y)+\rad Y=Y.
\end{equation}
Here the trace is the sum of the images of all maps from \(\add\XX\) to
\(Y\).  The generation condition at the \(-\) vertices gives
\[
 \rad_B(Y^{\rm s})=\bigoplus_i(Y^{\rm s})_{i^-},\qquad
 \top_B(Y^{\rm s})=\bigoplus_i(Y^{\rm s})_{i^+}.
\]
For every vertex \(i\), the surjection in
\eqref{eq:radsq-hom-sequence} gives the equality
\[
 \bigl(\tr_{\add(\XX^{\rm s})}(Y^{\rm s})\bigr)_{i^+}
 =\left(\frac{\tr_{\add\XX}(Y)+\rad Y}{\rad Y}\right)e_i
 \subseteq (Y^{\rm s})_{i^+}.
\]
Nakayama's lemma therefore gives
\[
 Y\in\Gen_A(\add\XX)
 \quad\Longleftrightarrow\quad
 Y^{\rm s}\in\Gen_B(\add(\XX^{\rm s})).
\]
Moreover, for \(\Lambda\in\{A,B\}\), a simple \(\Lambda\)-module \(S\),
and a nonsimple indecomposable \(\Lambda\)-module \(Y\),
\[
 \Image g\subseteq\rad_\Lambda Y
 \qquad\text{for every }g\in\Hom_\Lambda(S,Y);
\]
otherwise \(S\) would be a direct summand of \(Y\).
Consequently, Lemma~\ref{lem:indecomposable-quotient-test} shows that
\(\XX\) is the set of nonsimple members of a quotient-closed subcategory of
\(\mod A\) if and only if \(\XX^{\rm s}\) is the set of nonsimple members of
a quotient-closed subcategory of \(\mod B\).  By
\eqref{eq:radsq-top-coverage}, these two conditions are equivalent to
\begin{equation}\label{eq:radsq-nonsimple-closed}
 \tr_{\add\XX}(Y)+\rad Y=Y \quad\Longrightarrow\quad Y\in\XX
 \qquad\text{for every nonsimple indecomposable \(Y\)}.
\end{equation}

Let \(\mathfrak F\) be the collection of sets \(\XX\) of nonsimple
indecomposable \(A\)-modules satisfying
\eqref{eq:radsq-nonsimple-closed}.

\medskip
\noindent\textbf{Step 2: Counting the choices of simple modules.}
Fix \(\XX\in\mathfrak F\).  A quotient-closed subcategory of \(\mod A\) with
nonsimple members \(\XX\) must contain the simples in \(T(\XX)\), because
each is a quotient of a member of \(\XX\).  Every simple outside \(T(\XX)\)
may independently be included or omitted.  Indeed, for every simple module
\(S\) and every nonsimple indecomposable \(Y\), the containment
\(\Image g\subseteq\rad Y\) for \(g\in\Hom_A(S,Y)\) gives
\[
 \frac{\tr_{\add(\XX\cup\{S\})}(Y)+\rad Y}{\rad Y}
 =
 \frac{\tr_{\add\XX}(Y)+\rad Y}{\rad Y}.
\]
Hence
\begin{equation}\label{eq:radsq-A-sum}
 \RQ(A;q)=
 \sum_{\XX\in\mathfrak F}
 q^{|\XX|+|T(\XX)|}(1+q)^{r-|T(\XX)|}.
\end{equation}
For \(B\), every quotient-closed subcategory with nonsimple members
\(\XX^{\rm s}\) must contain the simple at \(i^+\) for each vertex
\(i\in T(\XX)\).  The simples at the \(+\)-vertices outside \(T(\XX)\) and
at all \(r\) vertices \(i^-\) may be chosen independently, giving the exponent
\((r-|T(\XX)|)+r=2r-|T(\XX)|\).  Therefore
\begin{equation}\label{eq:radsq-B-sum}
 \RQ(B;q)=
 \sum_{\XX\in\mathfrak F}
 q^{|\XX|+|T(\XX)|}(1+q)^{2r-|T(\XX)|}.
\end{equation}
Equations~\eqref{eq:radsq-A-sum} and~\eqref{eq:radsq-B-sum} give
\begin{equation}\label{eq:radsq-factor}
 \RQ(A;q)=\frac{\RQ(B;q)}{(1+q)^r}.
\end{equation}

\medskip
\noindent\textbf{Step 3: Evaluation of the generating polynomial for \(B\).}
By Lemma~\ref{lem:radsq-separated-dynkin}, \(G\) is a disjoint union of
Dynkin diagrams.  Hence \(B\) is representation-finite.  Apply
Theorem~\ref{thm:dynkin} to every Dynkin component; each isolated vertex
contributes the factor \(1+q\).  Multiplying the component formulas gives
\(\RQ(B;q)=P_G(q)\).  Equation~\eqref{eq:radsq-factor} now gives
\(\RQ(A;q)=P_G(q)/(1+q)^r\).

\medskip
\noindent\textbf{Step 4: The submodule-closed case by duality.}
The underlying graph of the separated quiver of \(Q^{\op}\) is isomorphic
to \(G\), by exchanging the \(+\) and \(-\) copies.  Applying the established
formula for \(\RQ(A;q)\) to \(A^{\op}\) gives
\(\RQ(A^{\op};q)=P_G(q)/(1+q)^r\).
Remark~\ref{rem:opposite-duality} gives
\[
 \RS(A;q)=\RQ(A^{\op};q)
 =\frac{P_G(q)}{(1+q)^r}.
\]
\end{proof}

For the quiver \(1\to2\leftarrow3\) in Example~\ref{ex:inward-a3}, the
separated graph is \(A_3\sqcup A_1^3\), so
Theorem~\ref{thm:radical-square-zero} recovers
\eqref{eq:inward-a3-polynomial}.

\begin{remark}\label{rem:radsq-coset}
The vertices \(i^-\) generate a standard parabolic subgroup \(W_J\) of
type \(A_1^r\).  Hence both \(\RQ(A;q)\) and \(\RS(A;q)\) are the Poincar\'e
polynomial of the minimal coset representatives for \(W(G)/W_J\)
\cite[Lemma~7.1.2]{BjornerBrenti}.
\end{remark}

\section{Factor categories and Auslander--Reiten criteria for closure}
\label{sec:categorical}

For representation-finite algebras, we give criteria for quotient and
submodule closure using the Auslander--Reiten quiver.

\subsection{Factor ladders as projective covers of radical layers}
\label{sec:factor-ladders}

Assume that \(A\) is representation-finite.  For a
nonprojective indecomposable \(X\), write
\[
 0\longrightarrow\tau X\longrightarrow \theta X\longrightarrow X
 \longrightarrow0
\]
for its almost split sequence.  For an indecomposable projective \(P\), put
\(\theta P=\rad P\), the domain of the minimal right almost split map to
\(P\).

\begin{lemma}[{\cite[Proposition~IV.4.9]{AssemSimsonSkowronski}}]
\label{lem:representation-finite-no-multiple-arrows}
The Auslander--Reiten quiver of \(A\) has no multiple arrows.
\end{lemma}

\begin{corollary}\label{cor:ar-middle-terms-squarefree}
Let \(Y\in\ind A\).  No indecomposable occurs more than once in \(\theta Y\).
\end{corollary}

\begin{proof}
For each \(X\in\ind A\), the multiplicity of \(X\) in \(\theta Y\) is the number
of arrows \(X\to Y\) in the Auslander--Reiten quiver.  Apply
Lemma~\ref{lem:representation-finite-no-multiple-arrows}.
\end{proof}

Fix a subset \(K\subseteq\ind A\), and put \(D=\ind A\setminus K\).  We call
the members of \(D\) \emph{omitted}.  We now use the Auslander--Reiten quiver
to define a recursion starting at each \(X\in D\).

\begin{construction}\label{con:factor-ladder}
We use the free abelian group on \(D\).  Define
\(\theta_D\) on a basis element \(X\in D\) to be the sum, with
multiplicities, of the indecomposable summands belonging to \(D\) in
\(\theta X\).
Define
\[
 \tau_DX=
 \begin{cases}
 \tau X,&X\text{ is nonprojective},\ \tau X\in D,
          \text{ and }\theta_DX\ne0,\\
 0,&\text{otherwise}.
 \end{cases}
\]
Extend both operators additively.  If \(v\) is an integral linear
combination of elements of \(D\), let \(v_+\) be obtained by replacing every
negative coefficient by zero.  The \emph{factor ladder} starting at \(X\) is
defined by
\[
 \theta^D_0X=X,\qquad \theta^D_1X=\theta_DX,
\]
and, for \(n\geq2\),
\begin{equation}\label{eq:factor-ladder}
 \theta^D_nX=
 \bigl(\theta_D\theta^D_{n-1}X
       -\tau_D\theta^D_{n-2}X\bigr)_+.
\end{equation}
If \(\theta^D_nX=0\), then
\(\theta^D_{n+1}X=(-\tau_D\theta^D_{n-1}X)_+=0\), and inductively every
later term is zero.
\end{construction}

We next give a categorical interpretation of the terms \(\theta_n^D X\).
Put
\[
 \overline{\AA}=(\mod A)/[\add K],
\]
where \([\add K]\) is the ideal of morphisms which factor through an object
of \(\add K\).
Write \(J_{\overline{\AA}}\) for the categorical radical of
\(\overline{\AA}\), namely the ideal generated by the maps between
indecomposable objects which are not isomorphisms.  For a finite formal sum
\(Z=\sum_i n_iZ_i\), where the \(Z_i\) are indecomposable and
\(n_i\in\NN\), use the additive convention
\[
 \mathsf P_Z^{\overline\AA}=\bigoplus_i
 \overline\AA(-,Z_i)^{n_i}.
\]

The factor category also records the quotient by the trace of \(\add K\).

\begin{lemma}\label{lem:trace-factor-hom}
For every \(A\)-module \(X\), evaluation at \(1\in A\) induces an
isomorphism
\[
 \frac{X}{\tr_{\add K}(X)}
 \cong\overline{\AA}(A,X).
\]
\end{lemma}

\begin{proof}
Under \(\Hom_A(A,X)\cong X\), the maps which factor through \(\add K\)
have values at \(1\) equal to the trace of \(\add K\) in \(X\).
\end{proof}

We use the following terminology for the factor category.

\begin{definition}[{\cite[Section~1.3]{IyamaRealization}}]
A \emph{\(\tau\)-category} is a Krull--Schmidt category satisfying
Iyama's right and left \(\tau\)-sequence axioms.  If \(\mathcal B\) is a
\(\tau\)-category and \(X\) is indecomposable, write
\[
 \tau_{\mathcal B}X\longrightarrow\theta_{\mathcal B}X
 \longrightarrow X
\]
for the right \(\tau\)-sequence ending at \(X\).  This sequence is
\emph{strict} if its first map is a monomorphism.  Dually, a left
\(\tau\)-sequence is strict if its second map is an epimorphism.  A
\(\tau\)-category is \emph{strict} if all its right and left
\(\tau\)-sequences are strict.
\end{definition}

\begin{proposition}\label{prop:iyama-factor-category}
The ideal quotient \(\overline\AA\) is a \(\tau\)-category whose
Auslander--Reiten species is obtained by restricting the species of
\(\mod A\) to the objects in \(D\).  Moreover, the following hold.
\begin{enumerate}
\item For \(X\in D\) and every \(n\geq0\), there is a projective cover in
the category of finitely presented contravariant functors on
\(\overline\AA\),
\[
 \mathsf P_{\theta^D_nX}^{\overline{\AA}}
 \longrightarrow J_{\overline{\AA}}^n(-,X)
 \longrightarrow0.
\]
\item The radical \(J_{\overline\AA}\) is nilpotent, and
\(\theta^D_nX=0\) for all sufficiently large \(n\).
\end{enumerate}
\end{proposition}

\begin{proof}
By \cite[1.4(2), (3)]{IyamaTauII}, the ideal quotient \(\overline\AA\) is a
\(\tau\)-category whose Auslander--Reiten species is obtained by restricting
the species of \(\mod A\).  Let \(\theta^+\) and \(\tau^+\) denote the
predecessor and translation operators of \(\overline\AA\).  The restricted
species description gives
\[
 \theta^+X=\theta_DX,
 \qquad
 \tau^+X=\tau_DX
 \quad (X\in D).
\]
In \cite[7.2]{IyamaTauI}, the operators \(\theta_n^+\) are defined by
\[
 \theta_0^+X=X,\qquad \theta_1^+X=\theta^+X,
 \qquad
 \theta_n^+X=
 \bigl(\theta^+\theta_{n-1}^+X
       -\tau^+\theta_{n-2}^+X\bigr)_+
 \quad(n\geq2).
\]
Consequently, \(\theta_n^+X=\theta_n^D X\) for every \(n\geq0\).

(1)
By \cite[7.2(2)]{IyamaTauI}, there is a projective cover
\[
 \mathsf P_{\theta^D_nX}^{\overline{\AA}}
 \longrightarrow J_{\overline{\AA}}^n(-,X)
 \longrightarrow0
\]
for this recursion.

(2)
By \cite[1.3(5)]{IyamaTauII}, the radical powers in \(\overline\AA\) are the
images of the radical powers in \(\mod A\).  Since
\(A\) is representation-finite, the radical of \(\mod A\) is nilpotent,
and hence so is \(J_{\overline{\AA}}\).  Choose \(n\) such that
\(J_{\overline{\AA}}^n=0\).  Part~(1) gives a projective cover
\(\mathsf P_{\theta^D_nX}^{\overline\AA}\to0\), so
\(\mathsf P_{\theta^D_nX}^{\overline\AA}=0\).  Every \(Y\in D\) is
nonzero in \(\overline\AA\): otherwise \(1_Y\) would factor through
\(\add K\), making \(Y\) a summand of an object of \(\add K\), contrary to
\(Y\notin K\).  Thus \(\overline\AA(-,Y)\ne0\) for every \(Y\in D\), and
since \(\mathsf P_{\theta^D_nX}^{\overline\AA}=0\), it follows that
\(\theta^D_nX=0\).  By
Construction~\ref{con:factor-ladder}, every later term is also zero.
\end{proof}

\subsection{Characterizations of quotient and submodule closure}
\label{sec:factor-ladder-criterion}

\begin{construction}\label{con:reverse-factor-ladder}
The \emph{reverse factor ladder} of \(X\in D\) is defined by duality:
form the factor ladder of \(\DDual X\) over \(A^{\op}\), with omitted set
\(\{\,\DDual Y\mid Y\in D\,\}\), and apply \(\DDual\) to each term.  Its
terms are again integral
combinations of modules in \(D\).
\end{construction}

\begin{theorem}\label{thm:factor-ladder}
Let \(K\subseteq\ind A\), and put \(D=\ind A\setminus K\).  Then
the following hold.
\begin{enumerate}
\item The subcategory \(\add K\) is quotient-closed if and only if, for
every \(X\in D\), some
term \(\theta^D_nX\) contains an indecomposable
projective module belonging to \(D\).
\item The subcategory \(\add K\) is submodule-closed if and only if, for
every \(X\in D\), some
term of the reverse factor ladder of \(X\) contains an indecomposable
injective module belonging to \(D\).
\end{enumerate}
\end{theorem}

\begin{proof}
(1)
By Lemma~\ref{lem:indecomposable-quotient-test}, \(\add K\) is
quotient-closed if and only if
\[
 \frac{X}{\tr_{\add K}(X)}\ne0
 \qquad\text{for every }X\in D.
\]
Lemma~\ref{lem:trace-factor-hom} identifies this quotient with
\(\overline\AA(A,X)\).

Decomposing \(A\) into indecomposable projectives shows that
\(\overline\AA(A,X)\) is nonzero if and only if
\(\overline{\AA}(P,X)\ne0\) for some projective \(P\in D\), because a
projective belonging to \(K\) is zero in the factor category.

It remains to determine from the recursion when
\(\overline{\AA}(P,X)\) is nonzero.  Since the radical of
\(\overline{\AA}\) is nilpotent
(Proposition~\ref{prop:iyama-factor-category}(2)),
\(\overline{\AA}(P,X)\ne0\) if and only if
\[
 \bigl(J_{\overline{\AA}}^n/
 J_{\overline{\AA}}^{n+1}\bigr)(P,X)\ne0
\]
for some \(n\).  The radical of \(J_{\overline\AA}^n(-,X)\) is
\(J_{\overline\AA}^{n+1}(-,X)\), obtained by composing once
more with the categorical radical.  By
Proposition~\ref{prop:iyama-factor-category}(1), its projective cover is
\(\mathsf P_{\theta^D_nX}^{\overline\AA}\).  Hence the radical layer
\(J_{\overline\AA}^n/J_{\overline\AA}^{n+1}\) is the direct sum, with
multiplicities, of the simple functors indexed by the indecomposable summands
of \(\theta^D_nX\).
Consequently,
\[
 \bigl(J_{\overline{\AA}}^n/
 J_{\overline{\AA}}^{n+1}\bigr)(P,X)\ne0
 \quad\Longleftrightarrow\quad
 P\text{ is an indecomposable summand of }\theta^D_nX.
\]
Hence
\(\overline{\AA}(P,X)\ne0\) precisely when \(P\) occurs in some term of
the factor ladder starting at \(X\), which proves (1).

(2)
Apply (1) to \(A^{\op}\) and then apply \(\DDual\).
The factor ladder of \(\DDual X\) becomes the reverse factor ladder of
\(X\), projective modules become injective modules, and quotient closure
becomes submodule closure.  This gives the criterion in (2).
\end{proof}

\begin{example}\label{ex:factor-ladder-a2}
Let \(A\) be the path algebra of \(1\to2\), with
\(\ind A=\{P_2,P_1,S_1\}\) and almost split sequence
\(0\to P_2\to P_1\to S_1\to0\).  At \(S_1\), the factor ladder is
\(S_1,P_1,(P_2-P_2)_+=0\) for \(K=\varnothing\), so it contains the
omitted projective \(P_1\).  For \(K=\{P_1\}\), it is \(S_1,0\), so the
criterion detects that \(\add P_1\) is not quotient-closed.
\end{example}

\begin{corollary}\label{cor:ar-determination}
For a representation-finite algebra, the two families \(\LQ(A)\) and
\(\LS(A)\), their inclusion orders, and their size generating polynomials
are determined by the Auslander--Reiten quiver.
\end{corollary}

\begin{proof}
The arrows and translation determine the factor and reverse factor ladders.
Theorem~\ref{thm:factor-ladder} therefore determines which vertex subsets
give quotient-closed and submodule-closed subcategories; the vertex subsets
also determine inclusion and size.
\end{proof}

Consequently, Conjecture~\ref{conj:equidistribution} depends only on the
Auslander--Reiten quiver of the module category.

\section{Proof of equidistribution for sizes and cosizes at most four}
\label{sec:coefficients}

Let \(N=|\ind A|\). The \emph{cosize} of a subcategory \(\CC\) is
\(N-|\ind\CC|\). In this section we prove
Theorem~\ref{thm:size-four-intro}: the numbers
of quotient-closed and submodule-closed subcategories agree at every
size at most four and at every cosize at most four.

\subsection{Equidistribution in sizes at most three}\label{sec:sizes-three}

In this subsection we prove equidistribution in sizes at most three and give
an explicit formula in size three.

Recall that a module is \emph{torsionless} if it embeds in a finite direct
sum of projective modules.
For a finite-dimensional algebra \(B\), put
\[
 \mathcal K_{\quot}(B)=\ind\Gen\bigl(\DDual({}_B B)\bigr),\qquad
 \mathcal K_{\sub}(B)=\ind\Cogen(B_B).
\]
The set \(\mathcal K_{\sub}(B)\) consists of the indecomposable torsionless
\(B\)-modules.

\begin{theorem}[{\cite[Theorem~1.1 and
Corollary~2.1]{RingelTorsionless}}]\label{thm:ringel-torsionless-bijection}
Let \(B\) be a finite-dimensional algebra.
There is a bijection between the indecomposable torsionless \(B\)-modules and
the indecomposable torsionless \(B^{\op}\)-modules.
\end{theorem}

\begin{proposition}\label{prop:ringel-torsionless}
For every finite-dimensional algebra \(B\), there is a bijection
\[
 \mathcal K_{\sub}(B)\longrightarrow\mathcal K_{\quot}(B).
\]
Write \(d(B)=|\mathcal K_{\quot}(B)|=|\mathcal K_{\sub}(B)|\).
In particular, \(d(B)\) is finite when \(B\) is
representation-finite.
\end{proposition}

\begin{proof}
Compose the bijection in Theorem~\ref{thm:ringel-torsionless-bijection} with
the duality from torsionless \(B^{\op}\)-modules to factor modules of
injective \(B\)-modules.  The resulting map is a bijection
\(\mathcal K_{\sub}(B)\to\mathcal K_{\quot}(B)\).  If \(B\) is
representation-finite, \(\mathcal K_{\sub}(B)\) and
\(\mathcal K_{\quot}(B)\) are finite.
\end{proof}

\begin{definition}\label{def:faithful-subcategory}
An additive subcategory \(\CC\subseteq\mod B\) is \emph{faithful} if
\(\ann_B\CC=\bigcap_{X\in\CC}\ann_BX=0\).
\end{definition}

\begin{lemma}\label{lem:faithful-core}
Let \(B\) be a finite-dimensional algebra.
\begin{enumerate}
\item The set \(\mathcal K_{\quot}(B)\) is the unique minimal faithful
quotient-closed set.
\item The set \(\mathcal K_{\sub}(B)\) is the unique minimal faithful
submodule-closed set.
\end{enumerate}
\end{lemma}

\begin{proof}
(1)
Let \(\CC\) be faithful.  Since \(B\) is finite dimensional, there is an object
\(X\in\CC\) with \(\ann_BX=0\), and hence an embedding \(B\hookrightarrow
X^n\) for some \(n\).  If \(\CC\) is quotient-closed and \(E\) is injective,
choose a surjection \(B^m\twoheadrightarrow E\). It extends across
\(B^m\hookrightarrow X^{nm}\), and the extension remains surjective. Thus every
injective module belongs to \(\CC\).  Since \(\DDual({}_B B)\) is faithful,
\(\mathcal K_{\quot}(B)\) is the unique minimal faithful
quotient-closed set.

(2)
Let \(\CC\) be a faithful submodule-closed subcategory, and choose
\(X\in\CC\) with \(\ann_BX=0\).  The resulting embedding
\(B\hookrightarrow X^n\) gives \(B_B\in\CC\).  Since \(B_B\) is faithful,
\(\mathcal K_{\sub}(B)\) is the unique minimal faithful submodule-closed set.
\end{proof}

\begin{proposition}\label{prop:bottom-two}
There is a bijection between
\(\LQ^2(A)\) and \(\LS^2(A)\) which preserves annihilators.  More precisely,
both sets are in bijection with
\[
 \mathcal J_2(A)=
 \{\,I\lhd A\mid d(A/I)=2\,\}.
\]
In particular, \(|\LQ^2(A)|=|\LS^2(A)|\).
\end{proposition}

\begin{proof}
Let \(\CC\in\LQ^2(A)\), put \(I=\ann_A\CC\) and \(B=A/I\), and view \(\CC\)
as a faithful subcategory of \(\mod B\).  Lemma~\ref{lem:faithful-core}
gives \(\mathcal K_{\quot}(B)\subseteq\CC\).  This set cannot have only one
element.  Indeed, if \(\mathcal K_{\quot}(B)=\{E\}\), then all
indecomposable injective modules are isomorphic, so \(B\) has a unique
simple module.  A simple quotient \(S\) of \(E\) belongs to the
quotient-closed set \(\mathcal K_{\quot}(B)\), and hence \(S\simeq E\).
Thus the unique simple module is injective, every finite-length \(B\)-module
is semisimple, and \(B\) has only one indecomposable module.  This
contradicts the fact that \(\CC\) has two indecomposable objects.  Hence the
nonempty set \(\mathcal K_{\quot}(B)\subseteq\CC\) has two elements, so
\(d(B)=2\) and
\(\CC=\mathcal K_{\quot}(B)\).

Conversely, for every \(I\in\mathcal J_2(A)\), restriction of scalars along
\(A\to A/I\) makes \(\mathcal K_{\quot}(A/I)\) a quotient-closed subcategory
of \(\mod A\).  Proposition~\ref{prop:ringel-torsionless} shows that
\(\mathcal K_{\sub}(A/I)\) also has two elements, and restriction of scalars
gives a submodule-closed subcategory with annihilator \(I\).  The mutually
inverse bijections are
\[
 \CC\longmapsto
 \mathcal K_{\sub}(A/\ann_A\CC),
 \qquad
 \DD\longmapsto
 \mathcal K_{\quot}(A/\ann_A\DD).
\]
\end{proof}

For the remainder of this subsection, assume that \(A\) is basic,
representation-finite, and defined over the algebraically closed field
\(\kk\).  Every factor \(A/I\) is then basic.  To write the quotient and
submodule cases simultaneously, let the subscript \(*\) range over
\(\{\quot,\sub\}\).  For a representation-finite algebra \(B\), let
\(\mathcal F_*^j(B)\) be the faithful members of \(\LL_*^j(B)\), and put
\[
 c_*^j(B)=|\LL_*^j(B)|,\qquad
 f_*^j(B)=|\mathcal F_*^j(B)|.
\]
Put
\[
 F_*(B;q)=\sum_j f_*^j(B)q^j,\qquad
 R_*(B;q)=\sum_jc_*^j(B)q^j.
\]

\begin{lemma}\label{lem:annihilator-strata}
Let \(B\) be representation-finite.
For every \(j\geq0\), restriction of scalars gives a decomposition as a
disjoint union
\[
 \LL_*^j(B)\simeq
 \coprod_{I\lhd B}\mathcal F_*^j(B/I).
\]
Consequently,
\[
 c_*^j(B)=\sum_{I\lhd B}f_*^j(B/I).
\]
Only finitely many summands are nonzero.
\end{lemma}

\begin{proof}
For \(\CC\in\LL_*^j(B)\), put \(I=\ann_B\CC\).  Then \(\CC\) is faithful
as a subcategory of \(\mod(B/I)\).  Conversely, restriction of scalars sends
a faithful member of \(\LL_*^j(B/I)\) to a member of \(\LL_*^j(B)\) with
annihilator \(I\).  These operations are inverse and prove the displayed
decomposition.  Taking cardinalities gives the formula for \(c_*^j(B)\).
Since \(B\) is representation-finite, only finitely many summands can be
nonzero.
\end{proof}

\begin{lemma}\label{lem:faithful-descent}
Let \(\mathscr C\) be a class of representation-finite
finite-dimensional algebras which is closed under factor algebras.  If
\(\RQ(C;q)=\RS(C;q)\) for every \(C\in\mathscr C\), then the faithful
size generating polynomials of every \(C\in\mathscr C\) also agree.
\end{lemma}

\begin{proof}
Induct on \(\dim_\kk C\).  Separate the term indexed by the zero ideal in
Lemma~\ref{lem:annihilator-strata}:
\[
 F_*(C;q)=R_*(C;q)-
 \sum_{0\ne I\lhd C}F_*(C/I;q).
\]
The hypothesis identifies \(R_{\quot}(C;q)\) with \(R_{\sub}(C;q)\).  For
every nonzero ideal \(I\), induction identifies \(F_{\quot}(C/I;q)\) with
\(F_{\sub}(C/I;q)\).
Hence \(F_{\quot}(C;q)=F_{\sub}(C;q)\).
\end{proof}

Let \(Q_A\) be the Gabriel quiver of \(A\).

\begin{lemma}[{\cite[Lemma~III.2.12]{AssemSimsonSkowronski}}]
\label{lem:no-multiple-gabriel-arrows}
The quiver \(Q_A\) has at most one arrow with any fixed ordered pair of
endpoints.
\end{lemma}

\begin{lemma}\label{lem:two-element-faithful-core-classification}
Let \(B\) be a nonzero basic representation-finite algebra over \(\kk\).
Then \(d(B)=2\) if and only if \(B\) is isomorphic to one of
\[
 \kk\times\kk,\qquad \kk A_2,\qquad
 \kk[\varepsilon]/(\varepsilon^2).
\]
\end{lemma}

\begin{proof}
Suppose that \(d(B)=2\).  Since the
indecomposable projectives belong to \(\mathcal K_{\sub}(B)\), the
algebra \(B\) has at most two simple modules.  If it has two, every
indecomposable torsionless module is
projective, so every right ideal is projective and \(B\) is hereditary.
By \cite[Theorem~VII.1.7(b)]{AssemSimsonSkowronski}, each connected block
of \(B\) is the path algebra of its Gabriel quiver.  Since \(B\) is
representation-finite,
\cite[Theorem~VII.5.10(a)]{AssemSimsonSkowronski} shows that the
underlying graph of each connected block is Dynkin.  As \(B\) has two
simple modules, its Gabriel quiver therefore consists either of two
isolated vertices or of an orientation of \(A_2\).  Hence
\[
 B\cong\kk\times\kk\qquad\text{or}\qquad B\cong\kk A_2.
\]
If
\(B\) has one simple module,
Lemma~\ref{lem:no-multiple-gabriel-arrows}, applied to \(B\), shows that
its Gabriel quiver has at most one loop.  The bound-quiver presentation
theorem then gives
\(B\cong\kk[x]/(x^m)\)
\cite[Theorem~II.3.7]{AssemSimsonSkowronski}, and
\(\mathcal K_{\sub}(B)\) contains its
modules of lengths \(1,\ldots,m\).  Hence \(m\leq2\).  The case
\(m=1\) has \(d(B)=1\), so \(m=2\).
Conversely, each of the three displayed algebras has exactly two
indecomposable torsionless modules.
\end{proof}

\begin{remark}\label{rem:size-two-gabriel-formula}
Proposition~\ref{prop:bottom-two} has the explicit consequence
\[
 |\LQ^2(A)|=|\LS^2(A)|=
 \binom r2+|(Q_A)_1|,
\]
where \(r\) is the number of simple \(A\)-modules.
Indeed, choose a bound-quiver presentation
\(A\cong\kk Q_A/J\), where \(J\) is contained in the square of the arrow
ideal, by applying
\cite[Theorem~II.3.7]{AssemSimsonSkowronski} to the connected blocks.
An unordered pair of vertices gives a factor \(\kk\times\kk\) by
killing every other vertex and every arrow.  An arrow between distinct
vertices gives a factor \(\kk A_2\), and a loop gives a factor
\(\kk[\varepsilon]/(\varepsilon^2)\), by keeping only that arrow and its
endpoints and killing all other vertices, arrows, and paths of length at
least two.  These maps annihilate \(J\).  Conversely,
Lemma~\ref{lem:two-element-faithful-core-classification} and the vanishing
of the radical square in the three factors show that an ideal
\(I\in\mathcal J_2(A)\) is determined by its surviving vertices and, in
the cases \(A/I\simeq\kk A_2\) and
\(A/I\simeq\kk[\varepsilon]/(\varepsilon^2)\), its surviving arrow in the
radical modulo its square.
Lemma~\ref{lem:no-multiple-gabriel-arrows} makes this arrow unique.  Thus
\(\mathcal J_2(A)\) is in bijection with the disjoint union of the unordered
pairs of vertices and the arrows of \(Q_A\).
\end{remark}

\begin{proposition}\label{prop:bottom-three}
One has \(\lvert\LQ^3(A)\rvert=\lvert\LS^3(A)\rvert\), and this number is
\[
 \#\{\,I\lhd A\mid d(A/I)=3\,\}
 +\#\{\,I\lhd A\mid A/I\text{ is Morita equivalent to }\kk A_2\,\}.
\]
\end{proposition}

\begin{proof}
By Lemma~\ref{lem:annihilator-strata}, it suffices to prove that every
factor algebra \(B=A/I\) satisfies
\[
 f_{\quot}^3(B)=f_{\sub}^3(B)=
 \begin{cases}
  1 & \text{if \(d(B)=3\) or \(B\) is Morita equivalent to \(\kk A_2\),}\\
  0 & \text{otherwise.}
 \end{cases}
\]
Indeed, Lemma~\ref{lem:faithful-core} shows that every faithful
\(*\)-closed subcategory of \(\mod B\) contains \(\mathcal K_*(B)\), which
has \(d(B)\) elements.  Therefore a factor with \(d(B)>3\) contributes
nothing, whereas for \(d(B)=3\), the set \(\mathcal K_*(B)\) is the unique
faithful \(*\)-closed subcategory of size three.  The zero algebra also
contributes nothing.

It remains to consider the nonzero factors \(B\) with \(d(B)<3\).  If
\(d(B)=1\), then \(B\) has a unique indecomposable projective module.  Since
\(B\) is basic, \(B_B\) is indecomposable.  If \(\rad B\ne0\), an
indecomposable summand of \(\rad B\) is a torsionless module not isomorphic
to \(B_B\), contradicting \(d(B)=1\).  Hence \(\rad B=0\), and
\(B\cong\kk\).  If \(d(B)=2\),
Lemma~\ref{lem:two-element-faithful-core-classification} gives
\[
 B\cong\kk\times\kk,\qquad B\cong\kk A_2,
 \qquad\text{or}\qquad
 B\cong\kk[\varepsilon]/(\varepsilon^2).
\]
Among these four factors with \(0<d(B)<3\), only \(\kk A_2\) has three
indecomposable modules.  Its full module category is both the unique
faithful quotient-closed subcategory and the unique faithful
submodule-closed subcategory of size three.  This proves the displayed formula
for every factor algebra \(B\).  Summing it over the ideals of \(A\) by
Lemma~\ref{lem:annihilator-strata} gives the asserted equality and count.
\end{proof}

\subsection{Equidistribution in size four}\label{sec:size-four}

We now prove equidistribution in size four.

\begin{lemma}\label{lem:small-faithful-core}
Let \(B\ne0\) be a basic representation-finite algebra over \(\kk\).  If
\(d(B)\leq3\), then one of the following holds:
\begin{enumerate}
\item \(B\) is Nakayama;
\item \(B\) is the path algebra of
\(1\leftarrow2\to3\) or \(1\to2\leftarrow3\);
\item after replacing \(B\) by \(B^{\op}\) if necessary,
\[
 B=L_0=\kk Q/(x^2,xa)
 \quad\text{or}\quad
 B=L_1=\kk Q/(x^2),
 \qquad
 Q:\ \begin{tikzcd}[ampersand replacement=\&,column sep=large,baseline=-.5ex]
  1 \arrow[loop left,"x"]\arrow[r,"a"]\&2.
 \end{tikzcd}
\]
\end{enumerate}
\end{lemma}

\begin{proof}
The classification is proved in Appendix~\ref{app:small-core}.
\end{proof}

\begin{lemma}\label{lem:exceptional-faithful-four}
The following statements hold.
\begin{enumerate}
\item For \(B=\kk(1\leftarrow2\to3)\) or
\(B=\kk(1\to2\leftarrow3)\), one has
\[
 f_{\quot}^4(B)=f_{\sub}^4(B)=3.
\]
\item For \(e=0,1\), one has
\[
 f_{\quot}^4(L_e)=f_{\sub}^4(L_e)=2,
\]
and
\[
 f_{\quot}^4(L_e^{\op})=f_{\sub}^4(L_e^{\op})=2.
\]
\end{enumerate}
\end{lemma}

\begin{proof}
The calculation is given in Appendix~\ref{app:small-core}.
\end{proof}

\begin{lemma}\label{lem:faithful-four}
For every basic representation-finite algebra \(B\) over \(\kk\),
\[
 f_{\quot}^4(B)=f_{\sub}^4(B).
\]
\end{lemma}

\begin{proof}
If \(B=0\), then \(\mod B\) has no indecomposable objects, so
\(f_{\quot}^4(B)=f_{\sub}^4(B)=0\).  Assume that \(B\ne0\), and put
\(d=d(B)\).  By Lemma~\ref{lem:faithful-core}, every faithful
\(*\)-closed subcategory of \(\mod B\) contains \(\mathcal K_*(B)\), which
has \(d\) elements.  Hence, if \(d>4\), both counts vanish.  If \(d=4\),
then \(\mathcal K_{\quot}(B)\) and \(\mathcal K_{\sub}(B)\) are the unique
faithful subcategories of size four on their respective sides, so both
counts are one.

It remains to treat \(d\leq3\).  By
Lemma~\ref{lem:small-faithful-core}, the algebra \(B\) belongs to one of the
three classes listed there.  If \(B\) is Nakayama, then every factor algebra
\(C\) of \(B\) is Nakayama.  Theorem~\ref{thm:nakayama} gives
\(R_{\quot}(C;q)=R_{\sub}(C;q)\) for every such \(C\), and
Lemma~\ref{lem:faithful-descent}, applied to the class of factor algebras of
\(B\), gives
\(F_{\quot}(B;q)=F_{\sub}(B;q)\).
Comparing the coefficients of \(q^4\) gives the desired equality.  If \(B\)
is one of the algebras in parts~(2) or~(3) of
Lemma~\ref{lem:small-faithful-core}, the equality follows from the explicit
calculations in Lemma~\ref{lem:exceptional-faithful-four}.
\end{proof}

\begin{proposition}\label{prop:bottom-four}
Let \(A\) be representation-finite over the algebraically closed field
\(\kk\).  Then
\[
 |\LQ^4(A)|=|\LS^4(A)|.
\]
\end{proposition}

\begin{proof}
By Morita invariance, we may assume that \(A\) is basic.  Apply
Lemma~\ref{lem:faithful-four} to every factor in the annihilator
decomposition:
\[
 |\LQ^4(A)|
 =\sum_{I\lhd A}f_{\quot}^4(A/I)
 =\sum_{I\lhd A}f_{\sub}^4(A/I)
 =|\LS^4(A)|.
\]
\end{proof}

\subsection{Equidistribution in cosizes at most two}
\label{sec:cosize-two}

The cases of cosize zero and one are Lemma~\ref{lem:boundary-coefficients}.
At cosize two, we prove that the two arrow counts of
Proposition~\ref{prop:cofinite-two} agree; this is the case of cosize
two in Theorem~\ref{thm:size-four-intro}.

\begin{corollary}\label{cor:cosize-two}
Let \(A\) be representation-finite, and put
\(N=|\ind A|\).  Let \(r\) be the number of simple \(A\)-modules.  Then
the following hold.
\begin{enumerate}
\item \(\beta_{\quot}(A)=\beta_{\sub}(A)\).
\item If \(N\geq2\), then
\[
 |\LQ^{N-2}(A)|=|\LS^{N-2}(A)|
 =\binom r2+\beta_{\quot}(A).
\]
\end{enumerate}
\end{corollary}

\begin{proof}
(1)
Let \(\PP\) be the set of indecomposable projective modules, and let
\(\mathcal I\) be the set of indecomposable injective modules.  For
\(\UU,\VV\subseteq\ind A\), put
\[
 e(\UU,\VV)
 =\bigl|\{(X,Y)\in\UU\times\VV\mid
             \text{there is an arrow }X\to Y\}\bigr|.
\]
Lemma~\ref{lem:ar-mesh-correspondence} and the bijectivity of \(\tau\) give
a bijection from the ordered adjacent pairs ending at nonprojective modules
to the ordered adjacent pairs starting at noninjective modules.  Hence
\[
 e(\ind A,\ind A\setminus\PP)
 =e(\ind A\setminus\mathcal I,\ind A).
\]
The set \(\ind A\) is finite.  Subtracting both sides from
\(e(\ind A,\ind A)\) gives
\[
 e(\ind A,\PP)=e(\mathcal I,\ind A).
\]
For every pair \((P,Z)\) counted by \(\beta_{\quot}(A)\),
Lemma~\ref{lem:ar-mesh-correspondence} identifies the adjacent pair
\((P,Z)\) with \((\tau Z,P)\).  Therefore
\[
 \beta_{\quot}(A)=e(\ind A\setminus\mathcal I,\PP).
\]
For every noninjective \(Z\), write \(Z=\tau Y\) with \(Y\) nonprojective.
Lemma~\ref{lem:ar-mesh-correspondence} identifies the adjacent pair
\((Z,I)\) with \((I,Y)\).
Hence
\[
 \beta_{\sub}(A)=e(\mathcal I,\ind A\setminus\PP).
\]
Therefore
\[
\begin{aligned}
 e(\ind A,\PP)
 &=e(\mathcal I,\PP)+\beta_{\quot}(A),\\
 e(\mathcal I,\ind A)
 &=e(\mathcal I,\PP)+\beta_{\sub}(A).
\end{aligned}
\]
The equality \(e(\ind A,\PP)=e(\mathcal I,\ind A)\) now gives
\(\beta_{\quot}(A)=\beta_{\sub}(A)\).

(2)
Apply Proposition~\ref{prop:cofinite-two} and part~(1).
\end{proof}

\subsection{Equidistribution in cosizes three and four}
\label{sec:cosizes-three-four}

In this subsection we prove the equality at every cosize at most four,
using the closure test of Theorem~\ref{thm:factor-ladder}.  Put
\(N=|\ind A|\), and write
\(\Gamma:=\Gamma(\mod A)\).
For \(D\subseteq\ind A\), let \(\Gamma[D]\) denote the full subquiver
of \(\Gamma\) with vertex set \(D\).
Let \(\PP,\mathcal I\subseteq\ind A\) be the sets of indecomposable
projective and injective \(A\)-modules, respectively.

\begin{definition}\label{def:projectively-rooted}
A subset \(D\subseteq\ind A\) is \emph{projectively rooted} if, for every
\(X\in D\), there exists \(P\in\PP\cap D\) together with a path from \(P\)
to \(X\) in \(\Gamma[D]\).  It is \emph{injectively corooted} if, for
every \(X\in D\), there exists \(I\in\mathcal I\cap D\) together with a path
from \(X\) to \(I\) in \(\Gamma[D]\).  Paths of length zero are
allowed.
\end{definition}

For a vertex set \(W\), write \(W^-\)
and \(W^+\) for the sets of immediate predecessors and successors of
vertices of \(W\), computed in the whole quiver \(\Gamma\) rather
than in \(\Gamma[W]\); these sets may meet \(W\).

\begin{lemma}\label{lem:factor-ladder-support-path}
Let \(D\subseteq\ind A\) and \(X\in D\).  If an indecomposable module
\(Y\) occurs in \(\theta_n^D X\), then there is a path from \(Y\) to
\(X\) in \(\Gamma[D]\).
\end{lemma}

\begin{proof}
We argue by induction on \(n\).  For \(n=0\), one has \(Y=X\), and the
case \(n=1\) follows directly from the definition of \(\theta_DX\).  Now
suppose that \(n\geq2\).  In the factor-ladder recursion, the subtracted
\(\tau_D\)-term cannot create a positive coefficient.  Hence \(Y\) occurs
in \(\theta_D Z\) for some indecomposable \(Z\) occurring in
\(\theta_{n-1}^D X\).  The definition of \(\theta_D\) gives an arrow
\(Y\to Z\), while the induction hypothesis gives a path from \(Z\) to
\(X\).  Their concatenation is the required path.
\end{proof}

\begin{corollary}\label{cor:reverse-factor-ladder-support-path}
Let \(D\subseteq\ind A\) and \(X\in D\).  If \(Y\) occurs in the reverse
factor ladder starting at \(X\), then there is a path from \(X\) to \(Y\) in
\(\Gamma[D]\).
\end{corollary}

\begin{proof}
Apply Lemma~\ref{lem:factor-ladder-support-path} to \(A^{\op}\) and then apply
\(\DDual\) to the resulting path.
\end{proof}

\begin{lemma}\label{lem:rooted-balance}
Let \(A\) be representation-finite.
For every \(j\geq0\), the numbers of projectively rooted and injectively
corooted sets with \(j\) elements in the Auslander--Reiten quiver agree.
\end{lemma}

\begin{proof}
The inclusion--exclusion count is given in
Appendix~\ref{sec:four-ladders}.
\end{proof}

A projectively rooted set \(D\) need not have quotient-closed complement:
there may be \(X\in D\) such that no term \(\theta_n^D X\) of its factor
ladder (Construction~\ref{con:factor-ladder}) contains a
projective summand belonging to \(D\).  Conversely,
Lemma~\ref{lem:factor-ladder-support-path} shows that the factor-ladder
condition in Theorem~\ref{thm:factor-ladder} forces \(D\) to be
projectively rooted.  Thus, among projectively rooted sets, failure of the
factor-ladder condition is the only obstruction.
Theorem~\ref{thm:factor-ladder}(2) and
Corollary~\ref{cor:reverse-factor-ladder-support-path} give the analogous
description for submodule closure in terms of reverse factor ladders.  The
next lemma balances the two failure counts for at most four omitted modules.

\begin{lemma}\label{lem:four-vertex-ladder}
Let \(A\) be representation-finite over an algebraically closed field.
For every \(0\leq j\leq4\), the following two classes have the same
cardinality.
\begin{enumerate}
\item The projectively rooted \(j\)-element subsets \(D\subseteq\ind A\)
for which there exists \(X\in D\) such that \(\theta_n^D X\) has no summand
in \(\PP\cap D\) for every \(n\geq0\).
\item The injectively corooted \(j\)-element subsets \(D\subseteq\ind A\)
for which there exists \(X\in D\) such that no term of the reverse factor
ladder starting at \(X\) has a summand in \(\mathcal I\cap D\).
\end{enumerate}
\end{lemma}

\begin{proof}
The proof is given at the end of Appendix~\ref{sec:four-ladders}.
\end{proof}

\begin{proposition}\label{prop:top-four}
Let \(A\) be representation-finite over the algebraically closed field
\(\kk\), and put \(N=|\ind A|\).  Then, for
\(0\leq i\leq\min\{4,N\}\), one has
\[
 |\LQ^{N-i}(A)|=|\LS^{N-i}(A)|.
\]
\end{proposition}

\begin{proof}
For \(i=0,1\), the required equality follows from
Lemma~\ref{lem:boundary-coefficients}; for \(i=2\), it is
Corollary~\ref{cor:cosize-two}.

For \(i=3,4\), Theorem~\ref{thm:factor-ladder}(1) says that a complement
is quotient-closed exactly when, for every \(X\) in its omitted set \(D\),
some \(\theta_n^D X\) contains a projective summand belonging to \(D\).
Theorem~\ref{thm:factor-ladder}(2) says that the complement is
submodule-closed exactly when, for every \(X\in D\), the reverse factor
ladder of \(X\) contains an injective summand belonging to \(D\).
Lemma~\ref{lem:rooted-balance} gives
\[
 \#\{D\subseteq\ind A\mid |D|=i,\ D\text{ projectively rooted}\}
 =
 \#\{D\subseteq\ind A\mid |D|=i,\ D\text{ injectively corooted}\}.
\]
Lemma~\ref{lem:four-vertex-ladder} shows that the number of projectively rooted
\(D\) with \(|D|=i\) that do not satisfy the condition in
Theorem~\ref{thm:factor-ladder}(1) equals the number of injectively corooted
\(D\) with \(|D|=i\) that do not satisfy the condition in
Theorem~\ref{thm:factor-ladder}(2).  Subtracting these failure counts from the
two rooted counts proves
\(|\LQ^{N-i}(A)|=|\LS^{N-i}(A)|\).
\end{proof}

\section{Classification by a Bruhat interval and equidistribution for
representation-directed algebras}
\label{sec:representation-directed}

Recall that \(A\) is \emph{representation-directed} if it is
representation-finite and there is no cycle of nonzero nonisomorphisms
between indecomposable modules.  Throughout this section, assume that
\(A\) is representation-directed.  We prove Theorem~\ref{thm:directed-main},
which classifies the quotient-closed and submodule-closed subcategories and
establishes the equidistribution conjecture for \(A\).

\begin{definition}\label{def:hom-compatible-order}
An order \(X_1<\cdots<X_N\) on \(\ind A\) is \emph{Hom-compatible} if
\begin{equation}\label{eq:directed-hom-order-property}
 \Hom_A(X_i,X_j)\ne0,\quad i\ne j
 \quad\Longrightarrow\quad i<j.
\end{equation}
\end{definition}

\subsection{A mixed-coordinate criterion for quotient closure}

We first express quotient closure as the nonnegativity of mixed coordinates.

Directedness gives a Hom-compatible total order
\begin{equation}\label{eq:directed-order}
 \ind A=\{X_1<\cdots<X_N\}.
\end{equation}
Indeed, directedness makes the relation generated by nonzero
nonisomorphisms a partial order, so it has a linear extension.  Moreover,
\(\End_A(X_i)\) is a division algebra: a nonzero noninvertible
endomorphism would itself be a cycle of nonzero nonisomorphisms.  Since \(A\) is finite dimensional and the field is
algebraically closed, this gives
\(\End_A(X_i)=\kk\).

For \(M\in\mod A\) and \(X\in\ind A\), put
\[
 \mathsf P_M=\Hom_A(-,M),\qquad
 \mathsf S_X=\mathsf P_X/\rad_A(-,X).
\]
\begin{proposition}[{\cite[Example~A.2.10, Lemma~IV.6.5,
Corollary~IV.6.7, and Lemma~IV.6.8]{AssemSimsonSkowronski}}]
\label{prop:directed-functor-bases}
The category \(\mod(\mod A)\) of finitely presented contravariant functors
on \(\mod A\) has finite length, the \(\mathsf P_X\) are
its indecomposable projectives, the \(\mathsf S_X\) are its simple objects,
and
\[
 [\mathsf P_Y:\mathsf S_X]=\dim_\kk\Hom_A(X,Y).
\]
\end{proposition}

Let \(\mathcal B\) be a finite additive category with split idempotents.
\begin{construction}\label{con:cartan-homomorphism}
Write \([M]_{\rm sp}\) for the class of an object \(M\) in its split
Grothendieck group \(K_0^{\rm sp}(\mathcal B)\).  Yoneda identifies this
group with
\(K_0(\proj(\mod(\mathcal B)))\).  The inclusion of projective functors
induces the \emph{Cartan homomorphism}
\begin{equation}\label{eq:directed-cartan-map}
 \mathfrak c_{\mathcal B}:K_0^{\rm sp}(\mathcal B)
 \longrightarrow K_0(\mod(\mathcal B)),\qquad
 [M]_{\rm sp}\longmapsto[\mathcal B(-,M)].
\end{equation}
\end{construction}

\begin{corollary}\label{cor:directed-unitriangular-bases}
The Cartan homomorphism \(\mathfrak c_{\mod A}\) is an isomorphism.
The simple and representable classes are bases of
\(K_0(\mod(\mod A))\).  With respect to the basis
\(\bigl([X_j]_{\rm sp}\bigr)_j\) of the source and the basis
\(\bigl([\mathsf S_{X_i}]\bigr)_i\) of the target, the matrix of
\(\mathfrak c_{\mod A}\) is
\[
 \bigl(\dim_\kk\Hom_A(X_i,X_j)\bigr)_{i,j};
\]
it is upper unitriangular in the Hom-compatible order, with diagonal
entries \(\dim_\kk\End_A(X_i)=1\).
\end{corollary}

\begin{proof}
The classes \([X_i]_{\rm sp}\) form a basis of
\(K_0^{\rm sp}(\mod A)\), while finite length gives the basis of simple
classes in the target.  The multiplicity formula in
Proposition~\ref{prop:directed-functor-bases} identifies the matrix of
\(\mathfrak c_{\mod A}\) with
\(\bigl(\dim_\kk\Hom_A(X_i,X_j)\bigr)_{i,j}\).
The Hom-compatible order makes it upper triangular, and
\(\End_A(X_i)=\kk\) gives diagonal entries one.  Hence the Cartan
homomorphism is an isomorphism over \(\ZZ\), and its images, the
representable classes, also form a basis.
\end{proof}

Recall that \(\theta X\) is the domain of the minimal right almost split map
to \(X\): it is the middle term of the almost split sequence ending at a
nonprojective \(X\), while \(\theta P=\rad P\) for a projective \(P\).

\begin{lemma}[{\cite[Theorem~IV.6.11(a)]{AssemSimsonSkowronski}}]
\label{lem:directed-functor-ar-resolutions}
The functorial Auslander--Reiten resolutions give
\begin{equation}\label{eq:directed-mesh-class}
 [\mathsf S_X]=[\mathsf P_X]-[\mathsf P_{\theta X}]
                 +[\mathsf P_{\tau X}]
 \qquad(X\text{ nonprojective}),
\end{equation}
\[
 [\mathsf S_P]=[\mathsf P_P]-[\mathsf P_{\theta P}]
 \qquad(P\text{ projective}).
\]
\end{lemma}

Proposition~\ref{prop:directed-functor-bases} and
Lemma~\ref{lem:directed-functor-ar-resolutions} show, by induction on
length, that every object of \(\mod(\mod A)\) has a projective resolution
of length at most two.  The inverse of \(\mathfrak c_{\mod A}\) is
therefore the Euler characteristic of a projective resolution.  In
particular,
\begin{equation}\label{eq:directed-cartan-inverse}
 \mathfrak c_{\mod A}^{-1}([\mathsf S_X])
 =
 \begin{cases}
  [X]_{\rm sp}-[\theta X]_{\rm sp}+[\tau X]_{\rm sp},
     &X\text{ nonprojective},\\
  [X]_{\rm sp}-[\theta X]_{\rm sp},&X\text{ projective}.
 \end{cases}
\end{equation}

\begin{definition}\label{def:directed-mixed-coordinates}
Let \(D\subseteq[N]\) be a set of omitted positions, and put
\[
 \CC_D=\add\{X_i\mid i\notin D\}.
\]
Use the mixed basis
\[
 \mathsf B_D(i)=
 \begin{cases}
 [\mathsf P_{X_i}],&i\notin D,\\
 [\mathsf S_{X_i}],&i\in D.
 \end{cases}
\]
It is a basis because its transition matrix is unitriangular in the chosen
order.  Define integers \(\mu_D(i,j)\), for all \(i,j\), by
\begin{equation}\label{eq:directed-mixed-expansion}
 [\mathsf P_{X_j}]=\sum_i\mu_D(i,j)\mathsf B_D(i).
\end{equation}
Extend \(\mu_D(i,-)\) additively to arbitrary modules, so that
\(\mu_D(i,X_j)=\mu_D(i,j)\).
Since \(\mathsf B_D(j)=[\mathsf P_{X_j}]\) for \(j\notin D\),
\begin{equation}\label{eq:directed-projective-column}
 \mu_D(i,j)=\delta_{i,j}\qquad(j\notin D).
\end{equation}
The set \(D\) is \emph{nonnegative} if \(\mu_D(i,j)\geq0\) for every pair
\(i,j\).
\end{definition}

\begin{remark}
The case \(D=[N]\) explains the relation between the mixed coordinates and
ordinary Hom spaces.  In this case the mixed basis is the simple basis, and
Proposition~\ref{prop:directed-functor-bases} gives
\[
 [\mathsf P_{X_x}]
 =\sum_{a=1}^N\dim_\kk\Hom_A(X_a,X_x)[\mathsf S_{X_a}].
\]
Thus
\[
 \mu_{[N]}(a,x)=\dim_\kk\Hom_A(X_a,X_x).
\]
For a general subset \(D\), the mixed basis makes this replacement only at
the omitted positions: it uses \([\mathsf S_{X_a}]\) for \(a\in D\) and
uses \([\mathsf P_{X_a}]\) for \(a\notin D\).  The
Auslander--Reiten resolutions in
Lemma~\ref{lem:directed-functor-ar-resolutions} express each replaced simple
class as an alternating sum of representable classes, and therefore give a
recurrence for \(\mu_D\).

When \(\CC_D\) is quotient-closed, the trace construction in the proof of
Proposition~\ref{prop:directed-closure-effectivity} identifies the omitted
coordinates with ordinary Hom dimensions in the ideal quotient:
\[
 \mu_D(a,x)
 =\dim_\kk\bigl((\mod A)/[\CC_D]\bigr)(X_a,X_x)
 \qquad(a,x\in D).
\]
Lemma~\ref{lem:directed-principal-positive} proves the same equality from the
Coxeter-theoretic condition introduced below, before quotient closure is
known.  Its nonnegative right-hand side then gives the positivity required
by Proposition~\ref{prop:directed-closure-effectivity}.
\end{remark}

By additivity, \(D\) is nonnegative if and only if
\(\mu_D(i,M)\geq0\) for every \(i\in[N]\) and every \(M\in\mod A\).

\begin{proposition}\label{prop:directed-closure-effectivity}
The subcategory \(\CC_D\) is quotient-closed if and only if \(D\) is
nonnegative.
\end{proposition}

\begin{proof}
\((\Rightarrow)\):
Suppose first that \(\CC_D\) is quotient-closed.  Fix \(j\in[N]\), put
\(Y=X_j\), and let \(\tr_{\CC_D}(Y)\) be the trace of \(\CC_D\) in \(Y\).  It
belongs to \(\CC_D\), and there is an exact sequence
of functors
\[
 0\longrightarrow\mathsf P_{\tr_{\CC_D}(Y)}
 \longrightarrow\mathsf P_Y\longrightarrow F_Y\longrightarrow0.
\]
The functor \(F_Y\) vanishes on \(\CC_D\), so all its composition
factors are \(\mathsf S_{X_i}\) with \(i\in D\), and \([F_Y]\) is a
nonnegative combination of the \(\mathsf B_D(i)\) with \(i\in D\).
Since \(\tr_{\CC_D}(Y)\in\CC_D\), the functor \(F_Y\) is the inflation of the
representable functor \(\bigl((\mod A)/[\CC_D]\bigr)(-,Y)\).
Since \(\tr_{\CC_D}(Y)\in\CC_D\) is a direct sum of modules \(X_i\) with
\(i\notin D\), the class \([\mathsf P_{\tr_{\CC_D}(Y)}]\) is a nonnegative
combination of the \(\mathsf B_D(i)\) with \(i\notin D\).  Hence
\([\mathsf P_Y]=[\mathsf P_{\tr_{\CC_D}(Y)}]+[F_Y]\) has nonnegative
coefficients in the mixed basis.  The uniqueness of the expansion
\eqref{eq:directed-mixed-expansion} proves that \(\mu_D(i,j)\geq0\) for
every \(i\).  For \(a\in D\), evaluation at \(X_a\) also gives
\[
 \mu_D(a,j)
 =[F_Y:\mathsf S_{X_a}]
 =\dim_\kk\bigl((\mod A)/[\CC_D]\bigr)(X_a,Y).
\]
Since \(j\) was arbitrary, \(D\) is nonnegative.

\((\Leftarrow)\):
The proof that nonnegativity of \(D\) implies that \(\CC_D\) is
quotient-closed is given in Appendix~\ref{sec:directed-tau-closure}.
\end{proof}

Thus quotient closure has been reduced to the signs of the integers
\(\mu_D(i,j)\).  We next compute these integers from Auslander--Reiten meshes,
thereby expressing the categorical criterion in terms of a word.

\subsection{The Auslander--Reiten word and the mixed-coordinate recurrence}

We now translate the mixed-coordinate recurrence into Coxeter combinatorics.

\begin{construction}\label{constr:directed-ar-word}
Fix the Hom-compatible order \(X_1<\cdots<X_N\) in
\eqref{eq:directed-order}.  Define a finite simple graph \(\Delta_A\) as
follows.  Its vertices are the \(\tau\)-orbits in \(\ind A\), and two
distinct vertices are adjacent if an Auslander--Reiten arrow connects the
two \(\tau\)-orbits.  Write \(i\sim j\) when the vertices \(i\) and \(j\)
are adjacent.  Let \(W(\Delta_A)\) be the simply-laced Coxeter group
\[
 \left\langle s_i\ \middle|\
 s_i^2=1,\ (s_is_j)^3=1\text{ if }i\sim j,\
 (s_is_j)^2=1\text{ if }i\neq j\text{ and }i\not\sim j
 \right\rangle .
\]
For each \(d\in[N]\), let \(i_d\) be the vertex of \(\Delta_A\)
corresponding to the \(\tau\)-orbit of \(X_d\).  We call \(d\) a
\emph{position} of the word and \(i_d\) its \emph{label}.  We call
\begin{equation}\label{eq:directed-ar-word}
 \underline{w}_A:=s_{i_1}\cdots s_{i_N}
\end{equation}
the \emph{Auslander--Reiten word} in the simple generators of
\(W(\Delta_A)\) associated with the chosen order.
Let \(w_A\in W(\Delta_A)\) be the element represented by
\(\underline{w}_A\).
\end{construction}

In the Dynkin construction of Proposition~\ref{prop:ort-dynkin}, the
position matched with \(\tau^{-a}P_i\), where \(a\geq0\), receives the label \(i\).
Construction~\ref{constr:directed-ar-word} uses the same labeling by
\(\tau\)-orbits for a representation-directed algebra.  We next determine
how Auslander--Reiten arrows occur between two such labels.

\begin{lemma}\label{lem:directed-two-orbits}
Every \(\tau\)-orbit can be indexed as
\((U_0,\ldots,U_m)\) with \(U_0\) projective, \(U_m\) injective, and
\(U_{i+1}=\tau^{-1}U_i\) for \(0\leq i<m\).  For \(i<i'\), there is a
directed path \(U_i\rightsquigarrow U_{i'}\) in \(\Gamma(\mod A)\), but
there is no Auslander--Reiten arrow whose source and target belong to the
same \(\tau\)-orbit.

If two different orbits
\[
 (U_0,\ldots,U_m),\qquad(V_0,\ldots,V_n)
\]
are connected by an arrow, then, after possibly exchanging the two
orbits, all arrows between them occur in chains of the form
\begin{equation}\label{eq:directed-orbit-arrows}
 U_i\longrightarrow V_{i+c}\longrightarrow U_{i+1}
\end{equation}
for one fixed integer \(c\), whenever the displayed vertices exist.
\end{lemma}

\begin{proof}
See Appendix~\ref{sec:directed-orbit-geometry}.
\end{proof}

Lemma~\ref{lem:directed-two-orbits} allows the labeled word to recover each
almost split mesh.  We now define the corresponding positions; the precise
dictionary is given in Lemma~\ref{lem:directed-word-dictionary}.

\begin{construction}\label{constr:directed-ar-mesh-positions}
For the Auslander--Reiten word
\(\underline{w}_A=s_{i_1}\cdots s_{i_N}\) and a position
\(x\in[N]\), define
\[
 p(x):=\max\{\,y\in[N]\mid y<x,\ i_y=i_x\,\}
\]
when the set on the right is nonempty.  Define a subset
\(\mathcal W(x)\subseteq[N]\) by
\[
 \mathcal W(x):=
 \begin{cases}
  \{\,y\in[N]\mid p(x)<y<x\,\},&p(x)\text{ exists},\\
  \{\,y\in[N]\mid y<x\,\},&p(x)\text{ does not exist}.
 \end{cases}
\]
The set of \emph{middle positions at \(x\)} is the subset of \([N]\)
defined by
\begin{equation}\label{eq:directed-middle-positions}
 \operatorname{mid}(x):=
 \left\{
  \max\{\,y\in\mathcal W(x)\mid i_y=h\,\}
  \ \middle|\
  \begin{array}{c}
   h\sim i_x\text{ in }\Delta_A,\\[-2pt]
   \{\,y\in\mathcal W(x)\mid i_y=h\,\}\ne\varnothing
  \end{array}
 \right\}.
\end{equation}
\end{construction}

For \(D\subseteq[N]\), a position of \(\underline{w}_A\) is called
\emph{selected} if it belongs to \(D\), and \emph{unselected} otherwise.
Thus \(d\in D\) means simultaneously that the position \(d\) is selected
in the word and that the module \(X_d\) is omitted from \(\CC_D\).

\begin{lemma}\label{lem:directed-word-dictionary}
For every position \(x\), the following statements hold.
\begin{enumerate}
\item \(p(x)\) exists exactly when \(X_x\) is nonprojective, and then
\(X_{p(x)}=\tau X_x\);
\item the modules indexed by \(\operatorname{mid}(x)\) are exactly the
indecomposable summands of \(\theta X_x\), and each occurs with multiplicity
one;
\end{enumerate}
\end{lemma}

\begin{proof}
\((1)\)
The directed paths in Lemma~\ref{lem:directed-two-orbits} force the
Hom-compatible order to list
its modules from the projective end to the injective end.  Thus the first
occurrence is projective, and the preceding occurrence of the label of a
nonprojective \(X_x\) is \(\tau X_x\).

\((2)\)
Fix an orbit adjacent to that of \(X_x\).  The arrow family in
\eqref{eq:directed-orbit-arrows} shows that the unique arrow from this orbit
into \(X_x\), when it exists, starts at its last occurrence in
\((p(x),x)\), or at its last occurrence before \(x\) when \(X_x\) is
projective.  These sources are precisely the indecomposable summands of
\(\theta X_x\).  They are the
positions in \(\operatorname{mid}(x)\), and
Corollary~\ref{cor:ar-middle-terms-squarefree} gives multiplicity one.
\end{proof}

Consequently, the mesh identities for the inverse of the Cartan homomorphism
can be read using word positions alone.  Taking their coefficients in the
mixed basis gives the recurrence which will be compared with Coxeter root
reflections.

\begin{lemma}\label{lem:directed-coordinate-recurrence}
For every \(D\subseteq[N]\) and \(a,x\in[N]\), the mixed coordinates
satisfy
\begin{equation}\label{eq:directed-coordinate-recurrence}
\mu_D(a,x)=
\begin{cases}
\delta_{a,x},&x\notin D,\\
\displaystyle\delta_{a,x}
\mathbin{+}\displaystyle\sum_{y\in\operatorname{mid}(x)}\mu_D(a,y)
-\mu_D(a,p(x)),&x\in D,\ p(x)\text{ exists},\\
\displaystyle\delta_{a,x}
\mathbin{+}\displaystyle\sum_{y\in\operatorname{mid}(x)}\mu_D(a,y),
 &x\in D,\ p(x)\text{ does not exist}.
\end{cases}
\end{equation}
\end{lemma}

\begin{proof}
Rewrite \eqref{eq:directed-mesh-class} as
\[
 [\mathsf P_{X_x}]=[\mathsf S_{X_x}]+[\mathsf P_{\theta X_x}]
 -[\mathsf P_{\tau X_x}].
\]
When \(x\in D\), one has \([\mathsf S_{X_x}]=\mathsf B_D(x)\).  If
\(p(x)\) exists, Lemma~\ref{lem:directed-word-dictionary} identifies the
modules \(\theta X_x\) and \(\tau X_x\) with those indexed by
\(\operatorname{mid}(x)\) and \(p(x)\), respectively.  Taking the
coefficient of \(\mathsf B_D(a)\) gives the second case of
\eqref{eq:directed-coordinate-recurrence}.  If \(p(x)\) does not exist,
then \(X_x\) is projective, and the identity
\[
 [\mathsf S_{X_x}]
 =[\mathsf P_{X_x}]-[\mathsf P_{\theta X_x}]
\]
and Lemma~\ref{lem:directed-word-dictionary}(2) give the third case.
Finally, \eqref{eq:directed-projective-column} gives the first case when
\(x\notin D\).
\end{proof}

For comparison with suffixes of the word, we consider the selected positions
at and to the right of a fixed position.  For \(a\in[N]\), put
\[
 D[a]=\{a\}\cup\{d\in D\mid d>a\}.
\]
Thus the selected positions of \(D[a]\) are \(a\) and the positions of
\(D\) strictly to the right of \(a\).
Lemma~\ref{lem:directed-truncated-coordinates} shows that this truncation does
not change the mixed-coordinate row indexed by \(a\).

\begin{lemma}\label{lem:directed-truncated-coordinates}
For every \(D\subseteq[N]\) and \(a,x\in[N]\), one has
\[
 \mu_D(a,x)=\mu_{D[a]}(a,x).
\]
\end{lemma}

\begin{proof}
First let \(x<a\).  For \(E=D\) and \(E=D[a]\), induction on \(x\) using
Lemma~\ref{lem:directed-coordinate-recurrence} gives
\(\mu_E(a,x)=0\): the term \(\delta_{a,x}\) vanishes, and every other
term has second index smaller than \(x\).

At \(x=a\), the coefficient \(\mu_D(a,a)\) equals one by
\eqref{eq:directed-projective-column} when \(a\notin D\), and by
Lemma~\ref{lem:directed-coordinate-recurrence} and the vanishing just proved
when \(a\in D\).  Since \(a\in D[a]\), applying the recurrence and the same
vanishing gives \(\mu_{D[a]}(a,a)=1\).

Finally, suppose that \(x>a\).  Then
\(x\in D\) if and only if \(x\in D[a]\).  Hence the recurrence uses the
same case for \(D\) and \(D[a]\).  All
coefficients on its right-hand side have second index smaller than \(x\),
so induction on \(x\) proves the claimed equality.
\end{proof}

Thus each row may be computed from a selected suffix of the word.
Example~\ref{ex:directed-a3-negative} illustrates how a negative coordinate
records a failure of quotient closure.

\begin{example}\label{ex:directed-a3-negative}
Let \(A=\kk Q\), where
\[
 Q:\quad 1\longleftarrow2\longrightarrow3.
\]
Write \(M_{12},M_{23},M_{123}\) for the interval modules on the indicated
supports.  One Hom-compatible order is
\[
\begin{array}{c|cccccc}
x&1&2&3&4&5&6\\ \hline
X_x&S_1&S_3&M_{123}&M_{23}&M_{12}&S_2.
\end{array}
\]
The Auslander--Reiten quiver of \(A\) is
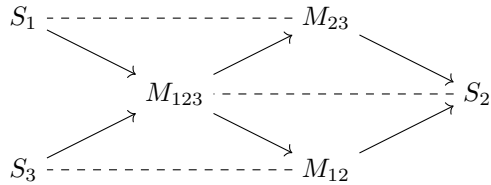
\begin{center}
\begin{tikzpicture}[xscale=2.0,yscale=1.0]
\node (S1) at (0,1) {\(S_1\)};
\node (S3) at (0,-1) {\(S_3\)};
\node (M123) at (1,0) {\(M_{123}\)};
\node (M23) at (2,1) {\(M_{23}\)};
\node (M12) at (2,-1) {\(M_{12}\)};
\node (S2) at (3,0) {\(S_2\)};
\draw[->] (S1) -- (M123);
\draw[->] (S3) -- (M123);
\draw[->] (M123) -- (M23);
\draw[->] (M123) -- (M12);
\draw[->] (M23) -- (S2);
\draw[->] (M12) -- (S2);
\draw[dashed] (M23) -- (S1);
\draw[dashed] (M12) -- (S3);
\draw[dashed] (S2) -- (M123);
\end{tikzpicture}
\end{center}
where a dashed line joins each nonprojective module to its translate.
The three \(\tau\)-orbits are
\[
 \{S_1,M_{23}\},\qquad
 \{M_{123},S_2\},\qquad
 \{S_3,M_{12}\}.
\]
The graph \(\Delta_A\) is the path \(1-2-3\), and its Auslander--Reiten word is
\[
 \underline{w}_A=s_1s_3s_2s_1s_3s_2.
\]

Take \(D=\{1,4\}\).  At position \(4\), one has \(p(4)=1\), and
\(\operatorname{mid}(4)=\{3\}\).  Therefore
\[
 \mu_D(1,4)=\delta_{1,4}+\mu_D(1,3)-\mu_D(1,1)=0+0-1=-1.
\]
This negative coordinate records an actual failure of quotient closure:
there is a surjection
\[
 X_3=M_{123}\twoheadrightarrow X_4=M_{23},
 \qquad X_3\in\CC_D,\quad X_4\notin\CC_D.
\]
\end{example}

For the implication from reduced selected suffixes to nonnegative mixed
coordinates, it remains to prove positivity when the selected subword is
reduced.  For a subset \(D\subseteq[N]\), the
\emph{\(D\)-subword} of \(\underline{w}_A\) is the subword on the positions in
\(D\), read in increasing order.  The precise positivity statement is
Lemma~\ref{lem:directed-principal-positive}.

\begin{lemma}\label{lem:directed-principal-positive}
If the \(D\)-subword of \(\underline{w}_A\) is reduced, then
\[
 \mu_D(a,x)
 =\dim_\kk\bigl((\mod A)/[\CC_D]\bigr)(X_a,X_x)
 \geq0
 \qquad(a,x\in D).
\]
\end{lemma}

\begin{proof}
The factor-category identification is proved in
Appendix~\ref{sec:directed-positivity}.
\end{proof}

\subsection{Nonnegative sets and canonical subwords}

We now identify nonnegative omitted sets with canonical reduced subwords of
\(\underline{w}_A\).

\begin{definition}\label{def:directed-canonical-position-sets}
Let \(\underline{w}=s_{j_1}\cdots s_{j_m}\) be a word in the simple generators
of a Coxeter group \(W\).  For \(u\in W\), define
\[
 \operatorname{RedPos}_{\underline{w}}(u):=
 \left\{\,
  \{d_1<\cdots<d_k\}\subseteq[m]
  \ \middle|\
  k=\ell(u),\quad
  s_{j_{d_1}}\cdots s_{j_{d_k}}=u
 \right\}.
\]
For two \(k\)-subsets \(D=\{d_1<\cdots<d_k\}\) and
\(E=\{e_1<\cdots<e_k\}\), define the \emph{lexicographic order} by
\(D<_{\rm lex}E\) if \(d_t<e_t\) at the smallest index where they differ,
and define the \emph{reverse lexicographic order} by
\(D<_{\rm rlex}E\) if \(d_t<e_t\) at the largest such index.  Whenever
\(\operatorname{RedPos}_{\underline{w}}(u)\) is nonempty, let
\(I_{\underline{w}}(u)\) be its lexicographically first member and let
\(I_{\underline{w}}^{\rm R}(u)\) be its
last member in reverse lexicographic order.  For \(D\subseteq[m]\), let
\(u_D\) be the element represented by the \(D\)-subword of
\(\underline{w}\).  The set \(D\) is \emph{\(\underline{w}\)-sorted} if
\(D=I_{\underline{w}}(u_D)\).
\end{definition}

For an element \(v\) of a Coxeter group \(W\), its left descent set is
\[
 \operatorname{Des}_L(v):=\{\,s_i\mid \ell(s_iv)<\ell(v)\,\}.
\]
For \(D\subseteq[m]\) and \(a\in[m]\), let \(v_a(D)\in W\) be the
element represented by the selected subword strictly to the right of
\(a\): if \(\{\,d\in D\mid d>a\,\}=\{d_1<\cdots<d_t\}\), then
\(v_a(D):=s_{j_{d_1}}\cdots s_{j_{d_t}}\).

\begin{proposition}[{\cite[Corollary~3.4]{Armstrong}}]
\label{prop:armstrong-sorting-criterion}
A subset \(D\subseteq[m]\) is \(\underline{w}\)-sorted if and only if the
following conditions hold.
\begin{enumerate}
\item The \(D\)-subword of \(\underline{w}\) is reduced.
\item For every \(a\in[m]\setminus D\), one has
\[
 s_{j_a}\notin\operatorname{Des}_L(v_a(D)).
\]
\end{enumerate}
\end{proposition}

\begin{corollary}\label{cor:directed-sorting-suffixes}
For \(D\subseteq[N]\), one has
\begin{equation}\label{eq:directed-armstrong}
 D\text{ is }\underline{w}_A\text{-sorted}
 \quad\Longleftrightarrow\quad
 \text{the }D[a]\text{-subword of }\underline{w}_A
 \text{ is reduced for every }a\in[N].
\end{equation}
\end{corollary}

\begin{proof}
Suppose first that \(D\) is \(\underline{w}_A\)-sorted.  By
Proposition~\ref{prop:armstrong-sorting-criterion}(1), the \(D\)-subword
of \(\underline{w}_A\) is reduced.  If \(a\in D\), the \(D[a]\)-subword of
\(\underline{w}_A\) is a suffix of this reduced word, and hence is reduced.
If \(a\notin D\), the subword of \(\underline{w}_A\) representing
\(v_a(D)\) is also a suffix of the \(D\)-subword of \(\underline{w}_A\)
and is therefore reduced.
Proposition~\ref{prop:armstrong-sorting-criterion}(2) gives
\(s_{i_a}\notin\operatorname{Des}_L(v_a(D))\).  Hence
\(\ell(s_{i_a}v_a(D))=\ell(v_a(D))+1\).
The \(D[a]\)-subword of \(\underline{w}_A\) is \(s_{i_a}\) followed by the
reduced subword representing \(v_a(D)\), and hence is reduced.

Conversely, suppose that the \(D[a]\)-subword of \(\underline{w}_A\) is reduced
for every \(a\in[N]\).  If \(D=\varnothing\), then the \(D\)-subword of
\(\underline{w}_A\) is the empty word and is reduced.  If
\(D\ne\varnothing\), take \(a=\min D\).  Then \(D[a]=D\), so the
\(D\)-subword of \(\underline{w}_A\) is again reduced.  Now let \(a\notin D\).
The \(D[a]\)-subword of \(\underline{w}_A\) is \(s_{i_a}\) followed by the
subword of \(\underline{w}_A\) representing \(v_a(D)\).  Since this
\(D[a]\)-subword is reduced, its suffix is reduced and
\(s_{i_a}\notin\operatorname{Des}_L(v_a(D))\).
Both conditions in Proposition~\ref{prop:armstrong-sorting-criterion}
hold, so \(D\) is \(\underline{w}_A\)-sorted.
\end{proof}

To prove that nonnegativity of \(D\) implies that \(D\) is
\(\underline{w}_A\)-sorted, we show that the first negative root forces a negative
mixed coordinate.  Lemma~\ref{lem:directed-first-cancellation} compares
the first sign change with the mixed-coordinate recurrence.

\begin{lemma}\label{lem:directed-first-cancellation}
Let \(D\subseteq[N]\) and \(a\in[N]\).  Define roots \(\beta_y\), for
\(a\leq y\leq N\), by \(\beta_a=\alpha_{i_a}\) and, for \(y>a\),
\[
 \beta_y=\begin{cases}
 s_{i_y}\beta_{y-1},&y\in D,\\
 \beta_{y-1},&y\notin D.
 \end{cases}
\]
Suppose that the root first becomes negative at the position \(x\),
which is then necessarily in \(D\).  For \(a<e<x\), put
\(c_e=(\beta_e)_{i_e}\).
Then the following statements hold.
\begin{enumerate}
\item One has \(c_e\geq0\) for every \(a<e<x\).
\item One has
\begin{equation}\label{eq:directed-first-cancellation}
 \mu_D(a,x)
 =-1-\sum_{\substack{a<e<x\\e\notin D}}c_e\mu_D(e,x).
\end{equation}
\end{enumerate}
\end{lemma}

\begin{proof}
See Appendix~\ref{sec:directed-cancellation}.
\end{proof}

\begin{lemma}[{\cite[Proposition~4.2.5]{BjornerBrenti}}]
\label{lem:coxeter-root-sign}
If \(v\) is represented by a reduced word, then \(s_iv\) is reduced if and
only if \(v^{-1}(\alpha_i)\) is a positive root.
\end{lemma}

\begin{theorem}\label{thm:directed-sorting}
For \(D\subseteq[N]\), the following are equivalent.
\begin{enumerate}
\item \(D\) is \(\underline{w}_A\)-sorted.
\item \(D\) is nonnegative.
\item \(\CC_D\) is quotient-closed.
\end{enumerate}
\end{theorem}

\begin{proof}
\((2)\Longleftrightarrow(3)\):
The equivalence is
Proposition~\ref{prop:directed-closure-effectivity}.

\((1)\Longrightarrow(2)\):
Suppose (1) holds.  By Corollary~\ref{cor:directed-sorting-suffixes}, the
\(D[a]\)-subword of \(\underline{w}_A\) is reduced for every \(a\).
Lemmas~\ref{lem:directed-truncated-coordinates}
and~\ref{lem:directed-principal-positive} give
\[
 \mu_D(a,x)=\mu_{D[a]}(a,x)\geq0
\]
when \(x\in D[a]\).  If \(x\notin D[a]\), then
Lemma~\ref{lem:directed-truncated-coordinates} and the first case of
Lemma~\ref{lem:directed-coordinate-recurrence} give
\[
 \mu_D(a,x)=\mu_{D[a]}(a,x)=\delta_{a,x}=0,
\]
where the last equality follows from \(a\in D[a]\).  Thus \(D\) is
nonnegative.

\((2)\Longrightarrow(1)\):
Assume that \(D\) is nonnegative but not sorted.  Choose the
largest \(a\) for which the \(D[a]\)-subword of \(\underline{w}_A\) is not
reduced.  Put \(S=\{\,d\in D\mid d>a\,\}\).
The set \(S\) is nonempty, since otherwise the \(D[a]\)-subword would be
the reduced word \(s_{i_a}\).  Write \(S=\{d_1<\cdots<d_t\}\), and put
\(v=s_{i_{d_1}}\cdots s_{i_{d_t}}\).
For \(b=d_1\), one has \(S=D[b]\).  Since \(b>a\), the maximality of
\(a\) shows that the \(D[b]\)-subword \(v\) is reduced.  On the other
hand, \(D[a]=\{a\}\cup S\), so the \(D[a]\)-subword is \(s_{i_a}v\),
which is not reduced by the choice of \(a\).
Lemma~\ref{lem:coxeter-root-sign} gives that
\(v^{-1}(\alpha_{i_a})\) is negative.  The root sequence defined in
Lemma~\ref{lem:directed-first-cancellation} satisfies
\[
 \beta_N=s_{i_{d_t}}\cdots s_{i_{d_1}}(\alpha_{i_a})
 =v^{-1}(\alpha_{i_a}),
\]
so \(\beta_N\) is negative.  Since \(\beta_a=\alpha_{i_a}\) is positive,
there is a first position \(x>a\) at which \(\beta_x\) is negative.
Lemma~\ref{lem:directed-first-cancellation}(1),(2) and the assumed
nonnegativity give
\[
 \mu_D(a,x)
 =-1-\sum_{\substack{a<e<x\\e\notin D}}c_e\mu_D(e,x)
 \leq-1,
\]
which contradicts the assumed nonnegativity.  Therefore \(D\) is sorted.
\end{proof}

Theorem~\ref{thm:directed-sorting} completes the translation from quotient
closure to canonical subwords.  It remains to index these subwords uniformly
and then obtain the submodule-closed classification by duality.

\subsection{Classification by a Bruhat interval and the duality bijection}

The Bruhat order provides the required indexing set because it records which
elements are represented by reduced subwords of a fixed reduced word.

\begin{definition}[{\cite[Theorem~2.2.2]{BjornerBrenti}}]
\label{def:bruhat-order}
For elements \(u,w\) of a Coxeter group, write \(u\leq w\) if every
reduced expression for \(w\) contains a reduced expression for \(u\) as a
subword.  We call this partial order the \emph{Bruhat order}.
\end{definition}

We first classify the quotient-closed subcategories.
Applying Theorem~\ref{thm:directed-sorting} to \(D=[N]\) shows that the word
\(\underline{w}_A\) is reduced.  The subword property for Bruhat order then indexes
all of its sorted supports.

\begin{proposition}\label{prop:directed-quotient-classification}
The following statements hold.
\begin{enumerate}
\item The word \(\underline{w}_A\) is reduced.
\item One has
\[
 \LQ(A)=\{\,\CC_{I_{\underline{w}_A}(u)}\mid u\leq w_A\,\}.
\]
\end{enumerate}
\end{proposition}

\begin{proof}
\((1)\)
For \(D=[N]\), the subcategory \(\CC_D=0\) is quotient-closed.
Theorem~\ref{thm:directed-sorting} therefore makes \([N]\) an
\(\underline{w}_A\)-sorted set.  A sorted position set is reduced, so the full
word \(\underline{w}_A\) is reduced.

\((2)\)
By (1), the word \(\underline{w}_A\) is a reduced expression for \(w_A\).
Definition~\ref{def:bruhat-order} therefore gives
\[
 u\leq w_A
 \quad\Longleftrightarrow\quad
 \text{a reduced expression for \(u\) occurs in }\underline{w}_A.
\]
Thus \(\operatorname{RedPos}_{\underline{w}_A}(u)\) is nonempty exactly when
\(u\leq w_A\).  For every such \(u\),
the position set \(I_{\underline{w}_A}(u)\) represents \(u\), so
Definition~\ref{def:directed-canonical-position-sets} makes it an
\(\underline{w}_A\)-sorted position set.  Conversely, if \(D\) is
\(\underline{w}_A\)-sorted, then
\(D=I_{\underline{w}_A}(u_D)\); in particular, \(D\)
indexes a reduced subword of \(\underline{w}_A\), so
Definition~\ref{def:bruhat-order} gives \(u_D\leq w_A\).  Therefore
\[
 \{\,D\subseteq[N]\mid D\text{ is }\underline{w}_A\text{-sorted}\,\}
 =
 \{\,I_{\underline{w}_A}(u)\mid u\leq w_A\,\}.
\]
By Theorem~\ref{thm:directed-sorting}, the quotient-closed subcategories are
exactly the \(\CC_D\) indexed by the position sets on the left.  Substituting
the right-hand description gives
\[
 \LQ(A)=
 \{\,\CC_{I_{\underline{w}_A}(u)}\mid u\leq w_A\,\}.
\]
\end{proof}

Thus the quotient-closed subcategories have been classified.  To obtain the
submodule-closed subcategories, we apply
Proposition~\ref{prop:directed-quotient-classification} to \(A^{\op}\).
Duality reverses the Hom-compatible order and exchanges the lexicographically
first supports with the last supports in reverse lexicographic order.

\begin{lemma}\label{lem:directed-opposite-word}
Put \(\iota(d)=N+1-d\).  Under the duality
\(\DDual\), the following statements hold.
\begin{enumerate}
\item The reversed order
\[
 \DDual X_N<\cdots<\DDual X_1
\]
is again Hom-compatible for \(A^{\op}\).
The graph \(\Delta_{A^{\op}}\) is naturally identified with \(\Delta_A\).
Under this identification, the Auslander--Reiten word for the reversed order
is \(\underline{w}_A^{\rm rev}\), which represents \(w_A^{-1}\).
\item For every \(u\in W(\Delta_A)\), one has
\begin{equation}\label{eq:directed-right-support}
 I_{\underline{w}_A}^{\rm R}(u)
 =\iota\bigl(I_{\underline{w}_A^{\rm rev}}(u^{-1})\bigr).
\end{equation}
\end{enumerate}
\end{lemma}

\begin{proof}
\((1)\)
Duality reverses nonzero maps and almost split sequences, and therefore
reverses each translation orbit and the chosen order.  The reversed word
represents \(w_A^{-1}\) because every simple reflection is an involution.

\((2)\)
Reversal of increasing position tuples turns
lexicographic minimum into the maximum in reverse lexicographic order, proving
\eqref{eq:directed-right-support}.
\end{proof}

The inversion in Lemma~\ref{lem:directed-opposite-word} preserves Bruhat
order:

\begin{lemma}[{\cite[Corollary~2.2.5]{BjornerBrenti}}]
\label{lem:bruhat-inversion}
For elements \(u,w\) of a Coxeter group,
\[
 u\leq w\quad\Longleftrightarrow\quad u^{-1}\leq w^{-1}.
\]
\end{lemma}

Combining Proposition~\ref{prop:directed-quotient-classification} with
Lemmas~\ref{lem:directed-opposite-word} and~\ref{lem:bruhat-inversion} gives
the representation-directed classification and the desired
equidistribution.

\begin{theorem}\label{thm:directed-main}
Let \(A\) be a representation-directed algebra over an algebraically
closed field, and let \(\underline{w}_A\) and \(w_A\) be given by
Construction~\ref{constr:directed-ar-word}.  The following statements hold.
\begin{enumerate}
\item The word \(\underline{w}_A\) is reduced, and
\begin{equation}
 \LQ(A)
 =\{\,\CC_{I_{\underline{w}_A}(u)}\mid u\leq w_A\,\}.
 \label{eq:directed-LQ-classification}
\end{equation}
\item One has
\begin{equation}
 \LS(A)
 =\{\,\CC_{I_{\underline{w}_A}^{\rm R}(u)}\mid u\leq w_A\,\}.
 \label{eq:directed-LS-classification}
\end{equation}
\item The map
\[
 \CC_{I_{\underline{w}_A}(u)}
 \longmapsto
 \CC_{I_{\underline{w}_A}^{\rm R}(u)}
 \qquad(u\leq w_A)
\]
is a size-preserving bijection from \(\LQ(A)\) to \(\LS(A)\).  Consequently,
\begin{equation}\label{eq:directed-profile}
 \RQ(A;q)=\RS(A;q)=
 \sum_{u\leq w_A}q^{N-\ell(u)}.
\end{equation}
\end{enumerate}
\end{theorem}

\begin{proof}
\((1)\)
This is Proposition~\ref{prop:directed-quotient-classification}(1),(2).

\((2)\)
Lemma~\ref{lem:bruhat-inversion} gives
\(u\leq w_A\) if and only if \(u^{-1}\leq w_A^{-1}\).
Duality also exchanges quotient closure for \(A^{\op}\) with submodule
closure for \(A\).  Applying
Proposition~\ref{prop:directed-quotient-classification}(2) to \(A^{\op}\),
and then using Lemma~\ref{lem:directed-opposite-word}(2), gives
\eqref{eq:directed-LS-classification}.

\((3)\)
The sets \(I_{\underline{w}_A}(u)\) and
\(I_{\underline{w}_A}^{\rm R}(u)\) both have
\(\ell(u)\) members.  Thus
\[
 \CC_{I_{\underline{w}_A}(u)}
 \longmapsto
 \CC_{I_{\underline{w}_A}^{\rm R}(u)}
\]
is a bijection preserving size.  Summing
\(q^{N-\ell(u)}\) proves \eqref{eq:directed-profile}.
\end{proof}

The remainder of the section compares the classification with the Dynkin
formula and illustrates it in examples.

\begin{remark}\label{rem:directed-recovers-ort}
For Dynkin path algebras, Theorem~\ref{thm:directed-main} gives precisely the
correspondence in Proposition~\ref{prop:ort-dynkin}.  Let
\(A=\kk Q\), where \(Q\) is Dynkin, put \(r=|Q_0|\), and order the
indecomposable modules as
\[
 P_1,\ldots,P_r,\tau^{-1}P_1,\ldots,\tau^{-1}P_r,\ldots,
\]
with zero modules omitted, using the numbering in
Proposition~\ref{prop:ort-dynkin}.  This order is Hom-compatible.  Its
\(\tau\)-orbits are indexed by \(Q_0\), the graph \(\Delta_A\) is the
underlying Dynkin graph of \(Q\), and \(\underline{w}_A\) is obtained from
\(c^\infty=(s_1\cdots s_r)^\infty\) by omitting the positions matched with
zero modules.  Gabriel's theorem
\cite[Theorem~VII.5.10(b),(c)]{AssemSimsonSkowronski} and the standard root
formula for length \cite[Proposition~4.4.4]{BjornerBrenti} give
\[
 N=|\Phi_Q^+|=\ell(w_0),
\]
where \(w_0\) is the longest element of \(W_Q\).  Since
Theorem~\ref{thm:directed-main}(1) proves that \(\underline{w}_A\) is
reduced of length
\(N\), it represents \(w_0\).

We compare the two canonical subwords.  If a position of \(c^\infty\)
matched with \(\tau^{-a}P_i=0\) for some \(a\) is omitted from
\(\underline{w}_A\), then no later
occurrence of \(s_i\) belongs to \(\underline{w}_A\).  Hence the subword of
\(\underline{w}_A\) in the positions strictly after this omitted position represents
an element of the standard parabolic subgroup generated by the \(s_j\) with
\(j\ne i\), and \(s_i\) is not a left descent of that element.  The
recognition criterion in
\cite[Corollary~3.4 and Section~7]{Armstrong} therefore shows that
\(\underline{w}_A\) is the lexicographically first reduced subword of \(c^\infty\)
representing \(w_0\).  Since \(w_0\) is the longest element, Armstrong's
comparison result \cite[Theorem~4.2 and Section~7]{Armstrong} shows that, for
every \(u\in W_Q\), the lexicographically first reduced subword of
\(c^\infty\) representing \(u\) is contained in \(\underline{w}_A\).  After the
zero positions are deleted, this
subword is \(I_{\underline{w}_A}(u)\).
Consequently, the quotient-closed subcategory in
\eqref{eq:directed-LQ-classification} is
\[
 \add\bigl(\ind\kk Q\setminus\XX(u)\bigr),
\]
with \(\XX(u)\) as in Proposition~\ref{prop:ort-dynkin}.  Hence the
representation-directed classification gives another proof of the Dynkin
specialization of the Oppermann--Reiten--Thomas theorem.
\end{remark}

Example~\ref{ex:directed-affine-d4} shows that the finite graph \(\Delta_A\)
can have an infinite Coxeter group.

\begin{example}\label{ex:directed-affine-d4}
Let \(A=\kk Q/\rad^2 \kk Q\), where
\[
 Q:\quad 1\longrightarrow2\longrightarrow3\longrightarrow4\longrightarrow5.
\]
This Nakayama algebra has nine indecomposable modules: the simples
\(S_1,\ldots,S_5\) and the modules \(P_1,\ldots,P_4\), where \(P_i\) is
uniserial of length two with top \(S_i\) and socle \(S_{i+1}\), and is
both projective and injective.
The Auslander--Reiten quiver is the chain
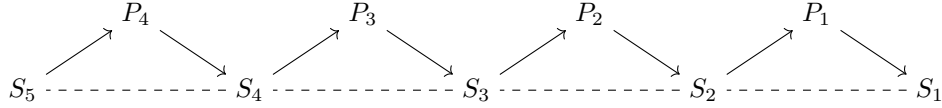
\begin{center}
\begin{tikzpicture}[xscale=1.5,yscale=1.0]
\node (S5) at (0,0) {\(S_5\)};
\node (P4) at (1,1) {\(P_4\)};
\node (S4) at (2,0) {\(S_4\)};
\node (P3) at (3,1) {\(P_3\)};
\node (S3) at (4,0) {\(S_3\)};
\node (P2) at (5,1) {\(P_2\)};
\node (S2) at (6,0) {\(S_2\)};
\node (P1) at (7,1) {\(P_1\)};
\node (S1) at (8,0) {\(S_1\)};
\draw[->] (S5) -- (P4);
\draw[->] (P4) -- (S4);
\draw[->] (S4) -- (P3);
\draw[->] (P3) -- (S3);
\draw[->] (S3) -- (P2);
\draw[->] (P2) -- (S2);
\draw[->] (S2) -- (P1);
\draw[->] (P1) -- (S1);
\draw[dashed] (S4) -- (S5);
\draw[dashed] (S3) -- (S4);
\draw[dashed] (S2) -- (S3);
\draw[dashed] (S1) -- (S2);
\end{tikzpicture}
\end{center}
where a dashed line joins each nonprojective module to its translate,
coming from the almost split sequences
\(0\to S_{i+1}\to P_i\to S_i\to0\).  The order from left to right in
the displayed chain is Hom-compatible, and \(A\) is
representation-directed.

There are five \(\tau\)-orbits: the orbit
\(\{S_1,\ldots,S_5\}\), labeled \(0\), and the four singletons
\(\{P_i\}\), labeled \(i\).  Construction~\ref{constr:directed-ar-word}
gives the star \(\Delta_A\) with center \(0\) and leaves \(1,2,3,4\), that
is, the affine Dynkin diagram of type \(\widetilde{D}_4\).  Thus
\(W(\Delta_A)\) is an infinite affine Weyl group.  Reading the orbit labels
along the chain gives
\[
 \underline{w}_A=s_0s_4s_0s_3s_0s_2s_0s_1s_0.
\]
Theorem~\ref{thm:directed-main}(1) shows that this word is reduced.  Its lower
Bruhat interval has \(162\) elements, and
\[
 \sum_{u\leq w_A}q^{9-\ell(u)}
 =[2]_q[3]_q^4,
\]
so
\[
 \RQ(A;q)=\RS(A;q)=[2]_q[3]_q^4.
\]
\end{example}

Example~\ref{ex:directed-a3-supports} displays the two canonical supports in
the preceding \(A_3\) example.

\begin{example}\label{ex:directed-a3-supports}
Continue with the quiver \(1\leftarrow2\to3\) in
Example~\ref{ex:directed-a3-negative}.  For \(u=s_1s_3\), the two
canonical position sets in \(\underline{w}_A\) are
\[
 I_{\underline{w}_A}(u)=\{1,2\},\qquad
 I_{\underline{w}_A}^{\rm R}(u)=\{4,5\}.
\]
The quotient-closed subcategory indexed by \(I_{\underline{w}_A}(u)\) and the
submodule-closed subcategory indexed by
\(I_{\underline{w}_A}^{\rm R}(u)\) are
\[
\begin{aligned}
 \CC_{\{1,2\}}&=\add\{M_{123},M_{23},M_{12},S_2\},\\
 \CC_{\{4,5\}}&=\add\{S_1,S_3,M_{123},S_2\}.
\end{aligned}
\]
Both contain four indecomposable modules, and
Theorem~\ref{thm:directed-main}(3) maps \(\CC_{\{1,2\}}\) to
\(\CC_{\{4,5\}}\).
\end{example}

\appendix

\section{Proofs for representation-directed algebras}
\label{sec:directed-technical}

This appendix contains the technical proofs deferred from
Section~\ref{sec:representation-directed}.  They are arranged in the order in
which their results are used.

\subsection{Bases of Grothendieck groups from exact
\texorpdfstring{\(\tau\)}{tau}-sequences}
\label{sec:directed-mesh-bases}

The two arguments with quotient categories use the following calculation.

\begin{lemma}\label{lem:directed-mesh-basis}
Let \(\Lambda\) be a finite-dimensional \(\kk\)-algebra, and let
\(\mathcal B\) be a \(\kk\)-linear category which is \(\kk\)-linearly
equivalent to
\(\proj\Lambda\).  Suppose that
\(\End_{\mathcal B}(X)=\kk\) for every indecomposable \(X\).  Suppose also
that, for every indecomposable \(X\), there is an exact sequence
\[
 0\longrightarrow\mathcal B(-,\tau_{\mathcal B}X)
 \longrightarrow\mathcal B(-,\theta_{\mathcal B}X)
 \longrightarrow\mathcal B(-,X)
 \longrightarrow S_X\longrightarrow0,
\]
where the first term may be zero and
\(S_X=\mathcal B(-,X)/J_{\mathcal B}(-,X)\).  Put
\[
 \kappa_X=[X]_{\rm sp}-[\theta_{\mathcal B}X]_{\rm sp}
                   +[\tau_{\mathcal B}X]_{\rm sp}.
\]
Then the following statements hold.
\begin{enumerate}
\item The classes \(\kappa_X\) form a basis of
\(K_0^{\rm sp}(\mathcal B)\).
\item For every \(M\in\mathcal B\), one has
\begin{equation}\label{eq:directed-mesh-basis-expansion}
 [M]_{\rm sp}
 =\sum_{X\in\ind\mathcal B}
   \dim_\kk\mathcal B(X,M)\,\kappa_X.
\end{equation}
Consequently, every additive function \(f\) on \(\mathcal B\) satisfies
\begin{equation}\label{eq:directed-mesh-function-expansion}
 f(M)=\sum_{X\in\ind\mathcal B}
 \dim_\kk\mathcal B(X,M)f(\kappa_X).
\end{equation}
\end{enumerate}
\end{lemma}

\begin{proof}
\((1)\)
For each indecomposable \(X\), the assumed exact sequence gives
\[
 \mathfrak c_{\mathcal B}(\kappa_X)=[S_X]
 \quad\text{in }K_0(\mod(\mathcal B)).
\]
The simple classes \([S_X]\) form a basis of this Grothendieck group.
Both \(K_0^{\rm sp}(\mathcal B)\) and
\(K_0(\mod(\mathcal B))\) are free of rank
\(\lvert\ind\mathcal B\rvert\).  Since the Cartan homomorphism
\(\mathfrak c_{\mathcal B}\) from
Construction~\ref{con:cartan-homomorphism} maps the \(\kappa_X\) to the
simple basis, it is an isomorphism and the \(\kappa_X\) form a basis.

\((2)\)
Evaluation at an indecomposable \(X\) is exact, and
\[
 S_Y(X)=
 \begin{cases}
  \kk,&X\simeq Y,\\
  0,&X\not\simeq Y.
 \end{cases}
\]
Therefore, in \(K_0(\mod(\mathcal B))\), one has
\[
 [\mathcal B(-,M)]
 =\sum_{X\in\ind\mathcal B}
 \dim_\kk\mathcal B(X,M)[S_X].
\]
Applying \(\mathfrak c_{\mathcal B}^{-1}\) proves
\eqref{eq:directed-mesh-basis-expansion}.  Applying \(f\) to this identity
gives \eqref{eq:directed-mesh-function-expansion}.
\end{proof}

\begin{definition}[{\cite[Section~1.3.1]{IyamaRealization}}]
Let \(\mathcal B\) be a \(\tau\)-category.  A \emph{right additive
function} is a map \(f:\ind\mathcal B\to\NN_{>0}\) such that, after
extending \(f\) additively and setting \(f(0)=0\),
\[
 f(X)-f(\theta_{\mathcal B}X)
 +f(\tau_{\mathcal B}X)\geq0
\]
for every indecomposable \(X\), with equality whenever
\(\tau_{\mathcal B}X\ne0\).
\end{definition}

\begin{corollary}\label{cor:directed-right-additive-strictness}
Fix \(D\subseteq[N]\), and put
\(\overline{\AA}_D=(\mod A)/[\CC_D]\).
Suppose that \(\overline{\AA}_D\) admits a right additive function.  Then
\(\overline{\AA}_D\) is a strict \(\tau\)-category.  The classes
\[
 [X]_{\rm sp}-[\theta_DX]_{\rm sp}+[\tau_DX]_{\rm sp}
 \qquad(X\in\ind\overline{\AA}_D)
\]
form a basis of \(K_0^{\rm sp}(\overline{\AA}_D)\), and
\begin{equation}\label{eq:directed-quotient-class-expansion}
 [M]_{\rm sp}
 =\sum_{X\in\ind\overline{\AA}_D}
 \dim_\kk\overline{\AA}_D(X,M)
 \bigl([X]_{\rm sp}-[\theta_DX]_{\rm sp}+[\tau_DX]_{\rm sp}\bigr).
\end{equation}
Consequently, every additive function \(f\) on \(\overline{\AA}_D\)
satisfies
\begin{equation}\label{eq:directed-quotient-function-expansion}
 f(M)=\sum_{X\in\ind\overline{\AA}_D}
 \dim_\kk\overline{\AA}_D(X,M)
 \bigl(f(X)-f(\theta_DX)+f(\tau_DX)\bigr).
\end{equation}
Here \(f(0)=0\).
\end{corollary}

\begin{proof}
Proposition~\ref{prop:iyama-factor-category} shows that
\(\overline{\AA}_D\) is a \(\tau\)-category.  It is equivalent to
the category of projective modules over a finite-dimensional algebra.  For
every \(x\in D\), the identity of \(X_x\) does not factor through
\(\CC_D\), and \(\End_A(X_x)=\kk\).  Hence
\(\End_{\overline{\AA}_D}(X_x)=\kk\).
Iyama's strictness theorem
\cite[Theorem~2.1, (4)\(\Rightarrow\)(3)]{IyamaRealization} shows that
\(\overline{\AA}_D\) is strict.  Consequently, the right
\(\tau\)-sequence ending at each indecomposable \(X\) induces an exact
sequence
\[
 0\longrightarrow\overline{\AA}_D(-,\tau_DX)
 \longrightarrow\overline{\AA}_D(-,\theta_DX)
 \longrightarrow\overline{\AA}_D(-,X)
 \longrightarrow S_X\longrightarrow0.
\]
These sequences satisfy the hypotheses of
Lemma~\ref{lem:directed-mesh-basis}, which gives the basis,
\eqref{eq:directed-quotient-class-expansion}, and
\eqref{eq:directed-quotient-function-expansion}.
\end{proof}

\subsection{Nonnegativity and quotient closure via a quotient
\texorpdfstring{\(\tau\)}{tau}-category}
\label{sec:directed-tau-closure}

We prove the implication \((\Leftarrow)\) in
Proposition~\ref{prop:directed-closure-effectivity}.  The proof has three
steps.  First, we define from the nonnegative mixed coordinates an additive
function \(\lambda\) which is positive on every omitted
indecomposable module.  Second, the Auslander--Reiten functor resolutions
determine the values
\(\lambda(X)-\lambda(\theta_DX)+\lambda(\tau_DX)\) on the quotient
\(\tau\)-category \((\mod A)/[\CC_D]\).  Finally,
Corollary~\ref{cor:directed-right-additive-strictness} identifies \(\lambda\)
with the dimension of a Hom space from the
indecomposable projective \(A\)-modules belonging to \(D\).  Positivity of
these Hom dimensions excludes omitted modules as quotients of objects of
\(\CC_D\), and hence proves that \(\CC_D\) is quotient-closed.

Throughout this subsection, \([F]\) denotes the class of a functor in
\(K_0(\mod(\mod A))\), whereas \([M]_{\rm sp}\) denotes the class of a
module in the split Grothendieck group \(K_0^{\rm sp}(\mod A)\).  For
modules \(X,Y\), abbreviate
\([X,Y]=\dim_\kk\Hom_A(X,Y)\).

We use the following consequence of directedness.

\begin{lemma}\label{lem:directed-hom-rigidity}
Let \(Y\) be indecomposable and let \(T\) have no direct summand isomorphic
to \(Y\).  Suppose that
\[
 [X,T]\leq[X,Y]
 \qquad\text{for every indecomposable \(X\)}.
\]
Then there is an indecomposable projective \(P\) such that
\[
 [P,T]<[P,Y].
\]
\end{lemma}

\begin{proof}
Suppose, to the contrary, that \([P,T]=[P,Y]\) for every indecomposable
projective \(P\).
Put \(d_X=[X,Y]-[X,T]\).  The multiplicity formula in
Proposition~\ref{prop:directed-functor-bases} gives the following identity
in \(K_0(\mod(\mod A))\):
\[
 [\mathsf P_Y]-[\mathsf P_T]
 =\sum_{X\in\ind A}d_X[\mathsf S_X].
\]
The hypothesis gives \(d_X\geq0\) for every \(X\), and the contrary
assumption gives \(d_P=0\) for every indecomposable projective \(P\).  Applying
\(\mathfrak c_{\mod A}^{-1}\) and using
\eqref{eq:directed-cartan-inverse} gives, in \(K_0^{\rm sp}(\mod A)\),
\begin{equation}\label{eq:directed-hom-rigidity-split}
 [Y]_{\rm sp}-[T]_{\rm sp}
 =\sum_{X\text{ nonprojective}}d_X
   \bigl([X]_{\rm sp}-[\theta X]_{\rm sp}
                         +[\tau X]_{\rm sp}\bigr).
\end{equation}
For every indecomposable \(Z\), applying \(\Hom_A(-,Z)\) to the almost
split sequence ending at \(X\) gives
\[
 [X,Z]-[\theta X,Z]+[\tau X,Z]
 =\dim_\kk\operatorname{Coker}
   \bigl(\Hom_A(\theta X,Z)\longrightarrow\Hom_A(\tau X,Z)\bigr).
\]
The left almost split property says that every non-section
\(\tau X\to Z\) factors through \(\tau X\to\theta X\).  Thus this cokernel
is zero if \(Z\not\simeq\tau X\), while for \(Z\simeq\tau X\) it is
\(\End_A(\tau X)/\rad\End_A(\tau X)=\kk\).  Hence
\[
 [X,Z]-[\theta X,Z]+[\tau X,Z]
 =
 \begin{cases}
  1,&Z\simeq\tau X,\\
  0,&Z\not\simeq\tau X.
 \end{cases}
\]
Applying \([-,Z]\) to \eqref{eq:directed-hom-rigidity-split} gives
\[
 [Y,Z]-[T,Z]
 =\sum_{\substack{X\text{ nonprojective}\\\tau X\simeq Z}}d_X
 \geq0.
\]

The module \(T\) is nonzero: otherwise equality at all indecomposable
projectives would give \([A,Y]=0\).  Choose an indecomposable summand \(Z\) of
\(T\).  Then
\[
 [Z,Y]\geq[Z,T]>0,
 \qquad
 [Y,Z]\geq[T,Z]>0.
\]
Thus
\[
 \Hom_A(Z,Y)\ne0
 \qquad\text{and}\qquad
 \Hom_A(Y,Z)\ne0.
\]
Since \(Z\not\simeq Y\), every morphism occurring here is a
nonisomorphism, contrary to directedness.
\end{proof}

\begin{proof}[Proof of \((\Leftarrow)\) in
Proposition~\ref{prop:directed-closure-effectivity}]
Suppose that \(D\) is nonnegative.  By additivity, all coefficients
\(\mu_D(i,M)\) are nonnegative.  Put
\[
 K=[N]\setminus D,\qquad
 \mathcal D=\{X_d\mid d\in D\},\qquad
 \mathcal P_D=\{\,p\in D\mid X_p\text{ is projective in }\mod A\,\}.
\]
Put
\[
 \overline{\AA}_D=(\mod A)/[\CC_D],\qquad
 \Pi_D=\bigoplus_{p\in\mathcal P_D}X_p,
\]
and define \(h(M)=\dim_\kk\overline{\AA}_D(\Pi_D,M)\).
Since \(\Pi_D\) is a summand of the regular module \(A\), the inequality
\(h(Y)>0\) implies \(\overline{\AA}_D(A,Y)\ne0\).
Lemma~\ref{lem:trace-factor-hom} then gives
\(Y/\tr_{\CC_D}(Y)\ne0\), and
Lemma~\ref{lem:indecomposable-quotient-test} excludes \(Y\) as a quotient
of an object of \(\CC_D\).  It is therefore enough to prove
\(h(Y)>0\) for every \(Y\in\mathcal D\).  We carry out the argument in
three steps.

We first obtain positivity for a coordinate function before the
\(\tau\)-sequences in \(\overline{\AA}_D\) are known to be exact.
For a module \(M\), define
\begin{equation}\label{eq:directed-boundary-function}
 \lambda(M)=\sum_{p\in\mathcal P_D}\mu_D(p,M).
\end{equation}
This is an additive function of \(M\), and hence extends to a homomorphism
on \(K_0^{\rm sp}(\mod A)\).  Moreover,
\eqref{eq:directed-projective-column} gives
\begin{equation}\label{eq:directed-boundary-vanishing}
 \lambda(X_a)=0\qquad(a\in K).
\end{equation}

\medskip
\noindent\textbf{Step 1: Positivity of the coordinate function.}
For every \(Y\in\mathcal D\), one has \(\lambda(Y)>0\).
For \(Y=X_b\) with
\(b\in D\), set
\[
 T_Y=\bigoplus_{a\in K}X_a^{\mu_D(a,Y)}.
\]
The mixed expansion \eqref{eq:directed-mixed-expansion} is the identity
\begin{equation}\label{eq:directed-mixed-residual}
 [\mathsf P_Y]
 = [\mathsf P_{T_Y}]
   +\sum_{d\in D}\mu_D(d,Y)[\mathsf S_{X_d}]
 \quad\text{in }K_0(\mod(\mod A)).
\end{equation}
Evaluating at each indecomposable module gives
\[
 \begin{aligned}
 {}[X_a,T_Y]&=[X_a,Y]&&(a\in K),\\
 [X_d,Y]-[X_d,T_Y]&=\mu_D(d,Y)\geq0&&(d\in D).
 \end{aligned}
\]
Indeed, \(\mathsf S_{X_d}(X_i)\) has dimension \(1\) when \(i=d\) and is
zero otherwise.  The module \(T_Y\) belongs to \(\CC_D\), whereas
\(Y\notin\CC_D\), so \(T_Y\) has no direct summand isomorphic to \(Y\).
Lemma~\ref{lem:directed-hom-rigidity} gives an indecomposable projective
\(P=X_p\) such that \([P,T_Y]<[P,Y]\).
Since \([X_a,T_Y]=[X_a,Y]\) for every \(a\in K\), one has \(p\notin K\)
and hence \(p\in D\).  Thus \(p\in\mathcal P_D\), and
\(\mu_D(p,Y)=[P,Y]-[P,T_Y]>0\).
Consequently, \(\lambda(Y)\geq\mu_D(p,Y)>0\).

For \(X\in\mathcal D\), put
\[
 \kappa_X=[X]_{\rm sp}-[\theta_DX]_{\rm sp}
                         +[\tau_DX]_{\rm sp}
 \quad\text{in }K_0^{\rm sp}(\overline{\AA}_D),
\]
where \([0]_{\rm sp}=0\).

\medskip
\noindent\textbf{Step 2: Right additivity of the coordinate function.}
The function \(\lambda\) is right additive on \(\overline{\AA}_D\):
\begin{equation}\label{eq:directed-lambda-values}
 \lambda(\kappa_X)=
 \begin{cases}
  1,&X\text{ is projective in }\mod A,\\
  0,&X\text{ is nonprojective in }\mod A.
 \end{cases}
\end{equation}
For \(p\in\mathcal P_D\), let
\(\varepsilon_p:K_0(\mod(\mod A))\to\ZZ\) take the
\([\mathsf S_{X_p}]\)-coordinate in the mixed basis.  Thus
\(\varepsilon_p([\mathsf P_M])=\mu_D(p,M)\).  Put
\(\varepsilon=\sum_{p\in\mathcal P_D}\varepsilon_p\).
By the definition of \(\lambda\), one has
\[
 \varepsilon([\mathsf P_M])
 =\sum_{p\in\mathcal P_D}\mu_D(p,M)=\lambda(M)
 \qquad(M\in\mod A).
\]
If \(X=X_x\in\mathcal D\), then
\([\mathsf S_X]=\mathsf B_D(x)\), and hence
\[
 \varepsilon([\mathsf S_X])=
 \begin{cases}
  1,&X\text{ is projective in }\mod A,\\
  0,&X\text{ is nonprojective in }\mod A.
 \end{cases}
\]
Indeed, \(x\in\mathcal P_D\) precisely when \(X\) is projective.
Applying \(\varepsilon\) to the two identities in
Lemma~\ref{lem:directed-functor-ar-resolutions} therefore gives
\begin{equation}\label{eq:directed-lambda-ar}
 \lambda(X)-\lambda(\theta X)+\lambda(\tau X)=0
 \qquad(X\in\mathcal D\text{ nonprojective})
\end{equation}
and
\begin{equation}\label{eq:directed-lambda-projective}
 \lambda(P)-\lambda(\theta P)=1
 \qquad(P\in\mathcal D\text{ projective}).
\end{equation}
Since \(\lambda(X_a)=0\) for every \(a\in K\), the modules \(\theta X\)
and \(\theta_DX\) have the same \(\lambda\)-value.

Let \(X\in\mathcal D\).  If \(X\) is projective in \(\mod A\), then the
identity \eqref{eq:directed-lambda-projective} gives
\(\lambda(\kappa_X)=1\).  Suppose that \(X\) is nonprojective.  If
\(\tau_DX=\tau X\), then \eqref{eq:directed-lambda-ar} gives
\(\lambda(\kappa_X)=0\).  By
Proposition~\ref{prop:iyama-factor-category}(1), if
\(\tau_DX\ne\tau X\), then \(\tau_DX=0\).  If
\(\tau X\notin\mathcal D\), then
\(\lambda(\tau X)=0\), so \eqref{eq:directed-lambda-ar} again gives
\(\lambda(\kappa_X)=0\).  The only case left would have
\(\tau X\in\mathcal D\) and \(\theta_DX=0\).  In that case,
\eqref{eq:directed-lambda-ar} gives
\(\lambda(X)+\lambda(\tau X)=0\),
contrary to Step~1.  This proves
\eqref{eq:directed-lambda-values}.  The value \(\lambda(\kappa_X)\) is zero
when \(\tau_DX\ne0\) and is nonnegative when \(\tau_DX=0\).  Step~1 gives
\(\lambda(X)>0\) for every indecomposable object \(X\) of
\(\overline{\AA}_D\).  Thus \(\lambda\) is a right additive function on
\(\overline{\AA}_D\).

\medskip
\noindent\textbf{Step 3: Identification of the two functions by strictness.}
For every \(Y\in\mathcal D\), one has
\[
 \lambda(Y)=\dim_\kk\overline{\AA}_D(\Pi_D,Y).
\]
Corollary~\ref{cor:directed-right-additive-strictness} applies by Step~2.
Its formula \eqref{eq:directed-quotient-function-expansion} and
\eqref{eq:directed-lambda-values} give
\[
 \lambda(Y)
 =\sum_{X\in\mathcal D}
   \dim_\kk\overline{\AA}_D(X,Y)\lambda(\kappa_X)
 =\sum_{p\in\mathcal P_D}
   \dim_\kk\overline{\AA}_D(X_p,Y)
 =\dim_\kk\overline{\AA}_D(\Pi_D,Y).
\]

Finally, if \(Y=X_b\) is omitted, then
\(\dim_\kk\overline{\AA}_D(\Pi_D,Y)=\lambda(Y)>0\) by Steps~1 and~3.
Since every summand of \(\Pi_D\) is a summand of the
regular module \(A\), it follows that \(\overline{\AA}_D(A,Y)\ne0\).
Lemma~\ref{lem:trace-factor-hom} identifies
\(\overline{\AA}_D(A,Y)\) with
\(Y/\tr_{\CC_D}(Y)\).  Thus no
indecomposable \(Y\notin\CC_D\) is a quotient of an object of \(\CC_D\).
Lemma~\ref{lem:indecomposable-quotient-test} now shows that \(\CC_D\) is
quotient-closed.
\end{proof}

\subsection{Orbit geometry behind the Auslander--Reiten word}
\label{sec:directed-orbit-geometry}

We prove Lemma~\ref{lem:directed-two-orbits}.  The resulting alternation of
two labels will be used in both remaining appendix proofs.

Lemma~\ref{lem:ar-mesh-correspondence} gives the Auslander--Reiten arrow
correspondence
\begin{equation}\label{eq:directed-arrow-correspondence}
 Z\longrightarrow Y,\quad Y\text{ nonprojective}
 \quad\longleftrightarrow\quad
 \tau Y\longrightarrow Z.
\end{equation}

\begin{proof}[Proof of Lemma~\ref{lem:directed-two-orbits}]
The projective and injective modules exist in each \(\tau\)-orbit:
otherwise iteration of \(\tau\), or of \(\tau^{-1}\), in the finite
Auslander--Reiten quiver would repeat and produce a directed cycle.  Index
an orbit from its projective module to its injective module, so that
\(\tau U_{i+1}=U_i\).  For \(i<m\), choose an indecomposable summand \(Z\)
of the middle term of the almost split sequence ending at \(U_{i+1}\).
Its component maps give
\[
 U_i=\tau U_{i+1}\longrightarrow Z\longrightarrow U_{i+1}.
\]
Consequently, there is a directed path
\(U_i\rightsquigarrow U_{i'}\) whenever \(i<i'\).

For an arrow \(X\to Y\) with \(Y\) nonprojective, call the arrow
\(\tau Y\to X\) given by \eqref{eq:directed-arrow-correspondence} its \emph{left
rotation}.  If \(X\) is noninjective, the inverse operation is the right
rotation
\[
 X\longrightarrow Y
 \quad\longmapsto\quad
 Y\longrightarrow\tau^{-1}X.
\]
Equation~\eqref{eq:directed-arrow-correspondence} and
Lemma~\ref{lem:representation-finite-no-multiple-arrows} make these mutually
inverse operations on arrows whenever both sides are defined.  They generate
the \(\sigma\)-orbits of arrows in the terminology of
Bongartz--Gabriel \cite[Section~4.2]{BongartzGabrielCovering}.

If the source and target have indices \(i\) and \(j\) in their respective
\(\tau\)-orbits, a left rotation decreases \(i+j\) by one.  Thus repeated
left rotation ends at an arrow whose target is projective.  An arrow inside
one orbit would therefore end at an arrow \(U_i\to U_0\).  For \(i>0\),
this arrow and the path \(U_0\rightsquigarrow U_i\) form a directed cycle;
for \(i=0\), it is a loop.  Hence there is no arrow inside one
\(\tau\)-orbit.

Now fix two different orbits connected by an arrow.  Every arrow between
them has a terminal left rotation of one of the forms
\[
 U_i\longrightarrow V_0,
 \qquad
 V_j\longrightarrow U_0.
\]
The two forms cannot both occur, since the paths within the two orbits
would give a directed cycle
\[
 U_0\rightsquigarrow U_i\longrightarrow V_0
 \rightsquigarrow V_j\longrightarrow U_0.
\]
After exchanging the orbits if necessary, every arrow therefore rotates
to an arrow \(U_i\to V_0\).

There is only one possible index \(i\).  Indeed, suppose that
\(U_i\to V_0\) and \(U_{i'}\to V_0\) are arrows with \(i<i'\).  The right
rotation of the first arrow exists and gives \(V_0\to U_{i+1}\).  Hence
\[
 V_0\longrightarrow U_{i+1}\rightsquigarrow U_{i'}
 \longrightarrow V_0
\]
is a directed cycle, where the middle path is allowed to be trivial.
Lemma~\ref{lem:representation-finite-no-multiple-arrows} now shows that there
is a unique terminal arrow \(U_t\to V_0\).

Every arrow between the two orbits belongs to the rotation orbit of this
arrow.  Its successive right rotations are
\[
 U_{t+n}\longrightarrow V_n\longrightarrow U_{t+n+1}
\]
for every \(n\) for which \(U_{t+n}\), \(V_n\), and \(U_{t+n+1}\) exist.
Taking \(c=-t\) and putting
\(i=t+n\) gives \eqref{eq:directed-orbit-arrows}.
\end{proof}

\begin{lemma}\label{lem:directed-two-label-alternation}
For each edge \(a\sim b\) of \(\Delta_A\), let
\(\underline{w}_A|_{a,b}\) be the subword obtained by deleting all letters other
than \(s_a\) and \(s_b\).  Write
\[
 \underline{w}_A|_{a,b}=B_1B_2\cdots B_t
\]
as its decomposition into maximal blocks of equal letters.  Then
\(\lvert B_h\rvert=1\) for every \(2\leq h\leq t-1\).
\end{lemma}

\begin{proof}
For two adjacent orbits, the arrows
\(U_i\to V_{i+c}\to U_{i+1}\) in
\eqref{eq:directed-orbit-arrows} place \(V_{i+c}\) strictly between
\(U_i\) and \(U_{i+1}\) in the Hom-compatible order.  The two labels
therefore alternate whenever both orbits occur; repetition is possible only
before the first or after the last occurrence of the other label.
\end{proof}

\subsection{Positivity for reduced subwords}
\label{sec:directed-positivity}

Put \(\overline{\AA}_D=(\mod A)/[\CC_D]\) as in
Appendix~\ref{sec:directed-tau-closure}.  The goal of this subsection
is the factor-category identification in
Lemma~\ref{lem:directed-principal-positive}: if the \(D\)-subword of
\(\underline{w}_A\) is reduced, then
\[
 \mu_D(a,x)
 =\dim_\kk\overline{\AA}_D(X_a,X_x)\geq0
 \qquad(a,x\in D).
\]
The Auslander--Reiten species of \(\overline{\AA}_D\) is obtained by
deleting the positions outside \(D\).  We first describe the
\(\tau_D\)-orbits in \(\overline{\AA}_D\)
and their Coxeter word.  As in Appendix~\ref{sec:directed-tau-closure}, we
then apply Corollary~\ref{cor:directed-right-additive-strictness}.  In that
argument the function came from the assumed nonnegative coordinates; here
it comes from the heights of positive inversion roots.  Equation
\eqref{eq:directed-quotient-class-expansion} then identifies the mixed
coordinates with dimensions of Hom spaces in \(\overline{\AA}_D\).

We use the following standard facts about the geometric representation.

\begin{lemma}[{\cite[Proposition~4.2.5 and
equations~(4.24)--(4.25)]{BjornerBrenti}}]
\label{lem:coxeter-geometric-representation}
Let \(\Delta\) be a finite simple graph.  The geometric representation of
\(W(\Delta)\) on the real span of the simple roots \(\alpha_h\) has the
following properties.  If \(z=\sum_hz_h\alpha_h\), then
\begin{equation}\label{eq:directed-geometric-reflection}
 (s_i z)_i=-z_i+\sum_{h\sim i}z_h,\qquad
 (s_i z)_h=z_h\quad(h\ne i).
\end{equation}
The roots are the vectors \(w\alpha_h\) with \(w\in W(\Delta)\).
Every root is either positive, with all coordinates nonnegative, or
negative, with all coordinates nonpositive.
\end{lemma}

Let \(I\) be a finite totally ordered set and let
\(\underline w=(s_{j_x})_{x\in I}\) be an \(I\)-indexed word.  For
\(J\subseteq I\), write \(\prod_{y\in J}^{\nearrow}s_{j_y}\) for the
product in increasing order.  The inversion root at \(x\in I\) is
\[
 \beta_x=
 \left(\prod_{\substack{y\in I\\y<x}}^{\nearrow}s_{j_y}\right)
 (\alpha_{j_x}).
\]
\begin{lemma}[{\cite[Proposition~4.2.5 and
Equation~(4.25)]{BjornerBrenti}}]
\label{lem:coxeter-inversion-root-criterion}
Let \(I\) be a finite totally ordered set and let
\(\underline w=(s_{j_x})_{x\in I}\) be an \(I\)-indexed word.  Then
\(\underline w\) is reduced if and only if all its inversion roots are
positive.
\end{lemma}

For a positive root \(\beta=\sum_i b_i\alpha_i\), put
\(\operatorname{ht}(\beta)=\sum_i b_i\).

We extend the position construction in
Construction~\ref{constr:directed-ar-mesh-positions} to an arbitrary
finite ordered word.

\begin{construction}\label{constr:word-mesh-positions}
Let \(\Delta\) be a finite simple graph and
\(\underline w=(s_{j_x})_{x\in I}\) an \(I\)-indexed word in the simple
generators of \(W(\Delta)\), where \(I\) is a finite totally ordered set
and \(j_x\in V(\Delta)\).  For \(x\in I\), define
\[
 p_{\underline w}(x):=\max\{\,y\in I\mid y<x,\ j_y=j_x\,\}
\]
when the set on the right is nonempty.  Put
\[
 \mathcal W_{\underline w}(x):=
 \begin{cases}
  \{\,y\in I\mid p_{\underline w}(x)<y<x\,\},
     &p_{\underline w}(x)\text{ exists},\\
  \{\,y\in I\mid y<x\,\},&p_{\underline w}(x)\text{ does not exist}.
 \end{cases}
\]
Define
\[
 \operatorname{mid}_{\underline w}(x):=
 \left\{
  \max\{\,y\in\mathcal W_{\underline w}(x)\mid j_y=h\,\}
  \ \middle|\
  \begin{array}{c}
   h\sim j_x\text{ in }\Delta,\\[-2pt]
   \{\,y\in\mathcal W_{\underline w}(x)\mid j_y=h\,\}\ne\varnothing
  \end{array}
 \right\}.
\]
\end{construction}

\begin{construction}\label{constr:directed-factor-word}
Fix \(D\subseteq[N]\).  By
Proposition~\ref{prop:iyama-factor-category}, the ideal quotient
\(\overline{\AA}_D=(\mod A)/[\CC_D]\) is a \(\tau\)-category whose
indecomposable objects are the modules
\(X_x\) with \(x\in D\), and whose Auslander--Reiten species is obtained by
restriction from \(\mod A\).  Let \(\Delta_D\) be the graph whose vertices
are the \(\tau\)-orbits of \(\overline{\AA}_D\), that is, the orbits of its
translation \(\tau_D\).  Two vertices are joined when the
Auslander--Reiten species of \(\overline{\AA}_D\) has an arrow between
objects in the corresponding \(\tau_D\)-orbits.
Lemma~\ref{lem:directed-two-orbits} ensures that this graph has no loops.

For \(x\in D\), let \(j_x\in V(\Delta_D)\) be the \(\tau_D\)-orbit
containing \(X_x\).  The order inherited from \(X_1,\ldots,X_N\) is
Hom-compatible for \(\overline{\AA}_D\), and defines its
Auslander--Reiten word
\[
 \underline{w}_D=(s_{j_x})_{x\in D},
\]
with the positions read in increasing order.  Since \(\tau_DX\) is either
\(\tau X\) or zero, every \(\tau_D\)-orbit in \(\overline{\AA}_D\) is
contained in a unique \(\tau\)-orbit in \(\mod A\).  Define
\(\pi_D:V(\Delta_D)\to V(\Delta_A)\) by letting \(\pi_D(b)\) be the unique
\(\tau\)-orbit in \(\mod A\) that
contains \(b\).  Replacing every label \(j_x\) by \(\pi_D(j_x)\) recovers
the \(D\)-subword of
\(\underline{w}_A\).

Apply Construction~\ref{constr:word-mesh-positions} to
\(\underline{w}_D\), and, for \(x\in D\), put
\[
 p_D(x):=p_{\underline{w}_D}(x),\qquad
 \operatorname{mid}_D(x):=\operatorname{mid}_{\underline{w}_D}(x).
\]
\end{construction}

\begin{example}\label{ex:directed-factor-word}
Let \(A=\kk Q\) for the linearly oriented quiver
\[
 Q:\quad 1\longrightarrow2\longrightarrow3.
\]
For the interval modules on this quiver, one Hom-compatible order and its
Auslander--Reiten quiver are
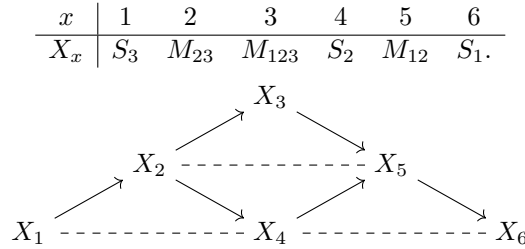
\[
\begin{array}{c|cccccc}
x&1&2&3&4&5&6\\ \hline
X_x&S_3&M_{23}&M_{123}&S_2&M_{12}&S_1.
\end{array}
\]
\begin{center}
\begin{tikzpicture}[xscale=1.6,yscale=0.9]
\node (x1) at (0,-1) {\(X_1\)};
\node (x2) at (1,0) {\(X_2\)};
\node (x3) at (2,1) {\(X_3\)};
\node (x4) at (2,-1) {\(X_4\)};
\node (x5) at (3,0) {\(X_5\)};
\node (x6) at (4,-1) {\(X_6\)};
\draw[->] (x1) -- (x2);
\draw[->] (x2) -- (x3);
\draw[->] (x2) -- (x4);
\draw[->] (x3) -- (x5);
\draw[->] (x4) -- (x5);
\draw[->] (x5) -- (x6);
\draw[dashed] (x4) -- (x1);
\draw[dashed] (x6) -- (x4);
\draw[dashed] (x5) -- (x2);
\end{tikzpicture}
\end{center}
where a dashed line joins each nonprojective module to its translate.
The three \(\tau\)-orbits, indexed from their projective ends, are
\[
 \mathcal O_1=(X_1,X_4,X_6),\qquad
 \mathcal O_2=(X_2,X_5),\qquad
 \mathcal O_3=(X_3).
\]
Thus \(\Delta_A\) is the path \(1-2-3\), and
\[
 \underline{w}_A=s_1s_2s_3s_1s_2s_1.
\]

Take \(D=\{1,2,3,5,6\}\).  Thus
\(\overline{\AA}_D=(\mod A)/[\add X_4]\).
Its Auslander--Reiten quiver is
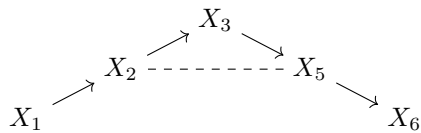
\[
\begin{tikzpicture}[baseline=-2pt,xscale=1.25,yscale=0.65]
\node (x1) at (0,-1) {\(X_1\)};
\node (x2) at (1,0) {\(X_2\)};
\node (x3) at (2,1) {\(X_3\)};
\node (x5) at (3,0) {\(X_5\)};
\node (x6) at (4,-1) {\(X_6\)};
\draw[->] (x1) -- (x2);
\draw[->] (x2) -- (x3);
\draw[->] (x3) -- (x5);
\draw[->] (x5) -- (x6);
\draw[dashed] (x5) -- (x2);
\end{tikzpicture}
\]
where the dashed line records \(\tau_DX_5=X_2\).  This is the only nonzero
value of \(\tau_D\), and the four \(\tau_D\)-orbits in
\(\overline{\AA}_D\) are
\[
 a_0=\{X_1\},\qquad b=\{X_2,X_5\},\qquad
 c=\{X_3\},\qquad a_1=\{X_6\}.
\]
Identifying the vertices in each \(\tau_D\)-orbit and forgetting the arrow
orientations gives
\[
\begin{tikzpicture}[baseline=-2pt,xscale=1.3,yscale=0.8]
\node (a0) at (0,0) {\(a_0\)};
\node (b) at (1,0) {\(b\)};
\node (c) at (2,0) {\(c\)};
\node (a1) at (1,-1) {\(a_1\)};
\draw (a0) -- (b) -- (c);
\draw (b) -- (a1);
\node at (3.0,-0.2) {\(=\Delta_D\)};
\end{tikzpicture}
\]
and the Auslander--Reiten word of \(\overline{\AA}_D\) is
\[
 \underline{w}_D=s_{a_0}s_bs_cs_bs_{a_1}.
\]
The map \(\pi_D\) sends \(a_0,a_1\) to \(1\), sends \(b\) to \(2\),
and sends \(c\) to \(3\).  Replacing every label \(v\) by \(\pi_D(v)\)
therefore gives the reduced \(D\)-subword \(s_1s_2s_3s_2s_1\) of
\(\underline{w}_A\).  Notice that position \(6\) does not take position
\(1\) as its predecessor in \(\underline{w}_D\): the object
\(X_{p(6)}=X_4\) is zero in \(\overline{\AA}_D\), so \(X_1\) and \(X_6\)
belong to different \(\tau_D\)-orbits.
\end{example}

The root calculation below uses the following property of an ordered word.

\begin{definition}\label{def:directed-alternation-property}
Let \(\Delta\) be a simple graph, let \(I\) be a finite totally ordered
set, and let \(\underline w=(s_{j_x})_{x\in I}\) be an \(I\)-indexed word
whose labels belong to \(V(\Delta)\).  For each edge \(a\sim b\), let
\(\underline w|_{a,b}\) be the
subword obtained by deleting all letters other than \(s_a\) and \(s_b\), and
write
\[
 \underline w|_{a,b}=B_1B_2\cdots B_t
\]
as its decomposition into maximal blocks of equal letters.  The word
\(\underline w\) has the \emph{alternation property} if
\(\lvert B_h\rvert=1\) for every \(2\leq h\leq t-1\) and every edge
\(a\sim b\) of \(\Delta\).
\end{definition}

\begin{lemma}\label{lem:directed-factor-word-alternation}
The word \(\underline{w}_D\) has the alternation property.
\end{lemma}

\begin{proof}
Every \(\tau_D\)-orbit is a consecutive segment of the \(\tau\)-orbit in
\(\mod A\) that contains it.  If \(b,c\in V(\Delta_D)\) are adjacent, then
\(\pi_D(b)\) and \(\pi_D(c)\) are adjacent in \(\Delta_A\).
Lemma~\ref{lem:directed-two-label-alternation} shows that the occurrences of
these two labels in \(\underline{w}_A\) alternate, except possibly before
the first or after the last occurrence of the other label.  Restricting to
the segments \(b\) and \(c\) proves the assertion for
\(\underline{w}_D\).
\end{proof}

\begin{lemma}\label{lem:directed-delete-separated-edge}
Let \(\Delta\) be a finite simple graph, and let \(\underline w\) be a reduced
word in \(W(\Delta)\).  Suppose that \(a\sim b\) in \(\Delta\) and that every
occurrence of \(s_a\) in \(\underline w\) precedes every occurrence of \(s_b\).
Let \(\Delta'\) be the graph obtained from \(\Delta\) by deleting the edge
\(a\sim b\).  Then \(\underline w\) is reduced in \(W(\Delta')\).
\end{lemma}

\begin{proof}
Write \(\underline w=\underline u\,\underline v\), where no \(b\) occurs
in \(\underline u\) and no \(a\) occurs in \(\underline v\).  For a
position in \(\underline u\), the \(b\)-coordinate of its inversion root
is zero, so deleting the \(z_b\)-summand from the action of \(s_a\) in
\eqref{eq:directed-geometric-reflection} does not change that root.

For a position in \(\underline v\), first form its inversion root \(\gamma\)
using only the prefix inside \(\underline v\).  It is positive because
\(\underline v\) is a suffix of a reduced word.  No \(s_a\) occurs there,
so \(\gamma_a=0\).  Put \(c=\gamma_b\geq0\).  While the prefix
\(\underline u\) acts, no \(s_b\) occurs, so
the \(b\)-coordinate stays equal to \(c\).  If \(c=0\), deleting the edge
changes nothing.  If \(c>0\), the resulting real root cannot be negative,
because its \(b\)-coordinate is positive.  Thus all inversion roots remain
positive, and the word remains reduced.
\end{proof}

\begin{proposition}\label{prop:directed-factor-word-reduced}
If the \(D\)-subword of \(\underline{w}_A\) is reduced, then
\(\underline{w}_D\) is reduced in
\(W(\Delta_D)\).
\end{proposition}

\begin{proof}
Let \(\widehat\Delta_D\) be the graph with vertex set \(V(\Delta_D)\) in
which \(b\sim c\) if and only if \(\pi_D(b)\sim\pi_D(c)\) in \(\Delta_A\).
The assignment \(s_b\mapsto s_{\pi_D(b)}\) respects the Coxeter relations
and gives a homomorphism
\(\varphi:W(\widehat\Delta_D)\to W(\Delta_A)\).
The element \(\varphi(\underline{w}_D)\) is represented by the reduced
\(D\)-subword of \(\underline{w}_A\), which has length \(|D|\).  Since
\(\varphi\) sends simple reflections to simple reflections, it does not
increase Coxeter length; this gives the first inequality below.  The second
follows because, by definition, \(\underline{w}_D\) has \(|D|\) letters:
\[
 |D|\leq\ell_{W(\widehat\Delta_D)}(\underline{w}_D)\leq|D|.
\]
Thus \(\ell_{W(\widehat\Delta_D)}(\underline{w}_D)=|D|\), so
\(\underline{w}_D\) is reduced in
\(W(\widehat\Delta_D)\).

Now let \(b-c\) be an edge of \(\widehat\Delta_D\) which is not an edge of
\(\Delta_D\).  Thus the Auslander--Reiten species of
\(\overline{\AA}_D\) has no arrow between objects belonging to the orbits
\(b\) and \(c\).  The orbits \(b\) and \(c\) are consecutive segments of
the \(\tau\)-orbits \(\pi_D(b)\) and \(\pi_D(c)\) in \(\mod A\).  If neither
every occurrence of \(b\) precedes every occurrence of \(c\) nor every
occurrence of \(c\) precedes every occurrence of \(b\), then some object
belonging to \(b\) and some object belonging to \(c\) are consecutive in the
alternating pattern \eqref{eq:directed-orbit-arrows}.  That pattern gives an
Auslander--Reiten arrow between them, a contradiction.  Hence all
occurrences of one of \(b,c\) precede all occurrences of the other.  By
Lemma~\ref{lem:directed-delete-separated-edge}, the word
remains reduced after this edge is deleted.  Deleting all such edges changes
\(\widehat\Delta_D\) into \(\Delta_D\).
\end{proof}

We next compare \(p_D(x)\) and \(\operatorname{mid}_D(x)\) with
\(p(x)\) and \(\operatorname{mid}(x)\), respectively.

\begin{lemma}\label{lem:directed-factor-mesh-positions}
Suppose that the \(D\)-subword of \(\underline{w}_A\) is reduced.  For
\(x\in D\), the following statements hold.
\begin{enumerate}
\item One has
\begin{equation}
 \theta_DX_x
 =\bigoplus_{y\in\operatorname{mid}(x)\cap D}X_y.
 \label{eq:directed-factor-middle}
\end{equation}
If \(p(x)\in D\), then \(\theta_DX_x\ne0\).
\item The predecessor \(p_D(x)\) exists exactly when \(p(x)\in D\), and
then \(p_D(x)=p(x)\).  Consequently,
\begin{equation}
 \tau_DX_x
   =
   \begin{cases}
    X_{p(x)},&p(x)\in D,\\
    0,&p(x)\text{ does not exist or }p(x)\notin D.
   \end{cases}
   \label{eq:directed-factor-translation}
\end{equation}
\item One has
\[
 \operatorname{mid}_D(x)=\operatorname{mid}(x)\cap D.
\]
\end{enumerate}
\end{lemma}

\begin{proof}
\((1)\)
Proposition~\ref{prop:iyama-factor-category} says that the
Auslander--Reiten species of \(\overline{\AA}_D\) is obtained by restricting
the species of \(\mod A\).  Lemma~\ref{lem:directed-word-dictionary}(2)
therefore gives \eqref{eq:directed-factor-middle}.  Suppose that
\(p(x)\in D\) and
\(\operatorname{mid}(x)\cap D=\varnothing\).  For every
\(h\sim i_x\) in \(\Delta_A\),
Lemma~\ref{lem:directed-two-label-alternation} shows that at most one
position labeled \(h\) lies between \(p(x)\) and \(x\).
Lemma~\ref{lem:directed-word-dictionary}(2) identifies that position, when
it exists, with a member of \(\operatorname{mid}(x)\).
Hence every letter of the \(D\)-subword between \(p(x)\) and \(x\)
commutes with \(s_{i_x}\).  Moving the two occurrences of \(s_{i_x}\)
together would produce \(s_{i_x}s_{i_x}\), contrary to the assumed
reducedness of the \(D\)-subword.  Thus
\(\operatorname{mid}(x)\cap D\ne\varnothing\),
and \eqref{eq:directed-factor-middle} gives \(\theta_DX_x\ne0\).

\((2)\)
If \(p(x)\in D\), part~(1) and the definition of \(\tau_D\) give
\(\tau_DX_x=X_{p(x)}\).  Thus \(p(x)\) and \(x\) are consecutive
positions labeled \(j_x\) in \(\underline{w}_D\), so
\(p_D(x)=p(x)\).  If \(p(x)\)
does not exist or \(p(x)\notin D\), then \(\tau_DX_x=0\), and \(x\) is the
first position labeled \(j_x\) in \(\underline{w}_D\).  Hence \(p_D(x)\) is
undefined.  This proves (2) and
\eqref{eq:directed-factor-translation}.

\((3)\)
Let \(y\in\operatorname{mid}(x)\cap D\).
Lemma~\ref{lem:directed-word-dictionary}(2) gives an Auslander--Reiten arrow
\(X_y\to X_x\), so \(j_y\sim j_x\) in \(\Delta_D\).  If \(p(x)\in D\), part~(2) and
Lemma~\ref{lem:directed-factor-word-alternation} make \(y\) the last
position labeled \(j_y\) in
\((p_D(x),x)=(p(x),x)\).  If \(p(x)\) does not exist or \(p(x)\notin D\),
then \(p_D(x)\) is undefined, and \(y\) is the last occurrence of the
label \(j_y\) before \(x\).  Thus
\(y\in\operatorname{mid}_D(x)\).

Conversely, let \(y\in\operatorname{mid}_D(x)\).  Then
\(j_y\sim j_x\) in \(\Delta_D\).  For the \(\tau\)-orbits
\(\pi_D(j_y)\) and \(\pi_D(j_x)\) in \(\mod A\),
Lemma~\ref{lem:directed-two-orbits} gives the alternating arrow pattern
\(U_i\to V_{i+c}\to U_{i+1}\).  The position \(y\) is, by
Construction~\ref{constr:word-mesh-positions}, the last occurrence of the
label \(j_y\) in the interval preceding \(x\).  The arrow pattern in
\eqref{eq:directed-orbit-arrows} therefore makes \(X_y\) the source of an
arrow into \(X_x\).
Thus \(X_y\to X_x\), and
Lemma~\ref{lem:directed-word-dictionary}(2) gives
\(y\in\operatorname{mid}(x)\).  Since \(y\in D\) by construction, this
proves (3).
\end{proof}

For \(x\in D\), put
\[
 \kappa_x^D=[X_x]_{\rm sp}-[\theta_DX_x]_{\rm sp}
                         +[\tau_DX_x]_{\rm sp}
 \quad\text{in }K_0^{\rm sp}(\overline{\AA}_D).
\]

\begin{lemma}\label{lem:directed-factor-mesh-basis}
Suppose that the \(D\)-subword of \(\underline{w}_A\) is reduced.  The
classes \(\kappa_x^D\), for \(x\in D\), form a basis of
\(K_0^{\rm sp}(\overline{\AA}_D)\), and
\begin{equation}\label{eq:directed-factor-mesh-expansion}
 [X_x]_{\rm sp}
 =\sum_{a\in D}\dim_\kk\overline{\AA}_D(X_a,X_x)\,\kappa_a^D
 \qquad(x\in D).
\end{equation}
\end{lemma}

\begin{proof}
By Corollary~\ref{cor:directed-right-additive-strictness}, it suffices to
construct a right additive function on \(\overline{\AA}_D\).  We construct
one from the heights of the inversion roots of \(\underline{w}_D\).

Let \(\beta_x\) be the inversion root of \(\underline{w}_D\) at \(x\):
\[
 \beta_x=
 \left(\prod_{\substack{y\in D\\y<x}}^{\nearrow}s_{j_y}\right)
 (\alpha_{j_x}).
\]
The word \(\underline{w}_D\) is reduced by
Proposition~\ref{prop:directed-factor-word-reduced}, so every \(\beta_x\) is
positive by Lemma~\ref{lem:coxeter-inversion-root-criterion}.

Define \(\operatorname{ht}_D(X_x)=\operatorname{ht}(\beta_x)\) and extend
\(\operatorname{ht}_D\) additively.  The positivity of \(\beta_x\) gives
\(\operatorname{ht}_D(X_x)>0\) for every indecomposable object \(X_x\) of
\(\overline{\AA}_D\).  We prove that \(\operatorname{ht}_D\) is right
additive.  By Lemma~\ref{lem:directed-factor-mesh-positions}(2), it is enough
to prove
\[
 \operatorname{ht}_D(\kappa_x^D)=
 \begin{cases}
  0,&p_D(x)\text{ exists},\\
  >0,&p_D(x)\text{ is undefined}.
 \end{cases}
\]

Lemma~\ref{lem:directed-factor-mesh-positions} gives
\[
 \kappa_x^D=[X_x]_{\rm sp}
 -\sum_{y\in\operatorname{mid}_D(x)}[X_y]_{\rm sp}
 +[X_{p_D(x)}]_{\rm sp},
\]
where the last term is omitted when \(p_D(x)\) is undefined.

Suppose that \(p=p_D(x)\) exists.  Starting immediately after the reflection
at \(p\), put
\[
 \gamma_y=
 \left(\prod_{\substack{z\in D\\z\leq y}}^{\nearrow}s_{j_z}\right)
 (\alpha_{j_x})
 \qquad(p\leq y<x).
\]
Then \(\gamma_p=-\beta_p\), while the value after the last position before
\(x\) is \(\beta_x\).  As \(y\) passes through the positions with
\(p<y<x\), equation~\eqref{eq:directed-geometric-reflection} adds
\(\beta_y\) when \(j_y\sim j_x\) and makes no change otherwise.  Hence
\begin{equation}\label{eq:directed-root-mesh}
 \beta_p+\beta_x
 =\sum_{\substack{p<y<x\\j_y\sim j_x}}\beta_y
 =\sum_{y\in\operatorname{mid}_D(x)}\beta_y.
\end{equation}
The second equality follows from
Lemma~\ref{lem:directed-factor-word-alternation}: between consecutive
occurrences of \(j_x\), each neighboring label occurs at most once.
If \(p_D(x)\) is undefined, applying
\eqref{eq:directed-geometric-reflection} successively at the positions of
\(\underline{w}_D\) before \(x\), starting from \(\alpha_{j_x}\), gives
\begin{equation}\label{eq:directed-root-projective}
 \beta_x-\sum_{y\in\operatorname{mid}_D(x)}\beta_y
 =\alpha_{j_x}
  +\sum_{\substack{y<x,\ j_y\sim j_x\\
  y\text{ is not the last occurrence of }j_y\text{ before }x}}
  \beta_y.
\end{equation}

If \(p_D(x)\) exists, \eqref{eq:directed-root-mesh} gives
\(\operatorname{ht}_D(\kappa_x^D)=0\).  If \(p_D(x)\) is undefined,
\eqref{eq:directed-root-projective} gives
\(\operatorname{ht}_D(\kappa_x^D)>0\).  Thus \(\operatorname{ht}_D\) is a
right additive function on \(\overline{\AA}_D\).
Corollary~\ref{cor:directed-right-additive-strictness} now shows that the
classes \(\kappa_x^D\), for \(x\in D\), form a basis.  Its formula
\eqref{eq:directed-quotient-class-expansion} is
\eqref{eq:directed-factor-mesh-expansion}.
\end{proof}

\begin{proof}[Proof of Lemma~\ref{lem:directed-principal-positive}]
For \(a\in D\) and \(y\notin D\), the first case of
Lemma~\ref{lem:directed-coordinate-recurrence} gives
\(\mu_D(a,y)=\delta_{a,y}=0\).  Define
\[
 \widetilde\mu_a:K_0^{\rm sp}(\overline{\AA}_D)\longrightarrow\ZZ,
 \qquad [X_x]_{\rm sp}\longmapsto\mu_D(a,x)
 \quad(a,x\in D).
\]
For \(x\in D\), Lemma~\ref{lem:directed-factor-mesh-positions} gives
\[
 \kappa_x^D
 = [X_x]_{\rm sp}
   -\sum_{y\in\operatorname{mid}(x)\cap D}[X_y]_{\rm sp}
   +
   \begin{cases}
    [X_{p(x)}]_{\rm sp},&p(x)\in D,\\
    0,&p(x)\text{ does not exist or belongs to }[N]\setminus D.
   \end{cases}
\]
Since \(\mu_D(a,y)=0\) for \(a\in D\) and \(y\notin D\), applying
\(\widetilde\mu_a\) and then
\eqref{eq:directed-coordinate-recurrence} gives
\begin{align*}
 \widetilde\mu_a(\kappa_x^D)
 &=\mu_D(a,x)
   -\sum_{y\in\operatorname{mid}(x)}\mu_D(a,y)
   +
   \begin{cases}
    \mu_D(a,p(x)),&p(x)\text{ exists},\\
    0,&p(x)\text{ does not exist}
   \end{cases}\\
 &=\delta_{a,x}
 \qquad(a,x\in D).
\end{align*}
Thus the \(\widetilde\mu_a\) are the coordinate functionals dual to the
basis \(\{\kappa_x^D\mid x\in D\}\).  Applying \(\widetilde\mu_a\) to
\eqref{eq:directed-factor-mesh-expansion} gives
\[
 \mu_D(a,x)=\dim_\kk\overline{\AA}_D(X_a,X_x)\geq0
 \qquad(a,x\in D).
\]
\end{proof}

\subsection{A negative mixed coordinate from the first negative root}
\label{sec:directed-cancellation}

\begin{proof}[Proof of Lemma~\ref{lem:directed-first-cancellation}]
Lemma~\ref{lem:directed-first-cancellation} computes \(\mu_D(a,x)\) at the
first position \(x\) at which
the root sequence \(\beta_y\) becomes negative.  These roots are produced
by a process which applies the simple reflections at the positions
\(p\in D\) to
\(\alpha_{i_a}\).  Immediately before \(x\), the root is
\(\beta_{x-1}=\alpha_{i_x}\), and applying \(s_{i_x}\) changes it to
\(-\alpha_{i_x}\): a simple reflection sends a positive root to a negative
root only when that root is the corresponding simple root.  All earlier
roots are positive, so every nonzero
\(c_e=(\beta_e)_{i_e}\) is a positive integer; in particular,
\(c_e\geq0\) for every \(a<e<x\).

The proof compares two processes.  The \emph{root process} produces the
roots \(\beta_y\), while a \emph{coordinate process} produces the
coefficients \(\mu_D(e,x)\).  At a position \(p\in D\), the two processes
take the same step on the current root.  At a position \(e\notin D\), the
root process keeps the current vector while the coordinate process removes
its \(i_e\)-coordinate, namely \(c_e\alpha_{i_e}\).  The later coordinate
steps carry this removed simple root to the final \(i_x\)-coordinate with
coefficient \(\mu_D(e,x)\).  Thus
\eqref{eq:directed-first-cancellation} sums the contributions of these local
differences.

We now define the coordinate process precisely.  Its state is a vector
\(z=(z_h)_h\), where, after each position, \(z_h\) records the coefficient
at the most recent occurrence of the label \(h\).  It starts immediately
after position \(a\) with \(z=\alpha_{i_a}\).
At a position \(e\notin D\), the root process does nothing, while the
coordinate process sets the \(i_e\)-coordinate to zero and leaves the other
coordinates unchanged.  At a position \(p\in D\), the
root process applies \(s_{i_p}\), while the coordinate process replaces
\(z_{i_p}\) by
\[
 -z_{i_p}+
 \sum_{y\in\operatorname{mid}(p)}z_{i_y}.
\]
It leaves every \(z_h\) with \(h\ne i_p\) unchanged.  The coordinate
step at \(p\in D\) is the \(p\in D\) case of
Lemma~\ref{lem:directed-coordinate-recurrence}, while the step at
\(e\notin D\) is its \(e\notin D\) case.  Induction on the position
therefore shows that, if the coordinate process
starts immediately after a position \(e\) with the vector
\(\alpha_{i_e}\), then its
\(i_x\)-coordinate after
position \(x\) is \(\mu_D(e,x)\).  In particular, when it begins at \(a\),
this coordinate is \(\mu_D(a,x)\).

\medskip
\noindent\textbf{Claim.}\ \emph{For every \(p\in D\cap(a,x]\), applying
the coordinate step to \(\beta_{p-1}\)
has the same effect as applying \(s_{i_p}\).}

\noindent\emph{Proof.}
The two steps
can differ only by the contribution of a label \(h\sim i_p\) for which no
\(y\in\operatorname{mid}(p)\) satisfies \(i_y=h\).
We prove
\[
 (\beta_{p-1})_h=0
\]
for every such \(h\).

If no occurrence of \(h\) precedes \(p\), its coordinate has remained
zero.  If \(p\) is the first occurrence of the label \(i_p\), then the
most recent occurrence of every neighboring label before \(p\) belongs to
\(\operatorname{mid}(p)\).  Thus suppose that \(h\) occurs before \(p\)
and that \(p\) is not the first occurrence of \(i_p\).  Since no position
in \(\operatorname{mid}(p)\) is labeled \(h\), the most recent occurrence
of \(h\) precedes the preceding occurrence of \(i_p\).  After all other
labels are deleted, the two \(i_p\)-occurrences form a consecutive string
preceded by an occurrence of \(h\).
Lemma~\ref{lem:directed-two-label-alternation} makes
this the last string of \(i_p\)'s, so \(h\) has no occurrence in \((p,x]\).
The \(h\)-coordinate of the root process is therefore unchanged from
\(\beta_{p-1}\) to \(\beta_{x-1}=\alpha_{i_x}\), and hence it is zero.
For \(p=x\), the equality
\(\beta_{x-1}=\alpha_{i_x}\) makes every neighboring coordinate zero
directly.  Therefore
the two steps agree at every \(p\in D\cap(a,x]\).
\hfill\(\square\)

Let \(z_y\) be the state of the coordinate process after position \(y\).
For every \(e\in(a,x)\setminus D\), let \(z_y^{(e)}\), for \(y\geq e\),
be the state obtained by starting immediately after \(e\) with
\(z_e^{(e)}=\alpha_{i_e}\) and applying the same later coordinate steps.
We claim that
\begin{equation}\label{eq:directed-discrepancy-induction}
 z_y=\beta_y-
 \sum_{\substack{a<e\leq y\\e\notin D}}c_ez_y^{(e)}
 \qquad(a\leq y\leq x).
\end{equation}
This follows by induction on \(y\).  At a position \(y\in D\), the
coordinate and root steps agree on \(\beta_{y-1}\), so the coordinate
step sends this root term to \(\beta_y\); linearity carries every earlier
error term through the same step.  At a position \(e\notin D\), the
root is unchanged and the coordinate step removes exactly
\(c_e\alpha_{i_e}=c_ez_e^{(e)}\).  This proves
\eqref{eq:directed-discrepancy-induction}.

The \(i_x\)-coordinate of \(z_x\) is \(\mu_D(a,x)\), and that of
\(z_x^{(e)}\) is \(\mu_D(e,x)\).  Since
\(\beta_x=-\alpha_{i_x}\), taking the \(i_x\)-coordinate in
\eqref{eq:directed-discrepancy-induction} gives
\[
 \mu_D(a,x)
 =-1-\sum_{\substack{a<e<x\\e\notin D}}c_e\mu_D(e,x),
\]
which is \eqref{eq:directed-first-cancellation}.
\end{proof}

\section{Classification and calculations for size four}
\label{app:small-core}

This appendix proves the two lemmas of Section~\ref{sec:size-four} used for
the equality in size four.  Recall from Lemma~\ref{lem:faithful-core} that
\(\mathcal K_{\quot}(B)=\ind\Gen\bigl(\DDual({}_B B)\bigr)\) and
\(\mathcal K_{\sub}(B)=\ind\Cogen(B_B)\) are the minimal faithful
quotient-closed and submodule-closed sets, that \(\mathcal K_{\sub}(B)\)
consists of the indecomposable torsionless modules, and that \(d(B)\) is the
cardinality of each of \(\mathcal K_{\quot}(B)\) and
\(\mathcal K_{\sub}(B)\)
(Proposition~\ref{prop:ringel-torsionless}).  We first classify the algebras
\(B\) with \(d(B)\leq3\) (Lemma~\ref{lem:small-faithful-core}).  We then
count the faithful closed subcategories of size four over the algebras which
that classification leaves outside the Nakayama case
(Lemma~\ref{lem:exceptional-faithful-four}).

\subsection{Classification of algebras with at most three indecomposable
torsionless modules}

\begin{proof}[Proof of Lemma~\ref{lem:small-faithful-core}]
We must show that \(B\) is Nakayama, the path algebra of one of the two
\(A_3\) orientations with a source or a sink of valency two, or, up to
passing to \(B^{\op}\), one of the algebras \(L_0=\kk Q/(x^2,xa)\) and
\(L_1=\kk Q/(x^2)\) on the quiver \(Q\) with a loop \(x\) at vertex \(1\)
and an arrow \(a:1\to2\).

First suppose that \(B\) is connected, and let \(r\) be its number of simple
modules.  The \(r\) indecomposable projectives belong to
\(\mathcal K_{\sub}(B)\), so \(r\leq d(B)\leq3\).
If \(r=d(B)\), every indecomposable torsionless module is projective.  Every
right ideal is a direct sum of indecomposable torsionless modules and is
therefore projective.  Hence \(B\) is hereditary.  Since \(B\) is basic,
representation-finite, and split over the algebraically closed field
\(\kk\), it is the path algebra of a connected Dynkin quiver with at most
three vertices
\cite[Theorems~VII.1.7 and VII.5.10(a)]{AssemSimsonSkowronski}.
Types \(A_1,A_2\), and the linear orientations of \(A_3\)
are Nakayama; the other two orientations of \(A_3\) are the source and sink
orientations in Lemma~\ref{lem:small-faithful-core}(2).

If \(r=1\), representation-finiteness permits at most one loop: two
independent loop classes would give a projective line of pairwise
nonisomorphic nonsplit self-extensions of the simple module.  Thus the
Gabriel quiver is a single vertex with at most one loop, so
\(B\cong\kk[x]/(x^m)\) for some \(m\geq1\) by the bound-quiver
presentation theorem \cite[Theorem~II.3.7]{AssemSimsonSkowronski}, and
\(B\) is Nakayama.

It remains to consider \(r=2\) and \(d(B)=3\).  There is a unique
nonprojective indecomposable torsionless module; call it \(T\).  Put
\(J=\rad B\).  Representation-finiteness also shows that there is at most
one Gabriel arrow with any fixed ordered pair of endpoints: two independent
arrow classes would give a projective line of pairwise nonisomorphic
extensions of length two.

Represent a Gabriel arrow \(i\to j\) by
\(\alpha\in e_iJe_j\setminus e_iJ^2e_j\).  The cyclic ideal \(\alpha B\) is
a submodule of \(P_i\), hence torsionless, and \(\alpha B/\alpha J\) is the
simple module \(S_j\), so \(\alpha B\) is indecomposable with top \(S_j\).
Hence \(\alpha B\) is isomorphic to either \(P_j\) or \(T\).  A loop ideal
is a proper submodule of \(P_i\) with
simple top \(S_i=\top P_i\), so it cannot be projective and must be \(T\).  In
particular, loops cannot occur at both vertices: the two loop ideals would
both be isomorphic to \(T\), although their simple tops would be the
nonisomorphic modules \(S_1\) and \(S_2\).  If no loop occurs, then, since
the quiver has two vertices and at most one arrow for each ordered pair,
each vertex is the target of at most one arrow and the source of at most
one arrow.  Hence \(B\) is Nakayama
\cite[Theorem~V.3.2]{AssemSimsonSkowronski}.

Suppose that there is a loop \(x:1\to1\).  Since \(B\) is connected, there
is at least one arrow between the two vertices.

\medskip
\noindent\textbf{Claim.}\ \emph{There cannot be arrows in both directions
between the two vertices.}

\noindent\emph{Proof.}
Suppose that there are arrows in both directions:
\[
 \begin{tikzcd}[ampersand replacement=\&,column sep=large]
  1 \arrow[loop left,"x"]\arrow[r,shift left,"a"]\&
  2 \arrow[l,shift left,"b"]
 \end{tikzcd}
\]
Since \(x\) is a loop, \(xB\) is a proper submodule of \(P_1\) with
simple top \(S_1\).  Hence \(xB\simeq T\).  The ideal \(aB\) has top
\(S_2\), while \(T\simeq xB\) has top \(S_1\); hence \(aB=P_2\).  The isomorphism
\(P_2\xrightarrow{\sim}aB\subset P_1\) is a proper embedding, so
\(\ell(P_2)<\ell(P_1)\).  If \(bB=P_1\), then
\(P_1\xrightarrow{\sim}bB\subset P_2\) would also be a proper embedding,
giving the opposite strict inequality.  Hence \(bB=T\).  Restricting the
isomorphism \(P_2\xrightarrow{\sim}aB\) to \(bB\) gives
\begin{equation}\label{eq:lollipop-cyclic-ideals}
 abB\simeq bB\simeq T\simeq xB.
\end{equation}

We now construct a representation-infinite factor of \(B\).  Let \(H\) be
the two-sided ideal generated by
\[
 J^3,\quad x^2,\quad xa,\quad bx,\quad ba.
\]
We first prove that \(ab\notin H\).  This strengthens
\eqref{eq:lollipop-cyclic-ideals}, since the ideal \(H\) might contain
\(ab\).  Suppose that \(ab\in H\).  The generators \(xa\), \(bx\), and
\(ba\) do not connect vertex \(1\) to itself, and the product of any
generator with a radical element lies in \(J^3\).  Taking the
\(e_1(-)e_1\)-component modulo \(J^3\) therefore leaves only the generator
\(x^2\), so \(ab-cx^2\in J^3\) for some \(c\in\kk\).  The Gabriel arrows
starting at vertex \(1\) are
\(x\) and \(a\), and a path which starts with \(a\) and returns to vertex
\(1\) must next use \(b\).  Hence every path of length at least three from
vertex \(1\) to itself begins with \(x\) or with \(ab\), and
\[
 e_1J^3e_1\subseteq xJ^2+abJ.
\]
It follows that
\[
 ab=cx^2+xu+abv
 \qquad(u\in J^2,\ v\in J).
\]
Since \(1-v\) is a unit, this implies \(abB\subseteq xB\).  Equation
\eqref{eq:lollipop-cyclic-ideals} gives
\(\ell(abB)=\ell(xB)\), so \(abB=xB\).  This equality would put
\(x\in abB\subseteq J^2\), although the
arrow \(x\) does not lie in \(J^2\).

Thus \(B/H\) is the string algebra with basis
\(e_1,e_2,x,a,b,ab\), in which \(ab\) is the only nonzero product of two
radical basis elements.  The cyclic string \(abx^{-1}\) is not a proper
power, and each of its
positive powers is a string; hence it is a band.  The string--band
classification gives infinitely many pairwise nonisomorphic indecomposable
\(B/H\)-modules \cite{ButlerRingelString}.  Thus \(B/H\), and hence \(B\),
is representation-infinite.  This contradiction shows that arrows cannot
occur in both directions.
\hfill\(\square\)

Since the vertices are joined and arrows cannot occur in both directions,
there is exactly one arrow between them.  Ringel's bijection in
Proposition~\ref{prop:ringel-torsionless} identifies the indecomposable
torsionless modules over \(B\) and over \(B^{\op}\), so
\(d(B^{\op})=d(B)=3\).  Lemma~\ref{lem:small-faithful-core}(3) allows replacing \(B\)
by \(B^{\op}\); we may therefore assume that the only arrow between
distinct vertices is
\(a:1\to2\).  If \(x^2\ne0\), then \(x^2B\) is an indecomposable
torsionless module with simple top \(S_1\).  It is a proper submodule of
\(xB=T\), so it is neither projective nor isomorphic to \(T\), contrary to
the uniqueness of \(T\).  Hence \(x^2=0\).  The only remaining path of
length two is \(xa\); if it is zero then \(B=L_0=\kk Q/(x^2,xa)\), and if
it is nonzero then \(B=L_1=\kk Q/(x^2)\).

Finally, let \(B=\prod_iB_i\) be the product of its nonzero connected
blocks.  Under the decomposition of \(\mod B\) by blocks, one has
\(\mathcal K_*(B)=\coprod_i\mathcal K_*(B_i)\).  Hence
\(d(B)=\sum_i d(B_i)\).
If there is more than one block and \(d(B)\leq3\), then
\(d(B_i)\leq2\) for every \(i\).  The connected classification just proved
shows that every \(B_i\) is Nakayama: the algebras in
Lemma~\ref{lem:small-faithful-core}(2),(3) all have \(d=3\).  Their product is Nakayama as well,
because each indecomposable projective or injective belongs to one block.
\end{proof}

\subsection{Faithful size-four counts for the exceptional algebras}

\begin{proof}[Proof of Lemma~\ref{lem:exceptional-faithful-four}]
\((1)\)
First take the sink orientation
\[
 Q:\quad 1\longrightarrow2\longleftarrow3.
\]
The six positive roots of type \(A_3\) are the dimension vectors of the six
indecomposable \(\kk Q\)-modules
\cite[Theorem~VII.5.10(b),(c)]{AssemSimsonSkowronski}.  These modules are the
three simples, the projective covers \(P_1\) and \(P_3\) of \(S_1\) and
\(S_3\), and the injective envelope \(E\) of \(S_2\).  Their dimension
vectors are
\[
 \underline{\dim}P_1=(1,1,0),\qquad
 \underline{\dim}P_3=(0,1,1),\qquad
 \underline{\dim}E=(1,1,1).
\]
Since \(\kk Q\) is hereditary, the torsionless indecomposables are the
projective modules.  Thus
\[
 \mathcal K_{\sub}(\kk Q)=\{S_2,P_1,P_3\},
 \qquad\text{and dually}\qquad
 \mathcal K_{\quot}(\kk Q)=\{S_1,S_3,E\}.
\]
Adjoining any one of the three remaining indecomposable modules to
\(\mathcal K_{\quot}(\kk Q)\) gives a quotient-closed subcategory, and
adjoining any one to \(\mathcal K_{\sub}(\kk Q)\) gives a submodule-closed
subcategory.  Hence each family has exactly three faithful subcategories
of size four.  Duality gives the source orientation.

\((2)\)
We now compute the two exceptional algebras \(L_0,L_1\).  Both are string
algebras.  Up to inversion, their strings are
\[
 \begin{array}{c|l}
 L_0&e_1,e_2,x,a,x^{-1}a,\\
 L_1&e_1,e_2,x,a,xa,x^{-1}a,a^{-1}xa.
 \end{array}
\]
There are no bands.  Since every indecomposable module over a string algebra
is a string or band module \cite{ButlerRingelString}, these lists show that
\(L_0\) has five indecomposable modules and \(L_1\) has seven.

A direct calculation of quotient and submodule closure in these two finite
lists gives
\[
 f_{\quot}^4(L_e)=f_{\sub}^4(L_e)=2
 \qquad(e=0,1).
\]
Linear duality gives
\[
 f_{\quot}^4(L_e^{\op})=f_{\sub}^4(L_e^{\op})=2
 \qquad(e=0,1).
\]
\end{proof}

\section{Proofs for cosizes three and four}
\label{sec:four-ladders}

This appendix proves Lemma~\ref{lem:rooted-balance} and
Lemma~\ref{lem:four-vertex-ladder}.  Put \(N=|\ind A|\) and
\(\Gamma:=\Gamma(\mod A)\), and let
\(\PP,\mathcal I\subseteq\ind A\) be the sets of indecomposable projective
and injective \(A\)-modules, respectively.
The vertex set of \(\Gamma\) is \(\Gamma_0=\ind A\).  For
\(S\subseteq\Gamma_0\), let \(\Gamma[S]\) denote the full subquiver with
vertex set \(S\).  A nonempty vertex set \(S\) is \emph{strongly connected}
if, for every \(X,Y\in S\), there are directed paths from \(X\) to \(Y\) and
from \(Y\) to \(X\) in \(\Gamma[S]\), where paths of length zero are
allowed.  A \emph{strongly connected
component} is a maximal strongly connected vertex set.
For a vertex set \(W\), write \(W^-\) and \(W^+\) for the sets of immediate
predecessors and successors of vertices of \(W\) in the whole quiver
\(\Gamma\); these sets may meet \(W\).

\begin{proof}[Proof of Lemma~\ref{lem:rooted-balance}]
For \(S\subseteq D\), say that \(S\) is \emph{predecessor-closed} in \(D\)
if no arrow of \(\Gamma[D]\) enters \(S\) from \(D\setminus S\).  The set
\(D\) is not projectively rooted exactly when it has a nonempty
predecessor-closed subset \(S\subseteq D\setminus\PP\).  Indeed, if \(D\)
is not projectively rooted, the members of
\(D\) which are not reachable in \(\Gamma[D]\) from a projective
member form such a subset; conversely, a path from a projective member
of \(D\) would have to enter \(S\) through an arrow from
\(D\setminus S\).

Every minimal nonempty predecessor-closed subset
\(S\subseteq D\setminus\PP\) is strongly connected.  Indeed, choose a
strongly connected component \(C\) of \(\Gamma[S]\) which receives no arrow
from another component.  Such a component exists because contracting the
strongly connected components gives a finite acyclic quiver.  No arrow enters
\(C\) from \(S\setminus C\), and no arrow enters \(S\) from
\(D\setminus S\).  Hence \(C\) is predecessor-closed in \(D\), so the
minimality of \(S\) gives \(C=S\).  Conversely, suppose that
\(S\subseteq D\setminus\PP\) is strongly connected and predecessor-closed
in \(D\).  For every
nonempty proper subset \(T\subsetneq S\), a directed path from
\(S\setminus T\) to \(T\) contains an arrow entering \(T\) from
\(S\setminus T\).  Thus \(T\) is not predecessor-closed in \(D\), and
\(S\) is minimal among the nonempty predecessor-closed subsets of
\(D\setminus\PP\).  If two such minimal subsets meet, their intersection is
again predecessor-closed in \(D\), so the two subsets are equal.  Hence
distinct minimal nonempty predecessor-closed subsets of \(D\setminus\PP\)
are disjoint, and no arrow joins one to another.

Consider the sets \(W\subseteq\ind A\setminus\PP\) for which \(\Gamma[W]\)
is a disjoint union of strongly connected components with no arrows between
different components.  Include \(W=\varnothing\), and let \(c(W)\) be the
number of components, with \(c(\varnothing)=0\).  Fix a set
\(D\subseteq\ind A\) with \(j\) elements.  The subsets \(W\subseteq D\)
for which \(\Gamma[W]\) is a disjoint union of strongly connected
components with no arrows between distinct components and no arrow enters
\(W\) from \(D\setminus W\) are exactly
the unions of some of the minimal nonempty subsets \(S\subseteq D\setminus
\PP\) which are predecessor-closed in \(D\).  If \(b(D)\) is the
number of these minimal subsets, their contribution is
\[
 \sum_{W}(-1)^{c(W)}
 =\sum_{t=0}^{b(D)}\binom{b(D)}t(-1)^t
 =(1-1)^{b(D)}.
\]
Thus the contribution is \(1\) when \(D\) is projectively rooted, because
then \(b(D)=0\), and is \(0\) otherwise.  This is the
inclusion--exclusion cancellation used below.
Summing over all \(D\) and exchanging the two summations gives the
number of projectively rooted sets with \(j\) elements as
\begin{equation}\label{eq:rooted-count-projective}
 \sum_W(-1)^{c(W)}
 \binom{N-|W\cup W^-|}{j-|W|},
\end{equation}
where \(W\) runs over all vertex sets for which \(\Gamma[W]\) is a disjoint
union of strongly connected components with no arrows between distinct
components.  The binomial
coefficient counts the sets \(D\supseteq W\) with \(j\) elements and no
arrow from \(D\setminus W\) to \(W\): the elements of \(D\setminus W\)
must avoid \(W\cup W^-\).  As usual, the binomial coefficient is zero
when its lower argument is outside the interval from zero to its upper
argument.

For every arrow \(X\to Y\) with \(Y\) nonprojective, Auslander--Reiten
translation gives the arrow \(\tau Y\to X\), and this correspondence is
bijective.  Applying it twice shows that, for nonprojective \(X\) and
\(Y\), there is an arrow \(X\to Y\) exactly when there is an arrow
\(\tau X\to\tau Y\).  Put \(\tau W=\{\tau X\mid X\in W\}\).  Thus
\(\tau\) induces an isomorphism
\(\Gamma[W]\simeq\Gamma[\tau W]\), and \(W\mapsto\tau W\)
is a bijection from the sets \(W\) under consideration to the vertex sets
avoiding injectives whose full subquivers are disjoint unions of strongly
connected components with no arrows between distinct components.  This
bijection preserves \(|W|\) and \(c(W)\).  Inclusion--exclusion
therefore computes the number of injectively corooted
sets with \(j\) elements as
\begin{equation}\label{eq:rooted-count-injective}
 \sum_W(-1)^{c(W)}
 \binom{N-|\tau W\cup(\tau W)^+|}{j-|W|},
\end{equation}
with successors in place of predecessors.

It remains to prove
\[
 |W\cup W^-|=|\tau W\cup(\tau W)^+|.
\]
This equality makes the \(W\)-indexed summands in
\eqref{eq:rooted-count-projective} and
\eqref{eq:rooted-count-injective} equal.  The arrows
leaving \(\tau W\) are the arrows \(\tau X\to Z\) with \(X\in W\), and
the correspondence \((Z\to X)\mapsto(\tau X\to Z)\) identifies them
with the arrows entering \(W\); hence \((\tau W)^+=W^-\).  Since
\(|\tau W|=|W|\), it suffices to show
\(|W^-\cap W|=|W^-\cap\tau W|\).  For \(X\in W\),
\[
\begin{aligned}
 X\in W^-
 &\Longleftrightarrow
 \text{there is \(Y\in W\) with \(X\to Y\)},\\
 \tau X\in W^-
 &\Longleftrightarrow
 \text{there is \(Z\in W\) with \(\tau X\to Z\)}\\
 &\Longleftrightarrow
 \text{there is \(Z\in W\) with \(Z\to X\)}.
\end{aligned}
\]
Within a strongly connected component of \(\Gamma[W]\) which contains an
arrow, every vertex occurs both as the initial endpoint and as the terminal
endpoint of an internal arrow.  In a component containing no arrow, no
vertex occurs as either kind of endpoint.  Writing \(W'\) for the union of
the components containing an arrow, we obtain \(W^-\cap W=W'\).  The
equivalence
\[
 \tau X\in W^-\quad\Longleftrightarrow\quad
 \text{there is \(Z\in W\) with \(Z\to X\)}
\]
shows that \(\tau X\in W^-\) if and only if \(X\in W'\).  Therefore
\(W^-\cap\tau W=\tau W'\), and both intersections have cardinality
\(|W'|\).
\end{proof}

The rest of this appendix proves
Lemma~\ref{lem:four-vertex-ladder}.  Fix a projectively rooted set
\(D\subseteq\ind A\) with \(|D|\leq4\).  We use the recursion from
Construction~\ref{con:factor-ladder}.  For \(X\in D\), this gives the
sequence \(\theta^D_0X,\theta^D_1X,\theta^D_2X,\ldots\) defined by
\eqref{eq:factor-ladder}.  For \(X\in D\), the element
\(\theta_DX\) is the sum, with multiplicities, of the indecomposable
summands belonging to \(D\) of \(\theta X\).  One puts \(\tau_DX=\tau X\) when
\(X\) is nonprojective, \(\tau X\in D\), and \(\theta_DX\ne0\), and
\(\tau_DX=0\) otherwise.  Both operators extend additively to the free
abelian group on \(D\), and \(v_+\) replaces every negative coefficient
of \(v\) by zero.  Then \(\theta^D_0X=X\), \(\theta^D_1X=\theta_DX\), and
\eqref{eq:factor-ladder} is the recurrence
\[
 \theta^D_nX=
 \bigl(\theta_D\theta^D_{n-1}X-\tau_D\theta^D_{n-2}X\bigr)_+
 \qquad(n\geq2).
\]
We also use the following features of the recursion: its
terms are supported on \(D\), it is determined by two consecutive terms,
and a nonzero repeated ordered pair
\((\theta^D_{n-1}X,\theta^D_nX)\) makes the sequence periodic.  Such
periodicity is impossible by
Proposition~\ref{prop:iyama-factor-category}.  We say that the sequence has
a projective summand if some \(\theta^D_nX\) has an indecomposable summand
belonging to \(\PP\cap D\).  We will study the case in which there is an
index \(n_0\) with \(\theta_{n_0}^D X=0\), while none of
\(\theta_0^D X,\ldots,\theta_{n_0-1}^D X\) has such a summand.

For the dual count, reverse the translation quiver by replacing every arrow
\(X\to Y\) with \(Y\to X\), replacing \(\tau\) with \(\tau^{-1}\) whenever
\(\tau^{-1}\) is defined, and exchanging \(\PP\) with \(\mathcal I\).  Under
this reversal, the factor ladder is the reverse factor ladder of
Construction~\ref{con:reverse-factor-ladder}.

\begin{definition}\label{def:killed-set}
A projectively rooted set \(D\) is \emph{killed} if, for some
\(X\in D\), the factor ladder starting at \(X\) becomes zero before any of
its nonzero terms has a summand in \(\PP\cap D\).  An injectively corooted
set \(D\) is \emph{killed} if, for some \(X\in D\), the reverse factor
ladder starting at \(X\) becomes zero before any of its nonzero terms has a
summand in \(\mathcal I\cap D\).
\end{definition}

\subsection{Auslander--Reiten constraints for factor ladders on at most four
vertices}

For a vertex set \(S\subseteq\Gamma_0\), write
\(S_D^\pm=S^\pm\cap D\).  We abbreviate \(\{x\}^\pm\) and
\(\{x\}_D^\pm\) to \(x^\pm\) and \(x_D^\pm\), respectively, and we write
\(\ell X\) for the composition length of \(X\).

By Corollary~\ref{cor:ar-middle-terms-squarefree}, the middle term of an
almost split sequence has no repeated indecomposable summand.
Lemma~\ref{lem:ar-mesh-correspondence} pairs each arrow \(X\to Y\), with \(Y\)
nonprojective, with the arrow \(\tau Y\to X\); we apply it in both
directions.  Here a path
\(X_0\to\cdots\to X_t\) is \emph{sectional} when
\(\tau X_{i+1}\ne X_{i-1}\) for every interior index \(i\); a
\emph{sectional bypass} of an arrow is a second sectional path with the same
endpoints.

\begin{lemma}[{\cite[Proposition~2]{CBHR}}]
\label{lem:four-no-sectional-bypass}
No arrow in \(\Gamma\) admits a sectional bypass.
\end{lemma}

\begin{lemma}\label{lem:four-no-translation-arrow}
There is no arrow \(\tau X\to X\).
\end{lemma}

\begin{proof}
By Lemma~\ref{lem:ar-mesh-correspondence}, an arrow \(\tau X\to X\) would
give a loop at \(\tau X\).  The Auslander--Reiten quiver has no loops.
\end{proof}

The arguments below repeatedly use pairs of distinct vertices joined by
arrows in both directions; we write \(Y\rightleftarrows Z\) for such a
pair.  The next lemma restricts these pairs.

\begin{lemma}\label{lem:two-cycle-local}
Let \(Y\) and \(Z\) be distinct vertices for which both arrows
\(Y\to Z\) and \(Z\to Y\) occur in \(\Gamma\).
\begin{enumerate}
\item One of \(Y,Z\) is fixed by \(\tau\).
\item If one member is projective, then it is also injective.
\end{enumerate}
\end{lemma}

\begin{proof}
\((1)\)
First suppose that \(P\rightleftarrows Z\), where \(P\) is projective.
The other vertex cannot also be projective.  Indeed, an irreducible arrow
between indecomposable projectives \(Q\to P\) makes \(Q\) a summand of
\(\rad P\), and hence \(\ell Q<\ell P\).  Arrows in both directions would
give both strict inequalities.  Thus \(Z\) is nonprojective.
Lemma~\ref{lem:ar-mesh-correspondence} applied to \(P\to Z\) gives
\(\tau Z\to P\).  If
\(Z\ne\tau Z\), both are summands of \(\rad P\), while \(P\) is a summand
of the middle term \(\theta Z\).  Hence
\[
 \ell(\rad P)\geq\ell Z+\ell(\tau Z)
 =\ell(\theta Z)\geq\ell P,
\]
a contradiction.  Thus \(\tau Z=Z\).

For a general pair \(Y\rightleftarrows Z\), applying the arrow
correspondence to each of the two arrows twice gives the pair
\(\tau Y\rightleftarrows\tau Z\), as long as all four vertices involved
are nonprojective; the inverse correspondence likewise produces
\(\tau^{-1}Y\rightleftarrows\tau^{-1}Z\) from noninjective modules.
Iterate in both directions until one of the two modules belongs to
\(\PP\) or \(\mathcal I\).  If a pair meets \(\PP\), the projective case
proved in the first paragraph gives a module fixed by \(\tau\); the case in which a pair meets
\(\mathcal I\) is dual.  Suppose instead that no pair meets
\(\PP\cup\mathcal I\).  The pairs
\[
 Y_i\rightleftarrows Z_i,\qquad Y_i=\tau^iY,\quad Z_i=\tau^iZ,
\]
then form a finite periodic family, since \(\Gamma\) is finite.
The meshes at \(Y_i\) and \(Z_i\) give
\[
\begin{aligned}
 \ell Y_i+\ell Y_{i+1}&\geq\ell Z_i+\ell Z_{i+1},\\
 \ell Z_i+\ell Z_{i+1}&\geq\ell Y_i+\ell Y_{i+1}.
\end{aligned}
\]
Equality holds in both inequalities.  Hence
\[
 E_{Y_i}\simeq Z_i\oplus Z_{i+1},
 \qquad
 E_{Z_i}\simeq Y_i\oplus Y_{i+1}.
\]
The finite periodic family is therefore
closed under predecessors and successors, so it is a whole
Auslander--Reiten component without a projective.  This is impossible: a
module \(M\) in the component admits an indecomposable projective \(P\) with
\(\Hom_A(P,M)\ne0\), and a path of irreducible maps joins \(P\) to \(M\)
\cite[Corollary~IV.5.6]{AssemSimsonSkowronski}.  Thus
one vertex of every pair \(Y\rightleftarrows Z\) is fixed by \(\tau\).

\((2)\)
Finally, suppose that \(P\rightleftarrows Z\) and that \(P\) is not
injective.  We have \(\tau Z=Z\).  Choose
\[
 P=P_0,P_1,\ldots,P_s=I,\qquad
 \tau P_j=P_{j-1}\quad(1\leq j\leq s),
\]
with \(I\) injective.  Repeated use of
Lemma~\ref{lem:ar-mesh-correspondence} gives
\[
 P_j\rightleftarrows Z\qquad(0\leq j\leq s).
\]
The arrows incident with \(P=P_0\) and \(I=P_s\) give
\(\ell P>\ell Z\) and \(\ell I>\ell Z\).  Both \(P\) and \(I\) are
summands of the middle term \(\theta Z\) of
\(0\to Z\to \theta Z\to Z\to0\).  Therefore
\[
 2\ell Z=\ell(\theta Z)\geq\ell P+\ell I>2\ell Z,
\]
a contradiction.
\end{proof}

\subsection{Classification of factor ladders that vanish without a projective
summand}

Fix a projectively rooted set \(D\) with at most four vertices.  We classify
the sequences \((\theta_n^D X)_{n\geq0}\) for which there is an index
\(n_0\) such that \(\theta_{n_0}^D X=0\) and none of
\(\theta_0^D X,\ldots,\theta_{n_0-1}^D X\) has a summand in \(\PP\cap D\).

\begin{definition}\label{def:four-hooks}
For a vertex set \(D\subseteq\Gamma_0\), put
\begin{equation}\label{eq:hook}
 \mathcal H(D)=\left\{(a,u,b)\in D^3\ \middle|\
 \begin{gathered}
 a\longrightarrow u\longrightarrow b,\qquad \tau b=a,\qquad
 u\notin\PP,\\
 u_D^-=\{a\},\qquad b_D^-=\{u\}
 \end{gathered}\right\}.
\end{equation}
We call the elements of \(\mathcal H(D)\) the \emph{hooks} of \(D\).
A vertex set \(D\) is \emph{hookless} if \(\mathcal H(D)=\varnothing\).
\end{definition}

The displayed conditions restrict \(u\) and \(b\) only; no condition is
imposed on the first vertex \(a\).
For \((a,u,b)\in\mathcal H(D)\), the sequence starting at
\(b\) is \(b,u,0\): the singleton predecessor sets give
\(\theta_1^Db=u\) and \(\theta_Du=a\), while \(\tau_Db=\tau b=a\), so
\(\theta_2^Db=(a-a)_+=0\).  Here \(u\notin\PP\) ensures that neither
nonzero term has a summand in \(\PP\cap D\).

\begin{lemma}\label{lem:small-support}
Let \(D\) be projectively rooted.  The following statements hold.
\begin{enumerate}
\item If \(|D|\leq2\) and \(\theta_{n_0}^D X=0\) for some \(n_0\), then
\(\theta_n^D X\) has a projective summand for some \(n<n_0\).
\item If \(|D|=3\), then \(D\) is killed if and only if
\(\mathcal H(D)=\{(P,u,b)\}\) with \(P\in\PP\).  In this case the
factor ladder starting at \(b\) is \(b,u,0\).
\end{enumerate}
\end{lemma}

\begin{proof}
All distances in this proof are directed distances in \(\Gamma[D]\).

(1)
If \(X\in\PP\), then \(X\) already occurs in \(\theta_0^D X\).  Since
\(D\) is projectively rooted, a nonprojective \(X\) lies on a path in
\(\Gamma[D]\) from a member of \(\PP\cap D\); with at most two vertices
this path is a single arrow \(P\to X\), so \(P\) is a summand of
\(\theta_1^D X\).

(2)
Suppose that \(D\) is killed.  Since every vertex of \(D\) lies on a path
in \(\Gamma[D]\) starting at a member of \(\PP\cap D\), this full
subquiver contains a path
\[
 P\longrightarrow u\longrightarrow b
\]
with \(\PP\cap D=\{P\}\).  Indeed, if \(D\) contained two projective
modules, it would contain at most one nonprojective module \(X\), and
\(\theta_1^D X\) would have a projective summand.  A starting vertex whose
factor ladder is killed is not \(P\), because \(\theta_0^DP=P\), and is
not \(u\), because \(P\) is a summand of \(\theta_1^Du\).  Hence the
starting vertex is \(b\).

Since no term in the sequence starting at \(b\) has \(P\) as a summand,
one has \(b_D^-=\{u\}\).  Since \(P\to u\), the projective \(P\) is a
summand of \(\theta_Du\); in
\(\theta_2^Db=(\theta_Du-\tau_Db)_+\) the only subtracted term is
\(\tau_Db\), so absence of \(P\) forces \(\tau b=P\).  If \(b\to u\),
then \(u\rightleftarrows b\).  Since \(b\) is not \(\tau\)-fixed,
Lemma~\ref{lem:two-cycle-local} gives \(\tau u=u\).  The arrow
correspondence then gives \(u\to P\), so the projective, noninjective
module \(P=\tau b\) occurs in the pair \(P\rightleftarrows u\), contrary
to Lemma~\ref{lem:two-cycle-local}.  Hence \(u_D^-=\{P\}\), and
\(\mathcal H(D)=\{(P,u,b)\}\).  Conversely, this hook gives the factor
ladder \(b,u,0\).
\end{proof}

We next name the three four-vertex configurations that occur in the
classification.

\begin{definition}\label{def:four-support-configurations}
Let \(D\subseteq\Gamma_0\) have four vertices.  We call \(D\) a
\emph{double hook} if it admits a labeling
\(D=\{P,u_1,u_2,u_3\}\), with \(P\in\PP\), such that
\[
 \mathcal H(D)=\{(P,u_1,u_2),(u_1,u_2,u_3)\}.
\]
In this case the two hooks are consecutive:
\begin{equation}\label{eq:double-hook}
 P\longrightarrow u_1\longrightarrow u_2\longrightarrow u_3,
 \qquad \tau u_2=P,\quad\tau u_3=u_1,
\end{equation}
and \((u_1)_D^-=\{P\}\),
\((u_2)_D^-=\{u_1\}\), and \((u_3)_D^-=\{u_2\}\).

Suppose that \(D\) is hookless and \(\PP\cap D=\{P\}\).
\begin{enumerate}
\item[\textup{(F)}] We call \(D\) of \emph{type \({\rm F}\)} if it
admits a labeling \(D=\{P,a,c,z\}\) such that
\[
\begin{tikzcd}[ampersand replacement=\&,column sep=large,baseline=-.5ex]
 P\arrow[r]\&a\arrow[r,shift left]\&
 c\arrow[l,shift left]\arrow[r,shift left]\&
 z\arrow[l,shift left]
\end{tikzcd}
\quad
 \tau z=a,\qquad\tau c=c,
\]
and
\[
 \tau a=z\ \text{ or }\ \tau a\notin D,\qquad
 P\nrightarrow z,\qquad
 z_D^-=\{c\},\qquad c_D^-=\{a,z\}.
\]
No condition is imposed on additional arrows ending at \(a\) or \(P\);
in particular, \(z\to P\) may be present or absent.

\item[\textup{(T)}] We call \(D\) of \emph{type \({\rm T}\)} if it
admits a labeling \(D=\{P,A_1,A_2,x\}\) such that
\[
\begin{tikzcd}[ampersand replacement=\&,column sep=large,row sep=large,
 baseline=-.5ex]
 x\arrow[r]\&A_2\arrow[d]\&P\arrow[l]\\
 \&A_1\arrow[ul]\arrow[ur]\&
\end{tikzcd}
\]
and
\[
 \tau A_1=P,\qquad\tau A_2=A_1,\qquad\tau x\notin D,
\]
\[
 (A_1)_D^-=\{A_2\},\qquad
 (A_2)_D^-=\{P,x\},\qquad x_D^-=\{A_1\}.
\]
No condition is imposed on additional arrows ending at \(P\).
\end{enumerate}
The label \({\rm F}\) refers to the \(\tau\)-fixed vertex \(c\), and
\({\rm T}\) to the triangle \(A_1\to P\to A_2\to A_1\).
\end{definition}

To separate the finite computation from its application to
Auslander--Reiten quivers, we encode only the arrow relation, the projective
vertices, and the partial translation.

\begin{definition}\label{def:four-vertex-test-pattern}
An \emph{admissible four-vertex pattern} consists of a set \(V\) with four
elements,
a loop-free arrow relation \(E\subseteq V\times V\), a nonempty subset
\(\Pi\subseteq V\), and a map
\[
 t:V\setminus\Pi\longrightarrow V\sqcup\{*\}.
\]
For \(w\in V\setminus\Pi\), the value \(t(w)=*\) records that the translate
is outside \(V\).  The following conditions are imposed.
\begin{enumerate}
\item Every vertex is reachable from \(\Pi\) by a directed path in \(E\).
\item The values of \(t\) which belong to \(V\) are distinct.
\item If \(t(w)=v\in V\), then
\[
 w_E^-=v_E^+,
 \qquad (v,w)\notin E.
\]
\item The relation \(E\) contains no transitive triangle.
\item If two distinct vertices form a two-cycle, one of them is fixed by
\(t\).  If a member of the two-cycle belongs to \(\Pi\), it is not in the
image of \(t\).
\item For each \(w\in V\), the factor-ladder recurrence determined by
\((E,\Pi,t)\) either has a summand in \(\Pi\) or reaches zero without
repeating a nonzero ordered pair of consecutive terms.
\end{enumerate}
For \(x\in V\), the symbols \(x_E^-\) and \(x_E^+\) denote the predecessor
and successor sets computed in \(E\).  The pattern is \emph{killed} if the recurrence for
some starting vertex reaches zero before any nonzero term has a summand in
\(\Pi\).
\end{definition}

For an admissible four-vertex pattern, hooks, double hooks, and types
\({\rm F}\) and \({\rm T}\) are defined by Definitions~\ref{def:four-hooks} and
\ref{def:four-support-configurations} after replacing \(D\) by \(V\),
membership in \(\PP\) by membership in \(\Pi\), \(\tau\) by \(t\), and
the arrow relation of \(\Gamma[D]\) by \(E\).  In this translation,
\(\tau a\notin D\) means \(t(a)=*\).

\begin{proposition}\label{prop:four-test-pattern-classification}
Every admissible four-vertex pattern satisfies the following statements.
\begin{enumerate}
\item It has at most two hooks, and it has two hooks if and only if it is a
double hook.
\item If it is hookless, then it is killed if and only if it has type
\({\rm F}\) or type \({\rm T}\).
\end{enumerate}
\end{proposition}

\begin{proof}
A Python script performs the following exhaustive enumeration.
There are \(2^{12}\) loop-free arrow relations on a labeled four-element
set, \(15\) nonempty choices of \(\Pi\), and finitely many maps \(t\).
For every choice, the script tests conditions~(1)--(6).  In doing so, it
evaluates the recurrence \eqref{eq:factor-ladder} with integer coefficients
for every starting vertex.  It stops when a term meets \(\Pi\) or becomes
zero; repetition of a nonzero ordered pair is detected exactly and is
excluded by condition~(6).

The script finds \(6933\) labeled admissible patterns, in
\(329\) isomorphism classes.  Their classification is as follows.
\[
\begin{array}{c|r|r}
 &\text{labeled patterns}&\text{isomorphism classes}\\ \hline
\text{killed}&1272&53\\
\text{hookless and killed}&96&4\\
\text{type \({\rm F}\)}&72&3\\
\text{type \({\rm T}\)}&24&1\\
\text{double hook}&96&4
\end{array}
\]
Every admissible pattern is killed if and only if it has a hook or has type
\({\rm F}\) or type \({\rm T}\).  Types
\({\rm F}\) and \({\rm T}\) are disjoint and have unique labelings.
Every admissible pattern has at most two hooks, and two hooks occur exactly in
a double hook.

The Python script used for this verification is available at
\begin{center}
\footnotesize
\url{https://github.com/haruhisa-enomoto/quotient-submodule-equidistribution/blob/main/verification/four_vertex_patterns.py}.
\end{center}
\end{proof}

\begin{proposition}\label{prop:four-support-classification}
Let \(D\) be a projectively rooted set with four vertices.  The following
statements hold.
\begin{enumerate}
\item The set \(D\) has at most two hooks, and it has two hooks if and only
if it is a double hook.
\item If \(D\) is hookless, then \(D\) is killed if and only if it has type
\({\rm F}\) or type \({\rm T}\).
\end{enumerate}
\end{proposition}

\begin{proof}
Let \(E\) be the arrow relation of \(\Gamma[D]\), put
\(\Pi=\PP\cap D\), and define
\[
 t:D\setminus\Pi\longrightarrow D\sqcup\{*\},\qquad
 t(X)=
 \begin{cases}
  \tau X&\text{if \(\tau X\in D\)},\\
  *&\text{if \(\tau X\notin D\)}.
 \end{cases}
\]
We verify that \((D,E,\Pi,t)\) is an admissible four-vertex pattern.
Since \(\Gamma\) has no loops, \(E\) is loop-free.  Since \(D\) is
projectively rooted, condition~(1) holds, and the injectivity of \(\tau\) gives
condition~(2).
Lemma~\ref{lem:ar-mesh-correspondence} gives the equality of predecessor and
successor sets in condition~(3), while
Lemma~\ref{lem:four-no-translation-arrow} gives
\((t(w),w)\notin E\) when \(t(w)\in D\).  Any transitive triangle would
give either a sectional
bypass or an arrow \(\tau X\to X\), so
Lemmas~\ref{lem:four-no-sectional-bypass} and
\ref{lem:four-no-translation-arrow} give condition~(4).
Lemma~\ref{lem:two-cycle-local} gives condition~(5): a projective member of
a two-cycle is injective and therefore does not belong to the image of
\(\tau\).  Proposition~\ref{prop:iyama-factor-category}(2) gives
\(\theta_n^D X=0\) for all sufficiently large \(n\), which is
condition~(6).

Corollary~\ref{cor:ar-middle-terms-squarefree} shows that the arrow relation
records the middle terms of the almost split sequences with multiplicity one.
Hence the factor-ladder recurrence for \((D,E,\Pi,t)\) is the recurrence for
\(D\).
Proposition~\ref{prop:four-test-pattern-classification} now gives both
statements.
\end{proof}

Let \(K_{4,\quot}\) be the number of projectively rooted sets of four vertices
for which some sequence becomes zero without a projective summand; this
is the quantity the two sides of Lemma~\ref{lem:four-vertex-ladder}
compare at \(j=4\).  Put
\[
 W_{4,\quot}=
 \sum_{\substack{D\text{ projectively rooted}\\|D|=4}}
 |\mathcal H(D)|.
\]
Let \(B_{4,\quot}\) be the number of double hooks.  Let \(F_{4,\quot}\) and
\(T_{4,\quot}\) be the numbers of sets of types \({\rm F}\) and \({\rm T}\),
respectively.
Proposition~\ref{prop:four-support-classification} gives
\begin{equation}\label{eq:killed-packets}
 K_{4,\quot}=W_{4,\quot}-B_{4,\quot}+F_{4,\quot}+T_{4,\quot}.
\end{equation}
Indeed, a set with \(|\mathcal H(D)|=1\) contributes once to \(W_{4,\quot}\), while
a double hook contributes \(2-1=1\) to
\(W_{4,\quot}-B_{4,\quot}\), and Proposition~\ref{prop:four-support-classification}(1) gives
\(|\mathcal H(D)|\leq2\).  A set with \(\mathcal H(D)=\varnothing\) for
which some sequence becomes zero without a projective summand satisfies
exactly one of types \({\rm F}\) and \({\rm T}\), and each such set is
counted once, because the labeling of its four vertices is determined
by the configuration: in type \({\rm F}\), the unique projective \(P\)
and the unique
\(\tau\)-fixed nonprojective \(c\) distinguish two vertices,
while \(P\to a\) and \(P\nrightarrow z\) distinguish the remaining two.  In
type \({\rm T}\), the unique projective determines
\(A_1=\tau^{-1}P\) and \(A_2=\tau^{-1}A_1\), leaving \(x\).  The two types
are disjoint because type \({\rm F}\) has a \(\tau\)-fixed nonprojective
vertex and type \({\rm T}\) does not.  The next subsection matches
\(W_{4,\quot}\), \(B_{4,\quot}\), \(F_{4,\quot}\), and \(T_{4,\quot}\) with their counterparts in the
reversed translation quiver by counting orbits of arrows.

We now compare the configurations in Lemma~\ref{lem:small-support} and
Proposition~\ref{prop:four-support-classification} under reversal.

The subscript \(\quot\) refers to \(\Gamma\), and the subscript \(\sub\)
refers to the translation quiver obtained by reversing \(\Gamma\).  We call
these the \emph{quotient side} and the \emph{submodule side}, respectively.
Thus the modules in \(\mathcal I\) play the role of projective modules on the
submodule side.  We first organize the arrows into translation orbits.

\begin{construction}\label{con:alternative-arrow-orbits}
For an arrow \(\alpha:c\to d\) with \(d\) nonprojective, put
\(\sigma\alpha=(\tau d\to c)\).
Its inverse sends an arrow \(c\to d\) with \(c\) noninjective to
\(d\to\tau^{-1}c\).  Starting from an arrow, iterate \(\sigma\) and its
inverse for as long as the required translations are defined.  We call the resulting maximal
indexed sequence a \emph{maximal \(\sigma\)-orbit}; it is either a cycle or a
finite chain.  Each indexed arrow in such a sequence is an \emph{arrow
occurrence}.  Thus different positions remain different occurrences even
if the orbit revisits a vertex.  We write a finite chain as
\begin{equation}\label{eq:alternative-chain}
 H=(y_0,\ldots,y_L),\qquad
 \beta_e=(y_{e+1}\to y_e),\qquad
 \tau y_e=y_{e+2}.
\end{equation}
Thus \(y_1\) is injective and \(y_{L-1}\) is projective.  In the reversed
translation quiver, this chain is read as \((y_L,\ldots,y_0)\).

If a second orbit is written as
\[
 C=(x_0,\ldots,x_M),\qquad \tau x_f=x_{f+2},
\]
we call a pair \((f,e)\) with \(x_f=y_e\) an \emph{indexed contact} of
\(C\) with \(H\).  Equalities at different positions give different indexed
contacts.
\end{construction}

\subsection{Balance for three omitted vertices}

\begin{lemma}\label{lem:alternative-three-killed}
The numbers of projectively rooted killed sets with three vertices and
injectively corooted killed sets with three vertices are equal.
\end{lemma}

\begin{proof}
By Lemma~\ref{lem:small-support}(2), a projectively rooted set with three
vertices is killed precisely when it has the form
\begin{equation}\label{eq:alternative-wall-hook}
 P\longrightarrow u\longrightarrow b,\qquad
 \tau b=P,\qquad
 u_D^-=\{P\},\quad b_D^-=\{u\}.
\end{equation}
For a chain \(H\) in \eqref{eq:alternative-chain} with \(L\geq3\), we call
\((y_{L-1},y_{L-2},y_{L-3})\) its \emph{projective-end triple}.  After
reversal, the projective-end triple is \((y_1,y_2,y_3)\).
Both triples satisfy \eqref{eq:alternative-wall-hook} on their respective
sides.  We verify this for
\((P,u,b)=(y_{L-1},y_{L-2},y_{L-3})\); the other end is dual.  The vertex
\(u\) is nonprojective because \(\tau u=y_L\) exists.  The arrow
\(P\to b\) is impossible because \(P=\tau b\).  If \(b\to u\), then
\(u\rightleftarrows b\).  Since \(\tau b=P\ne b\),
Lemma~\ref{lem:two-cycle-local} gives \(\tau u=u\).  The arrow
correspondence applied to \(P\to u\) then gives \(u\to P\), so
\(P\rightleftarrows u\).  This is impossible because the projective module
\(P=\tau b\) is not injective, again by
Lemma~\ref{lem:two-cycle-local}.  Hence the predecessor sets induced on the
triple are \(\{P\}\) at \(u\) and \(\{u\}\) at \(b\).

Conversely, if \((P,u,b)\) satisfies
\eqref{eq:alternative-wall-hook}, apply \(\sigma\) to the arrow \(P\to u\).
The resulting arrow \(\tau u\to P\) is the last arrow of its chain because
\(P\) is projective.  Hence \(P\to u\) immediately precedes it in the
\(\sigma\)-orbit.  If these
positions are \(y_{L-1}=P\) and \(y_{L-2}=u\), then
\(\tau y_{L-3}=P=\tau b\), so the injectivity of \(\tau\) gives
\(y_{L-3}=b\).  Thus reading a finite chain from its two ends gives a
bijection between the killed triples on the two sides.
\end{proof}

\subsection{The signed count for four omitted vertices}

Recall that \(K_{4,\quot}\) denotes the number of projectively rooted killed
sets with four vertices and that
\begin{equation}\label{eq:alternative-killed-packets}
 K_{4,\quot}=W_{4,\quot}-B_{4,\quot}+F_{4,\quot}+T_{4,\quot}.
\end{equation}
Here \(W_{4,\quot}\) counts pairs \((D,h)\), where \(h\in\mathcal H(D)\), the
term \(B_{4,\quot}\) corrects the double hooks, and
\(F_{4,\quot},T_{4,\quot}\) count the hookless configurations of types
\({\rm F}\) and \({\rm T}\) in
Definition~\ref{def:four-support-configurations}.

\begin{definition}\label{def:admissible-hook-triple}
A triple \(h=(a,u,b)\) is \emph{admissible} if
\[
 a\longrightarrow u\longrightarrow b,\qquad \tau b=a,
\]
the three vertices are distinct, \(u\notin\PP\), and the predecessor sets
induced on \(\{a,u,b\}\) are \(\{a\}\) at \(u\) and \(\{u\}\) at \(b\).
\end{definition}

Fix such a triple and put
\(A_h=u^-\) and \(B_h=b^-=a^+\).
The last equality follows from Lemma~\ref{lem:ar-mesh-correspondence}
applied at \(b\), since
\(\tau b=a\).  For \(D=\{a,u,b,z\}\), the condition
\(h\in\mathcal H(D)\) says exactly that \(z\notin A_h\cup B_h\).
We use the disjointness relations
\[
 A_h\cap B_h=\varnothing,
 \qquad u^+\cap b^+=\varnothing.
\]
A vertex in either intersection would give a sectional bypass, contrary to
Lemma~\ref{lem:four-no-sectional-bypass}.
If \(a\) is projective, the possible fourth vertices are
\[
 \bigl(\PP\setminus(A_h\cup B_h)\bigr)\ \sqcup\
 \bigl((u^+\cup b^+)
 \setminus(\PP\cup A_h\cup B_h\cup\{b\})\bigr).
\]
If \(a\) is not projective, they are the vertices in
\[
 (\PP\cap a^-)\setminus(A_h\cup B_h).
\]
The resulting fourth-vertex count has the following signed form.  If
\(a\in\PP\), put
\[
\begin{aligned}
 m_1(h)&=|\PP\cap A_h|,&
 m_2(h)&=|\PP\cap B_h|,\\
 m_3(h)&=|(u^+\cup b^+)
 \setminus(\PP\cup A_h\cup B_h\cup\{b\})|,&
 m_4(h)&=0.
\end{aligned}
\]
If \(a\notin\PP\), put \(m_1(h)=m_2(h)=m_3(h)=0\) and
\[
 m_4(h)=|(\PP\cap a^-)\setminus(A_h\cup B_h)|.
\]
Then the number of fourth vertices is \(|\PP|+e(h)\) for
\(a\in\PP\), and \(e(h)\) for \(a\notin\PP\), where
\begin{equation}\label{eq:alternative-signed-sum}
 e(h)=-m_1(h)-m_2(h)+m_3(h)+m_4(h).
\end{equation}

Associate an arrow occurrence with every term in
\eqref{eq:alternative-signed-sum} as follows:
\[
\begin{array}{c|c|c}
\text{term}&\text{condition on the fourth vertex }z&\text{arrow occurrence}\\ \hline
m_1&z\in\PP\cap A_h&z\to u\\
m_2&z\in\PP\cap B_h&z\to b\\
m_3&z\in u^+\cup b^+&u\to z\text{ or }b\to z\\
m_4&z\in\PP\cap a^-&z\to a.
\end{array}
\]
Let \(H\) be the maximal \(\sigma\)-orbit containing \(a\to u\); we call
\(H\) the \emph{hook orbit} of \(h\).  Assign each signed term to
\((H,C)\), where \(C\) is the maximal \(\sigma\)-orbit containing its
associated arrow.  For \(*\in\{\quot,\sub\}\), write
\(e_*(H,C)\) for the sum of the signs assigned to this pair on the \(*\)-side,
and let \(m_{j,*}(H,C)\) be the number of its \(m_j\)-terms.  Thus
\[
 e_*(H,C)=-m_{1,*}(H,C)-m_{2,*}(H,C)
 +m_{3,*}(H,C)+m_{4,*}(H,C).
\]
Let \(B_*(H)\) be the number of double-hook configurations whose two
hooks occur consecutively in the chain \(H\); then
\(B_{4,*}=\sum_HB_*(H)\).  Put
\(\PP_{\quot}=\PP\) and \(\PP_{\sub}=\mathcal I\), and let
\(a_0(*)\) be the number of admissible triples whose first vertex
belongs to \(\PP_*\).  Then
\begin{equation}\label{eq:alternative-wall-decomposition}
 W_{4,*}=|\PP_*|a_0(*)
 +\sum_{H,C}e_*(H,C).
\end{equation}
Lemma~\ref{lem:alternative-three-killed} gives
\(a_0(\quot)=a_0(\sub)\), and
\(|\PP|=|\mathcal I|\).  It remains to compare the signed sums.

\subsection{Balance of indexed contacts for chain hook orbits}

Let \(H=(y_0,\ldots,y_L)\) be a chain hook orbit; then \(L\geq3\).
A contribution of
type \(m_3\) counted before imposing the conditions
\(z\nrightarrow u\) and \(z\nrightarrow b\) will be called a
\emph{preliminary \(m_3\)-term}.  We first compare the signed boundary sums
before these two conditions are imposed.

Suppose that \(C=(x_0,\ldots,x_M)\) is a chain orbit, and put
\(\zeta_{f,e}=[\,x_f=y_e\,]\).
For \(L\geq5\), define
\begin{align}
 U_{\quot}(H,C)
 &=\sum_{e=2}^{L-4}\zeta_{M-2,e}
   +\sum_{f=2}^{M-2}(\zeta_{f,L-3}+\zeta_{f,L-2}),
 \label{eq:alternative-raw-q}\\
 U_{\sub}(H,C)
 &=\sum_{e=4}^{L-2}\zeta_{2,e}
   +\sum_{f=2}^{M-2}(\zeta_{f,2}+\zeta_{f,3}).
 \label{eq:alternative-raw-s}
\end{align}
An empty sum is zero.  The first sums encode the
\(m_4-m_1-m_2\) contribution, and the second sums encode the preliminary
\(m_3\)-terms.  For \(L=3\), define the two boundary sums by
\[
\begin{split}
 U_{\quot}(H,C)
 &=-\zeta_{M-2,0}-\zeta_{M-2,1}
   +\sum_{f=1}^{M-1}(\zeta_{f,0}+\zeta_{f,1}),\\
 U_{\sub}(H,C)
 &=-\zeta_{2,2}-\zeta_{2,3}
   +\sum_{f=1}^{M-1}(\zeta_{f,2}+\zeta_{f,3}),
\end{split}
\]
and for \(L=4\), define them by
\[
\begin{split}
 U_{\quot}(H,C)
 &=-\zeta_{M-2,1}
   +\sum_{f=1}^{M-1}(\zeta_{f,1}+\zeta_{f,2}),\\
 U_{\sub}(H,C)
 &=-\zeta_{2,3}
   +\sum_{f=1}^{M-1}(\zeta_{f,2}+\zeta_{f,3}).
\end{split}
\]
Missing indices are interpreted as zero.  If \(C\) is cyclic, the boundary
sum consists only of the preliminary \(m_3\)-terms, with the \(f\)-coordinate
read modulo the occurrence period of \(C\).

\begin{lemma}\label{lem:alternative-raw-contact-balance}
For every chain hook orbit \(H\) and every orbit \(C\),
\begin{equation}\label{eq:alternative-raw-balance}
 U_{\quot}(H,C)=U_{\sub}(H,C).
\end{equation}
\end{lemma}

\begin{proof}
Translation and its injectivity give
\begin{equation}\label{eq:alternative-contact-shift}
 \zeta_{f+2,e+2}=\zeta_{f,e}
\end{equation}
whenever both terms are defined.  For \(L\geq5\), consider the nonzero
indexed contacts in
\[
 2\leq f\leq M-2,\qquad 2\leq e\leq L-2.
\]
They form disjoint paths under \((f,e)\mapsto(f+2,e+2)\).  Since \(x_1\)
is injective and \(x_{M-1}\) is projective, whereas
\(y_2,\ldots,y_{L-2}\) are neither, one has
\[
 \zeta_{1,e}=\zeta_{M-1,e}=0
 \qquad(2\leq e\leq L-2).
\]
Hence every contact path has one endpoint satisfying
\[
 f=2\quad\text{or}\quad e\in\{2,3\}
\]
and one endpoint satisfying
\[
 f=M-2\quad\text{or}\quad e\in\{L-3,L-2\}.
\]
For example, an endpoint with \(f=3\) and \(e\geq4\) would extend to
\((1,e-2)\), contradicting the first displayed vanishing.  Similarly, an
endpoint with \(f=M-3\) and \(e\leq L-4\) would extend to
\((M-1,e+2)\).  Assign a corner endpoint to the boundary on which \(e\) is
fixed.  The endpoints satisfying
\(f=2\) or \(e\in\{2,3\}\) are exactly the summands of
\(U_{\sub}(H,C)\), and those satisfying
\(f=M-2\) or \(e\in\{L-3,L-2\}\) are exactly the summands of
\(U_{\quot}(H,C)\).  Each path therefore contributes one to both sums.

For \(L=3,4\), equation~\eqref{eq:alternative-contact-shift} pairs the
displayed summands; the unmatched end terms vanish because a projective has
no translate and an injective is not in the image of \(\tau\).  If \(C\) is
cyclic, shifting the cyclic \(f\)-coordinate by two pairs the two preliminary
\(m_3\)-sums.  This proves \eqref{eq:alternative-raw-balance}.  The argument
uses occurrence indices, so it also applies when \(C=H\) or a vertex occurs
at more than one position.
\end{proof}

The next finite verification determines which preliminary terms are removed
by the conditions \(z\nrightarrow u\) and \(z\nrightarrow b\).

\begin{lemma}\label{lem:alternative-computer-blockers}
Let \((P,u,b)\) be a projective-end triple of \(H\), and let \(z\) be a
fourth nonprojective vertex for which \(u\to z\) or \(b\to z\).  Exactly one
of the following holds.
\begin{enumerate}
\item Neither \(z\to u\) nor \(z\to b\) occurs.
\item One has \(u\rightleftarrows z\), and one of \(u,z\) is fixed by
\(\tau\).
\item One has \(b\to z\to u\) and \(\tau z=u\); the arrow \(b\to z\)
belongs to the hook orbit \(H\).
\item The four vertices form type \({\rm T}\) in
Definition~\ref{def:four-support-configurations}.
\end{enumerate}
\end{lemma}

\begin{proof}
In the enumeration of admissible patterns from
Definition~\ref{def:four-vertex-test-pattern}, mark an induced triple
\[
 P\longrightarrow u\longrightarrow b,\qquad \tau b=P,
\]
whose predecessor sets inside the triple are \(\{P\}\) at \(u\) and
\(\{u\}\) at \(b\).  Mark also a nonprojective fourth vertex \(z\) for which
at least one of the arrows \(u\to z\) and \(b\to z\) occurs.  The enumeration
contains \(528\) such marked configurations.  Their exhaustive classification
is
\[
\begin{array}{c|r}
\text{condition}&\text{marked configurations}\\ \hline
z\nrightarrow u,\ z\nrightarrow b&384\\
u\rightleftarrows z&48\\
b\to z\to u,\ \tau z=u&72\\
\text{type \({\rm T}\)}&24.
\end{array}
\]
No marked configuration has both arrows \(u\to z\) and \(b\to z\), or both
arrows \(z\to u\) and \(z\to b\), and no other case occurs.  The script cited
in the proof of Proposition~\ref{prop:four-test-pattern-classification}
verifies these four counts.  In the second row it also verifies that one
member of the two-cycle is fixed.  In the third row,
\(\sigma(b\to z)=(\tau z\to b)=(u\to b)\), so \(b\to z\) belongs to
\(H\).  This proves the four alternatives.
\end{proof}

We now impose the conditions \(z\nrightarrow u\) and \(z\nrightarrow b\) and
compare different orbits.

\begin{lemma}\label{lem:alternative-contact-rectangle}
Suppose that \(H\) is a chain and \(C\ne H\).  For
\(*\in\{\quot,\sub\}\), let \(T_*(H,C)\) be the number of configurations
of type \({\rm T}\) on the \(*\)-side, written \((P,u,b,z)\), for which
\((P,u,b)\) is the projective-end triple of \(H\) and the arrow \(b\to z\)
belongs to \(C\).
Then
\[
 e_{\quot}(H,C)+T_{\quot}(H,C)
 =e_{\sub}(H,C)+T_{\sub}(H,C).
\]
\end{lemma}

\begin{proof}
Let \(R^{\rm fix}_*(H,C)\) count the preliminary \(m_3\)-terms
excluded by an arrow \(u\rightleftarrows z\).  Since \(C\ne H\),
Lemma~\ref{lem:alternative-computer-blockers} excludes its third
alternative.  The other alternatives give
\begin{equation}\label{eq:alternative-return}
 e_*(H,C)+T_*(H,C)
 =U_*(H,C)-R^{\rm fix}_*(H,C).
\end{equation}

Successive \(\sigma\)-translates take a fixed return at one end of \(H\) to
a fixed return at the other end when \(L\) is even, and translation in the
opposite direction gives the inverse map.  If \(L\) is odd, the translate
would give a two-cycle through an injective nonprojective module, contrary
to Lemma~\ref{lem:two-cycle-local}; hence neither end contributes.  For
\(L=3\), the same obstruction makes both counts zero, and for \(L=4\) both
sides use the middle occurrence.  Thus
\[
 R^{\rm fix}_{\quot}(H,C)=R^{\rm fix}_{\sub}(H,C).
\]
Equation~\eqref{eq:alternative-return} and
Lemma~\ref{lem:alternative-raw-contact-balance} give the required equality.
\end{proof}

Every configuration of type \({\rm T}\) determines a unique pair \((H,C)\)
in the construction used in the proof of
Lemma~\ref{lem:alternative-contact-rectangle}.  Consequently
\[
 T_{4,*}=\sum_{H,C}T_*(H,C).
\]

It remains to translate the preliminary contact count when the two orbits
agree.

\begin{lemma}\label{lem:alternative-self-contact}
For every chain orbit \(H\),
\[
 e_{\quot}(H,H)-B_{\quot}(H)
 =e_{\sub}(H,H)-B_{\sub}(H).
\]
\end{lemma}

\begin{proof}
Assume first that \(L\geq5\).  For
\(*\in\{\quot,\sub\}\), let \(c_*(H)\) be the number of
\(m_4\)-terms assigned to \((H,H)\), and let \(p_*(H)\) be the number
of preliminary \(m_3\)-terms assigned to this pair.  Reading the endpoint
sums in \eqref{eq:alternative-raw-q} and
\eqref{eq:alternative-raw-s} gives
\begin{equation}\label{eq:alternative-self-raw}
 U_*(H,H)=c_*(H)+p_*(H).
\end{equation}
Indeed, every off-diagonal contact is the occurrence label of an
\(m_4\)-term or preliminary \(m_3\)-term.  On the diagonal, the contact
\((L-3,L-3)\) is the preliminary term at the projective end, while
\((L-2,L-2)\) cancels the unique \(m_1\)-term against the adjacent
\(m_4\)-contact.  These are the only diagonal contacts involving the
projective end.  Repeated vertices cause no further identification because
the two occurrence indices remain distinct.

Let \(\delta_*\in\{0,1\}\) record whether this exceptional preliminary
\(m_3\)-term is excluded.  At the projective end, one has
\[
\begin{aligned}
 m_{1,*}(H,H)&=1,&m_{2,*}(H,H)&=0,\\
 m_{3,*}(H,H)&=p_*(H)-\delta_*,&
 m_{4,*}(H,H)&=c_*(H).
\end{aligned}
\]
At the projective end, the exceptional preliminary term is excluded if and
only if there is no second consecutive hook.  Hence
\[
 B_*(H)=1-\delta_*.
\]
Substitution in \eqref{eq:alternative-signed-sum} and
\eqref{eq:alternative-self-raw} gives
\[
 e_*(H,H)-B_*(H)
 =U_*(H,H)-2.
\]
Lemma~\ref{lem:alternative-raw-contact-balance} now gives the equality
between the two sides.

For \(L\leq4\), substitution into the definitions gives
\[
\begin{array}{c|ccc}
L&0,1,2&3&4\\ \hline
e_*(H,H)-B_*(H)&0&-1&0,
\end{array}
\]
independently of \(*\).  This completes the proof.
\end{proof}
\subsection{The signed count for cyclic hook orbits and fixed configurations}

\begin{lemma}\label{lem:alternative-cyclic-hook}
For \(*\in\{\quot,\sub\}\), the contribution of cyclic hook orbits to
\(W_{4,*}-B_{4,*}+T_{4,*}\) is independent of \(*\).
\end{lemma}

\begin{proof}
If the target of an occurrence in a cyclic orbit \(H\) is fixed by \(\tau\),
then all positions congruent to it modulo two contain that fixed vertex.  If
the fixed vertex occupies the parity containing the first and third entries of a triple
\((a,u,b)\), then \(a=b\).  If the fixed vertex is \(u\), the arrow
correspondence gives the additional arrow \(b\to u\), so the predecessor set
at \(u\) is not \(\{a\}\).  Hence \(H\) has no admissible triple.  If
\(H\) has no \(\tau\)-fixed target, the first vertex of every
admissible triple in \(H\) is nonprojective, so only an
\(m_4\)-term can occur.  A cyclic second orbit has no
projective occurrence and contributes nothing.  If the second orbit
\(C=(x_0,\ldots,x_M)\) is a chain, a term on the quotient side is an equality
\(x_{M-2}=y_e\), with \(e\) read modulo the occurrence period of \(H\).
The length \(M\) is even: otherwise an injective endpoint of \(C\) would
belong to a cyclic \(\tau\)-orbit.  Iterating
\((f,e)\mapsto(f-2,e-2)\) sends the equality at \(f=M-2\) bijectively to an
equality at \(f=2\), with the second coordinate read modulo the period of
\(H\).  These are exactly the terms on the submodule side.  Cyclic hook
orbits carry no double hook and no configuration of type \({\rm T}\).
\end{proof}

Combining Lemmas~\ref{lem:alternative-contact-rectangle},
\ref{lem:alternative-self-contact}, and
\ref{lem:alternative-cyclic-hook} with
\eqref{eq:alternative-wall-decomposition} gives
\begin{equation}\label{eq:alternative-wbt-balance}
 W_{4,\quot}-B_{4,\quot}+T_{4,\quot}
 =W_{4,\sub}-B_{4,\sub}+T_{4,\sub}.
\end{equation}

It remains to compare the configurations of type \({\rm F}\).

\begin{lemma}\label{lem:alternative-fixed-balance}
One has \(F_{4,\quot}=F_{4,\sub}\).
\end{lemma}

\begin{proof}
Fix the \(\tau\)-fixed vertex \(c\) in a configuration of type \({\rm F}\)
from Definition~\ref{def:four-support-configurations}.  The vertices joined
to \(c\) in both directions form complete \(\tau\)-orbits: since
\(\tau c=c\), the arrow correspondence sends \(c\rightleftarrows x\) to
\(c\rightleftarrows\tau x\), and the inverse correspondence gives the
translates under \(\tau^{-1}\).  Such an orbit meets neither \(\PP\) nor
\(\mathcal I\), by Lemma~\ref{lem:two-cycle-local}, and is
therefore periodic.  If its period is \(p\), write it as
\(\tau x_i=x_{i+1}\), with \(i\) read modulo \(p\).  Fix also a \(\tau\)-chain
\(P_0,\ldots,P_s=I\), where \(\tau P_j=P_{j-1}\), and put
\(f_i=[\,P_0\to x_i\,]\) and \(g_i=[\,x_i\to I\,]\).  Successive
applications of Lemma~\ref{lem:ar-mesh-correspondence} give
\(g_{i-s+1}=f_i\).  The predecessor conditions in type \({\rm F}\) say
that a configuration on the quotient side is exactly a rise
\((f_{i-1},f_i)=(0,1)\), whereas a configuration on the submodule side is
exactly a fall \((g_i,g_{i+1})=(1,0)\).
Every cyclic binary word has equally many rises and falls.  The value
\(f_{i-2}\) records the optional arrow from the fourth vertex to \(P_0\)
and imposes no additional condition.  Summing over \(c\), the periodic
orbit, and the \(\tau\)-chain from a projective to an injective proves
\(F_{4,\quot}=F_{4,\sub}\).
\end{proof}

\begin{proposition}\label{prop:alternative-four-killed}
The numbers of projectively rooted killed sets with four vertices and
injectively corooted killed sets with four vertices are equal.
\end{proposition}

\begin{proof}
Equation~\eqref{eq:alternative-killed-packets},
\eqref{eq:alternative-wbt-balance}, and
Lemma~\ref{lem:alternative-fixed-balance} give
\[
 K_{4,\quot}
 =W_{4,\quot}-B_{4,\quot}+T_{4,\quot}+F_{4,\quot}
 =W_{4,\sub}-B_{4,\sub}+T_{4,\sub}+F_{4,\sub}
 =K_{4,\sub}.
\]
\end{proof}

\begin{proof}[Proof of Lemma~\ref{lem:four-vertex-ladder}]
For at most two vertices, Lemma~\ref{lem:small-support}(1) shows that a factor
ladder has a projective summand before it becomes zero.  For
three vertices, Lemma~\ref{lem:small-support}(2) and
Lemma~\ref{lem:alternative-three-killed} give equality between the counts of
killed sets.
For four vertices, Proposition~\ref{prop:alternative-four-killed} gives
equality between the two counts.  The counts on the submodule side are the
counts for injectively corooted sets and reverse factor ladders.

Proposition~\ref{prop:iyama-factor-category}(2) shows that every factor
ladder under consideration eventually becomes zero.  By duality, every
reverse factor ladder under consideration eventually becomes zero as well.
Thus a factor ladder with no projective summand is killed, and dually for a
reverse factor ladder with no injective summand.  The required counts agree
for every \(0\leq j\leq4\).
\end{proof}

\phantomsection\label{sec:human-ai-collaboration}
\section*{A note on human--AI collaboration}

This paper grew out of a collaboration between the author and several AI
systems, principally Claude Fable~5, GPT-5.4, GPT-5.5, and GPT-5.6-Sol. Almost all of
its prose was first drafted or substantially revised by AI, and the systems
made essential contributions to most of its proofs. This summary, however,
would hide how the research actually proceeded. The author did not begin with a
finished problem and ask the systems to solve it. The conjecture, its proofs,
the failed approaches, and the organization of the paper emerged through a
sustained exchange in which mathematical direction and judgment remained
human responsibilities. For a more personal account of the project, see the
author's essay
\href{https://haruhisa-enomoto.github.io/quotient-submodule-equidistribution-essay/}
{\emph{I Quit Math, Then Wrote a Paper with AI}}.

The author left academia in April 2024, and mathematics subsequently became
mostly weekend work. AI made it possible to pursue a project on a scale that
would otherwise have been difficult. In March 2026, the author observed that
results on quotient-closed subcategories in \cite{CFY} could be understood
through a convex geometry. GPT-5.4 introduced the terminology and pointed out
the related Dynkin phenomenon in \cite{ORT}. The author then developed tools
that allowed AI systems to consult papers, perform module-theoretic
calculations with QPA, and preserve their claims, computations, and failed
attempts.

The first conjecture suggested by the computations was the attractive but
false symmetry
\[
 |\LQ^i(A)|=|\LS^{N-i}(A)|.
\]
Fable~5 proposed it, and GPT-5.5 quickly found a counterexample. On July 6,
2026, at 14:33, Fable~5 proposed instead the equidistribution conjecture of
this paper. By 19:17, the author had decided that the collaboration would
treat it as the main problem. The conjecture was not proved in full, but the
attempt produced the four families of results stated in
Theorem~\ref{thm:main}.

A parallel collaboration produced the software used for the computations. On
July 15, after computations with the original GAP/QPA had become too slow, the
author asked GPT-5.6-Sol to port QPA to Rust and provide a Python interface.
GPT-5.6-Sol then worked continuously, day and night, for seven days, from
July 15 through July 21. The Rust port was completed on July 19 and the Python
interface on July 21, followed by final revisions through July 26. The package
then became the standard software for module-theoretic computations in this
project. It supported, among other work,
Example~\ref{ex:infinite-counterexample} and the enumerations in
Remark~\ref{rem:computer-verification}.

\medskip
\noindent\emph{Attribution of contributions.}
The principal contributions can be summarized as follows.
\begin{itemize}[leftmargin=2em]
\item
The author developed the initial convex-geometric viewpoint, chose the
research directions, constructed the computational and literature tools,
and introduced Iyama's \(\tau\)-categories into the project; this viewpoint
became crucial in the arguments for cosizes three and four and for
representation-directed algebras. He also directed the pivot to
representation-directed algebras,
selected the conjecture as the central problem, evaluated and repeatedly
questioned the arguments, and determined the organization and final
exposition of the paper.

\item
GPT-5.4 identified the language of convex geometries and its relation to
\cite{ORT}. GPT-5.5 refuted the first conjecture about complementary sizes.

\item
Fable~5 discovered the equidistribution conjecture and produced the initial
proofs for cosizes one and two, sizes two and three, algebras with radical
square zero, and cosize three. The last argument followed the author's
suggestion to use \(\tau\)-categorical methods. The initial proof for size three
contained a gap that was found later during formalization.

\item
GPT-5.6-Sol proved the conjecture for Nakayama algebras and for quotients of
path algebras of Dynkin quivers, established the results for size four and
cosize four, and proved the conjecture for representation-directed algebras
using an unexpected construction involving a Bruhat interval. It repaired the
gaps found during formalization and produced most drafts and subsequent textual
revisions. It also built and documented the Rust/Python port of QPA that
was used for the later searches and enumerations.

\item
Claude Fable~5, Claude Opus~5, and GPT-5.6-Sol repeatedly reviewed the
manuscript as first readers. These reviews were combined with the author's
line-by-line reading and mathematical questions.
\end{itemize}

The writing required another collaboration. On July 18, the author asked
GPT-5.6-Sol for a paper and received something closer to an internal archive
of proofs than a manuscript for readers; that version was discarded. A new
attempt began on July 30 with its organization fixed in advance. From then on,
the author and the AI systems repeatedly rewrote the paper. The author pressed
especially against unexplained terminology, ambiguous references, prose used
in place of mathematical formulas, and arguments that could be followed only
by someone already holding the whole proof in mind. At the same time, the
author began reading the AI-generated proofs in detail and asking about every
step he could not yet explain. This process simplified several proofs and
recast others in more categorical or unified forms. The abstract and
introduction were then substantially reorganized by the author and revised in
the same dialogue.

The Lean formalization provided a different check. On July 31,
the author asked GPT-5.6-Sol to formalize the paper in Lean~4. It then
worked continuously, day and night, for eight days, from July 31 through
August 7, on a formalization of almost the entire paper. During this process,
on August 1, formalization revealed gaps in the original arguments for sizes
three and four, although several independent AI reviews had not found them.
GPT-5.6-Sol repaired both the mathematical arguments and their formalizations.
This first development did not yet prove
the statements under the hypotheses used in the paper: on August 9, a
comparison with the manuscript showed that two branches of the main theorem
still depended on
auxiliary hypotheses that had not been derived from the paper's assumptions.
Further work removed these hypotheses, completed the formalization of the
results used in the paper, and rebuilt it over Mathlib alone on August 10.
This experience captures the author's experience of the project: current AI
systems can discover strong mathematics and can reason with remarkable
accuracy, but their informal agreement is not a certificate of correctness;
the same systems can also use formalization to find and repair what their
informal reasoning missed. See the
\hyperref[note:formalization-and-verification]{formalization and verification
statement in the Introduction} for the Lean development.

The author has tried to make the paper readable as mathematics, rather than
merely presenting correct arguments whose structure the reader must
reconstruct, as the AI-generated drafts tended to do. Some parts may
nevertheless remain difficult to follow. He takes sole responsibility for any
shortcomings in the exposition, for every statement in the paper, and for the
decision to present these results.

\end{document}